\documentclass[reqno, 11pt]{amsart}
\title[ The $\ell^p$-boundedness of wave operators]{ The $\ell^p$-boundedness of wave operators for the fourth order Schr\"{o}dinger operators on the lattice $\Z$}
\author{Sisi Huang and Xiaohua Yao}

\address{Sisi Huang, Department of Mathematics, Central China Normal University, Wuhan, 430079, P.R. China}
\email{hss@mails.ccnu.edu.cn}

\address{Xiaohua Yao, Department of Mathematics and  Key Laboratory of Nonlinear Analysis and Applications(Ministry of Education), Central China Normal University, Wuhan, 430079, P.R. China}
\email{yaoxiaohua@ccnu.edu.cn}

\thanks{The work is partially supported by NSFC No.12171182, 12531005 and and the Fundamental Research Funds for the Central Universities}
\date{\today}

\usepackage{color}
\usepackage{amsmath}
\usepackage{amssymb}
\usepackage[mathscr]{euscript}
\usepackage{paralist}
\usepackage{amsthm}
\usepackage{ulem}
\usepackage{bm}
\usepackage{extarrows}
\usepackage{diagbox}
\usepackage{multirow}
\usepackage{makecell}
\usepackage{threeparttable}
\usepackage{tikz}
\usetikzlibrary{patterns}
\usetikzlibrary{decorations.markings}
\usepackage{caption}
\usepackage{verbatim}

\usepackage{cite}
\usepackage{hyperref}
\newtheorem{definition}{Definition}[section]
\newtheorem{theorem}[definition]{Theorem}
\newtheorem{lemma}[definition]{Lemma}
\newtheorem{remark}[definition]{Remark}
\newtheorem{proposition}[definition]{Proposition}

\newcommand\R{\mathbb{R}}
\newcommand\Z{\mathbb{Z}}
\newcommand\B{\mathbb{B}}

\newcommand\N{\mathbb{N}}
\newcommand\T{\mathbb{T}}
\newcommand\mcaH{\mathcal{H}}
\newcommand\mcaF{\mathcal{F}}
\newcommand\mcaI{\mathcal{I}}

\newcommand\mcaA{\mathcal{A}}
\newcommand\mcaB{\mathcal{B}}
\newcommand\mcaC{\mathcal{C}}
\newcommand\mcaD{\mathcal{D}}
\newcommand\mcaK{\mathcal{K}}
\newcommand\mcaL{\mathcal{L}}

\newcommand\mcaM{\mathcal{M}}
\newcommand\mcaT{\mathcal{T}}

\numberwithin{equation}{section}

\begin{document}

\maketitle

\begin{abstract}
This paper investigates the $\ell^p$ boundedness of wave operators $W_\pm(H,\Delta^2)$ associated with discrete fourth-order Schr\"odinger operators $H = \Delta^2 + V$ on the lattice $\Z$, where
$$(\Delta\phi)(n)=\phi(n+1)+\phi(n-1)-2\phi(n),\quad n\in\Z,$$
and $V(n)$ is a real-valued potential on $\Z$.
Under suitable decay assumptions on $V$ (depending on the types of zero resonance of $H$), we show that the wave operators $W_{\pm}(H, \Delta^2)$ are bounded on $\ell^p(\mathbb{Z})$ for all $1 < p < \infty$:
$$
\|W_{\pm}(H, \Delta^2) f\|_{\ell^p(\mathbb{Z})} \lesssim \|f\|_{\ell^p(\mathbb{Z})}.
$$
In particular, if both thresholds $0$ and $16$ are regular points of $H$, we prove that $W_{\pm}(H, \Delta^2)$ are neither bounded on the endpoint space $\ell^1(\mathbb{Z})$ nor on $\ell^\infty(\mathbb{Z})$. We remark that the proof of these bounds relies fundamentally on the asymptotic expansions of the resolvent of $H$ near the thresholds $0$ and $16$, and on the theory of {\it discrete singular integrals} on the lattice.

As applications, we derive the following sharp $\ell^p-\ell^{p'}$ decay estimates for solutions to the discrete beam equation with a parameter $a\in\R$ on the lattice $\Z$:
\begin{equation*}
\|{\rm cos}(t\sqrt {H+a^2})P_{ac}(H)\|_{\ell^p\rightarrow\ell^{p'}}+\left\|\frac{{\rm sin}(t\sqrt {H+a^2})}{t\sqrt {H+a^2}}P_{ac}(H)\right\|_{\ell^p\rightarrow\ell^{p'}}\lesssim|t|^{-\frac{1}{3}(\frac{1}{p}-\frac{1}{p'})},\quad t\neq0,
\end{equation*}
where $1<p\le 2$, ${p'}$ is the conjugated index of $p$ and $P_{ac}(H)$ denotes the spectral projection onto the absolutely continuous spectrum space of $H$.

\end{abstract}

\tableofcontents

\section{Introduction and main results}
\subsection{Introduction}
Let $\Delta$ denote the discrete Laplacian on the lattice $\Z$, defined by
\begin{equation}\label{defin of laplacian}
(\Delta\phi)(n)=\phi(n+1)+\phi(n-1)-2\phi(n),\quad n\in\Z.
\end{equation}

We consider the following  fourth-order Schr\"{o}dinger operators $H$ acting on the space $\ell^2(\Z)$:
\begin{equation}\label{bi-Schrodinger}
H=\Delta^2+V,
\end{equation}
where $V(n)$ is a real-valued potential on $\Z$ and satisfies $|V(n)|\lesssim \left<n\right>^{-\beta}$ for some $\beta>0$ with $\left<n\right>=(1+|n|^2)^{\frac{1}{2}}$. Both $\Delta^2$ and $H$ are bounded self-adjoint operators on $\ell^2(\Z)$, generating the associated unitary groups $e^{it\Delta^2}$ and $e^{itH}$, respectively.

The {\textit{wave operators}} associated with $H$ are defined as the strong limits on $\ell^2(\Z)$:
\begin{equation}\label{defin of wave operators}
W_{\pm}:=W_{\pm}(H,\Delta^2)=s-\lim\limits_{t\rightarrow\mp\infty}e^{itH}e^{-it\Delta^2}.
\end{equation}
When $\beta>1$, it is known~(cf. \cite[Section XI.3]{RS79}) that \eqref{defin of wave operators} exist as partial isometries from $\ell^2(\Z)$ to $\mcaH_{ac}(H)$~(the absolutely continuous spectral subspace of $H$) and are asymptotically complete (i.e., Ran$W_+=$Ran$W_-=\mcaH_{ac}(H)$). Moreover, the {\textit{inverse}}~({\textit{dual}})~{\textit{wave operators}}
\begin{equation*}
W_{\pm}(\Delta^2,H)=s-\lim\limits_{t\rightarrow\mp\infty}e^{it\Delta^2}e^{-itH}P_{ac}(H)
\end{equation*}
also exist and satisfy $W_{\pm}(\Delta^2,H)=W^*_{\pm}$, where $P_{ac}(H)$ denotes the spectral projection onto $\mcaH_{ac}(H)$.

Such wave operators, initially introduced in quantum scattering theory by Moller and Friedrichs and later developed by Jauch, Cook and Kato etc.,
serve as indispensable tools for understanding the long-time behavior of evolution equations, cf.\cite{Kat95,RS79,Bec14}. Owing to their fundamental role in scattering theory, non-linear partial differential equations, and spectral theory, the study of wave operators occupies an important position in modern mathematical physics. In particular, the analysis of their $L^p$-boundedness
 has drawn growing interest and achieved substantial progress. This paper aims to investigate the $\ell^p$ boundedness of $W_{\pm}(H,\Delta^2)$ and $W^{*}_{\pm}(H,\Delta^2)$ associated with \eqref{bi-Schrodinger}. 

 In Euclidean space $\R^d$, it is  known that the study of $L^p$ boundedness of wave operators was initiated by K. Yajima in his seminal works \cite{Yaj93,Yaj95a,Yaj95b,Yaj97}
for second-order Schr\"{o}dinger operators $-\Delta_{\R^d}+V$ on $\R^d$ with $d\geq3$, where he proved that wave operators are $L^p$ bounded for all $1\leq p\leq\infty$ if zero energy is regular. Subsequently, this topic has been developed for lower dimensions, for instance, Jensen-Yajima \cite{Yaj99,JY02}, Erdo\u gan-Goldberg-Green\cite{EGG18} for $d=2$ and Weder \cite{Wed99}, Galtbayar-Yajima \cite{GY00}, D'Ancona-Fanelli \cite{DF06} for $d=1$ and references therein.

Moreover, significant advances of this issue have been made for the higher-order Schr\"{o}dinger operators $(-\Delta_{\R^d})^m+V$ with $m\geq2$, although which has only begun in very recent several years, compared to the second-order case studied since the 1990s. The first work \cite{GG21} by Goldberg and Green established $L^p$ boundedness for $1 < p < \infty$ in the regular case for $(m,d) = (2,3)$, which was subsequently extended to general case $d > 2m$ by Erdo\u{g}an and Green in \cite{EG22,EG23}. Further developments include Mizutani, Wan and Yao's investigation \cite{MWY24a,MWY25} of endpoint behavior and zero energy resonances for $(m,d) = (2,3)$ and their complete analysis \cite{MWY24} of all zero resonance types in $(m,d) = (2,1)$, along with Galtbayar-Yajima's study \cite{GY24} of the $(m,d) = (2,4)$ case. More recently, Erdo\u gan-Green-LaMaster \cite{EGL25} considered the case $d>4m$ while Cheng, Soffer, Wu and Yao \cite{CSWY25,CSWY25b} covered the remaining cases $1 \leq d \leq 4m$.

As a natural extension, the scattering theory on the lattice $\Z^d$ has attracted much attention in recent two decades and undergone significant developments~(cf.\cite{BEI24,AIM18,AIM16,IM15,IK12,BS12}). In contrast, the $\ell^p$ boundedness of wave operators on lattices seems largely unexplored. 
To the best of our knowledge, the only known result 
is due to Cuccagna \cite{Cuc09}, who established the $\ell^p$ bounds of wave operators  for $-\Delta+V$ on $\Z$. The corresponding problem for higher-order wave operators, even for  $H=\Delta^2+V$ on $\Z$,  remains open, which consists of
the main goal of this paper.

Furthermore, interestingly, 
once we have established the $\ell^p$-boundedness of $W_\pm(H,\Delta^2)$ and $W^*_\pm(H,\Delta^2)$, the following {\textit{intertwining property}}
\begin{equation}\label{interwinning property}
f(H)P_{ac}(H)=W_{\pm}f(\Delta^2)W^*_{\pm}
\end{equation}
enables us to reduce $\ell^p-\ell^q$ estimates for the perturbed operator $f(H)$ to the corresponding estimates for the free operator $f(\Delta^2)$:
\begin{equation*}
\|f(H)P_{ac}(H)\|_{\ell^p\rightarrow\ell^q}\leq\|W_{\pm}\|_{\ell^q\rightarrow\ell^q}\|f(\Delta^2)\|_{\ell^p\rightarrow\ell^q}\|W^*_{\pm}\|_{\ell^p\rightarrow\ell^p},
\end{equation*}
where $f$ is any Borel function on $\R$.

As applications, we will establish the time decay estimates for the solution to the discrete beam equation with parameter $a\in\R$ on the lattice $\Z$:
\begin{equation}\label{discrete Beam equation}
\left\{\begin{aligned}&(\partial_{tt}u)(t,n)+\big((H+a^2)u\big)(t,n)=0,\ \ (t,n)\in\R\times\Z,\\
&u(0,n)=\varphi_1(n),\ \left(\partial_{t}u\right)(0,n)=\varphi_2(n),\end{aligned}\right.
\end{equation}
whose solution is given by
\begin{equation*}
u_a(t,n)={\rm cos}(t\sqrt {H+a^2})\varphi_1(n)+\frac{{\rm sin}(t\sqrt {H+a^2})}{\sqrt {H+a^2}}\varphi_2(n).
\end{equation*}
Note that in the free case~(i.e., $V\equiv0$), by means of Fourier method, the following sharp $\ell^p-\ell^{p'}$ decay estimates hold for all $a\in\R$ (see Theorem \ref{free decay} below): 
\begin{equation*}
\|{\rm cos}(t\sqrt {\Delta^2+a^2})\|_{\ell^p\rightarrow\ell^{p'}}+\left\|\frac{{\rm sin}(t\sqrt {\Delta^2+a^2})}{t\sqrt {\Delta^2+a^2}}\right\|_{\ell^p\rightarrow\ell^{p'}}\lesssim|t|^{-\frac{1}{3}(\frac{1}{p}-\frac{1}{p'})},\quad t\neq0, 
\end{equation*}
where $1\le p\leq2$ and $\frac{1}{p}+\frac{1}{p'}=1$.
When $V\not\equiv0$, the decay estimates for the solution operators of equation \eqref{discrete Beam equation} are affected by the spectrum of $H$, which in turn depends on the conditions of potential $V$. In this paper, assuming that the potential $V$ has fast decaying and $H$ has no embedded positive eigenvalues in the continuous spectrum interval $(0,16)$, we prove that the wave operators $W_{\pm}(H,\Delta^2)$ are  bounded on $\ell^p(\Z)$ for all $1<p<\infty$ (see Theorem \ref{main theorem} below). As a consequence of this boundedness and the intertwining property \eqref{interwinning property}, we obtain the following $\ell^p-\ell^{p'}$ decay estimates:
\begin{equation*}
\|{\rm cos}(t\sqrt {H+a^2})P_{ac}(H)\|_{\ell^p\rightarrow\ell^{p'}}+\left\|\frac{{\rm sin}(t\sqrt {H+a^2})}{t\sqrt {H+a^2}}P_{ac}(H)\right\|_{\ell^p\rightarrow\ell^{p'}}\lesssim|t|^{-\frac{1}{3}(\frac{1}{p}-\frac{1}{p'})},\quad t\neq0,
\end{equation*}
for all $1<p\leq2$ and $a\in \R$. For more details, we refer to Section \ref{sec of application}.

To obtain the $\ell^p$ boundedness of $W_\pm(H,\Delta^2)$, our starting point is the stationary representation of $W_{\pm}$. We first establish the limiting absorption principle for the operator $H$, and then study the asymptotic expansions of $R^{\pm}_V(\lambda)$ near thresholds $0$ and $16$ for all resonance types~(see Definition \ref{defin of resonance types 0 and 16} below). Finally, we employ the Schur test lemma and the theory of discrete singular integral on the lattice to derive the desired boundedness.

\subsection{Main results}
In this subsection, we present our main results. First, we give some definitions. 
Unlike the continuous case where zero is the only critical value, our discrete setting will involve two critical values:~$0$ (degenerate) and $16$ (non-degenerate)~(see Subsection \ref{subsec of LAP}). This gives rise to more diverse resonance types, which can be characterized respectively by the solutions to difference equations $H\phi = 0$ and $H\phi = 16\phi$ in some intersection spaces
$W_{\sigma}(\Z)$:=$\underset{s>\sigma}{\bigcap}\ell^{2,-s}(\Z)$ with $\sigma\in\R$ and
$$\ell^{2,s}(\Z)=\Big\{\phi=\left\{\phi(n)\right\}_{n\in\Z}:\|\phi\|^2_{\ell^{2,s}}=\sum_{n\in\Z}^{}\left<n\right>^{2s}|\phi(n)|^2<\infty\Big\}.$$
More precisely, let $a\lesssim b$ denote $a\leq cb$ with some constant $c>0$ for $a,b\in\R^{+}$, we have 
\begin{definition}\label{defin of resonance types 0 and 16}
 Let $H=\Delta^2+V$ be defined on the lattice $\Z$ and $|V(n)|\lesssim \left<n\right>^{-\beta}$ for some $\beta>0$. 
\vskip0.2cm
{\bf{(I)~Classification of resonances at threshold 0}}
\vskip0.2cm
\begin{itemize}
\item[{\rm(i)}] $0$ is a regular point of $H$ if $H\phi=0$ has only zero solution in $W_{3/2}(\Z)$.
    \vskip0.1cm
\item [{\rm(ii)}] $0$ is a first kind resonance of $H$ if  $H\phi=0$ has nonzero solution in $W_{3/2}(\Z)$ but only zero solution in $W_{1/2}(\Z)$.
\vskip0.1cm
\item [{\rm(iii)}] $0$ is a second kind resonance of $H$ if $H\phi=0$ has nonzero solution in $W_{1/2}(\Z)$ but only zero solution in $\ell^2(\Z)$.
\vskip0.1cm
\item [{\rm(iv)}] $0$ is an eigenvalue of $H$ if  $H\phi=0$ has nonzero solution in $\ell^2(\Z)$.
\end{itemize}
\vskip0.2cm
\ \  {\bf{(II)~Classification of resonances at threshold 16}}
\vskip0.2cm
\begin{itemize}
\item [{\rm(i)}]  $16$ is a regular point of $H$ if $H\phi=16 \phi$ has only zero solution in $W_{1/2}(\Z)$.
\vskip0.1cm
\item [{\rm(ii)}]  $16$ is a resonance of $H$ if  $H\phi=16\phi$ has nonzero solution in $W_{1/2}(\Z)$ but only zero solution in $\ell^2(\Z)$.
\vskip0.1cm
\item [{\rm(iii)}] $16$ is an eigenvalue of $H$ if $H\phi=16\phi$ has nonzero solution in $\ell^2(\Z)$.
\end{itemize}
\end{definition}
Obviously, when $V\equiv0$, both 0 and 16 are the resonances of $H$. This can be verified by taking $\phi_1(n)=cn+d$ and $\phi_2(n)=(-1)^{n}c$ with $c\neq0$, which satisfy $\Delta^2\phi_1=0$ and $\Delta^2\phi_2=16\phi_2$. Beyond this special case, there are some other non-trivial zero/sixteen resonance examples.

{\bf{\underline{ 1. Example of resonance.}}} Consider the function $\phi(n)=2$ for $n=0$ and $\phi(n)=1$ for $n\neq0$ and define the potentials $V_1(n)=-\frac{(\Delta^2\phi)(n)}{\phi(n)}$ and $V_2(n)=16+V_1(n)$. Then
$$(\Delta^2+V_1)\phi=0,\quad (\Delta^2+V_2)\phi=16\phi.$$
In this case, we have
$$V_1(n)=\left\{\begin{aligned}-3,&\quad n=0,\\4,&\quad n=\pm1,\\-1,&\quad n=\pm2,\\
0,&\quad {\rm else},\end{aligned}\right.,\quad V_2(n)=\left\{\begin{aligned}13,&\quad n=0,\\20,&\quad n=\pm1,\\15,&\quad n=\pm2,\\
16,&\quad {\rm else}.\end{aligned}\right.$$
 It indicates that $0$ persists a resonance even for such compactly supported potential, and by shifting the potential by $16$ one can turns the resonance at $0$ into a resonance at $16$.

{\bf{\underline{2. Example of eigenvalue.}}} Take $\phi(n)=(1+n^2)^{-s}$ with $s>\frac{1}{4}$, $V_1(n)=-\frac{(\Delta^2\phi)(n)}{\phi(n)}$ and $V_2(n)=16+V_1(n)$. At this time,
$$V_1(n)=O(\left<n\right>^{-4}),\quad |n|\rightarrow\infty.$$
This implies that $0$ becomes an eigenvalue of $H$ under such slowly decaying potential. However, for potentials exhibiting faster decay with $\beta>9$, we can preclude the zero eigenvalue case. A detailed explanation of this can be found in \cite[Lemma 5.2]{HY25}.
\begin{remark}
{\rm We remark that  zero~(resp. sixteen) resonance are closely related to the asymptotic expansions of the resolvent $R_V(z)$ of the operator $H$ near zero~(resp. sixteen) energy, and are further connected to the asymptotic expansions of $M^{-1}(\mu)$ near $\mu=0$ (resp. $\mu=2$) through the formula \eqref{inverti relation}.

 For instance, the concepts of these zero resonances originate from the invertibility of specific operators $T_j$
restricted to the ranges of orthogonal projection operators
$S_j$ on $\ell^2(\Z)$ for $j=0,1,2$ during the computation of $M^{-1}(\mu)$ near $\mu=0$. This invertibility is equivalent to whether the corresponding kernel subspace
${\rm Ker}T_j\big|_{S_j\ell^2(\Z)}$ is non-trivial, which further manifests as whether the associated projection space $S_{j+1}\ell^2(\Z)$  is non-trivial.
The non-triviality of these projection spaces $S_{j+1}\ell^2(\Z)$~ in turn is equivalent to the existence of non-zero solutions to difference equation $H\phi=0$ in a suitable weighted space $W_{\sigma}(\Z)$
 (see Subsection \ref{subsec:asymptotics} below for more details).
}
\end{remark}
We now illustrate the main result. Denote by $\B(X,Y)$ the space of all bounded linear operators from $X$ to $Y$ and abbreviate $\B(X,X)$ as $\B(X)$ when $X=Y$.

\begin{theorem}\label{main theorem}
Let $H=\Delta^2+V$ with $|V(n)|\lesssim \left<n\right>^{-\beta}$ for some $\beta>0$. Suppose that $H$ has no positive eigenvalues in the interval $\rm{(}0,16\rm{)}$. If
\begin{align}\label{case and condition}
 \beta>\left\{\begin{aligned}&17,\ 0\ is\ a\ regular\ point\ of\ H,\\
&19,\ 0\ is\ a\ first\ kind\ resonance\ of\ H,\\
&27,\ 0\ is\ a\ second\ kind\ resonance\ of\ H,\end{aligned}\right.
\end{align}
then the wave operators $W_{\pm}\in\B(\ell^p(\Z))$ for all $1<p<\infty$.
\end{theorem}
Furthermore, we establish the following unboundedness results at the endpoints.
\begin{theorem}\label{theorem at endpoints}
Let $H=\Delta^2+V$ with $|V(n)|\lesssim \left<n\right>^{-\beta}$ for some $\beta>0$, and suppose that $H$ has no positive eigenvalues in the interval $\rm{(}0,16\rm{)}$.
\vskip0.15cm
{\rm(i)}~If both $0$ and $16$ are regular points of $H$ and $\beta>15$, then $W_{\pm}\notin\B(\ell^{1}(\Z))\cup\B(\ell^{\infty}(\Z))$.
\vskip0.15cm
{\rm(ii)}~Let  $V$ be compactly supported. Then
\begin{itemize}
\item If $0$ is a regular point and $16$ is a~{\rm(}an{\rm)} resonance or eigenvalue of $H$,  then $W_{\pm}\not\in\B(\ell^{\infty}(\Z))$. Moreover, $W_{\pm}\notin\B(\ell^1(\Z))$ given that in addition
\begin{align*}
 16(1\mp3\sqrt2)\neq\left\{\begin{aligned}&\mcaC_1,\ 16\ is\ a\ resonance\ of\ H,\\
&\mcaC_2,\ 16\ is\ an\ eigenvalue\ of\ H,\end{aligned}\right.
\end{align*}
 where $\mcaC_1$ and $\mcaC_2$ are constants given in \eqref{mcaC1} and \eqref{mcaC2}, respectively.
\vskip0.15cm
\item If $0$ is a first kind resonance of $H$ and  $\mcaC_3\neq0$, then  $W_{\pm}\not\in\B(\ell^{\infty}(\Z)).$ Moreover, $W_{\pm}\not\in\B(\ell^{1}(\Z))$ given additionally that
\begin{align*}
 192|D|\neq\left\{\begin{aligned}&16,\ 16\ is\ a\ regular\ point\ of\ H,\\
&\big|16-\mcaC_1\big|,\ 16\ is\ a\ resonance\ of\ H,\\
&\big|16-\mcaC_2\big|,\ 16\ is\ an\ eigenvalue\ of\ H,\end{aligned}\right.
\end{align*}
\noindent where $\mcaC_3$ and $D$ are  constants defined in \eqref{C} and \eqref{D}, respectively.
\vskip0.15cm
\item If $0$ is a second kind resonance of $H$ and $\mcaC_4\neq0$, then  $W_{\pm}\not\in\B(\ell^{\infty}(\Z))$. Moreover, $W_{\pm}\not\in\B(\ell^{1}(\Z))$  given additionally that
\begin{align*}
 192|E|\neq\left\{\begin{aligned}&16,\ 16\ is\ a\ regular\ point\ of\ H,\\
&\big|16-\mcaC_1\big|,\ 16\ is\ a\ resonance\ of\ H,\\
&\big|16-\mcaC_2\big|,\ 16\ is\ an\ eigenvalue\ of\ H,\end{aligned}\right.
\end{align*}
\noindent where $\mcaC_4$ and $E$ are  constants defined in \eqref{C4} and \eqref{E}, respectively.

\end{itemize}
\end{theorem}
\begin{remark}
{\rm Further remarks on Theorem \ref{main theorem} and Theorem \ref{theorem at endpoints} are given as follows.
\begin{itemize}
\item[{\rm (1)}]Although two resonance types of $H$ may coexist, we emphasize that the decay rate of the potential $V$ in \eqref{case and condition} above is fundamentally determined by the types of zero energy. The required rate $\beta$ of $V$ in Theorem \ref{main theorem}, derived from the asymptotical expansion of $R^\pm_V(\mu^4)$ 
(see Lemma \ref{Asymptotic expansion theorem} below), might not be optimal.
\item[{\rm (2)}] We point out that the coefficients $\mcaC_j$, $D$ and $E$ above are closely related to the operators in our expansions of $M^{-1}(\mu)$.

\item[{\rm (3)}]Our approach of proofs of Theorems \ref{main theorem} and  \ref{theorem at endpoints} is motivated by the recent work of Mizutani, Wan and Yao \cite{MWY24} on the continuous counterpart of this problem. However, some distinctive challenges arise in our discrete setting. First, the diversity of resonance types significantly complicates our boundedness analysis, particularly concerning endpoint unboundedness. Second, establishing $\ell^p$-boundedness in our framework involves {\it the discrete singular integral theory on lattices}.

\end{itemize}
}
\end{remark}

\begin{remark} {\rm As an application of Theorem \ref{main theorem} and the intertwining property \eqref{interwinning property}, it is known that the following $\ell^p-\ell^{p'}$ decay estimates for solutions to the discrete beam equation \eqref{discrete Beam equation} with a parameter $a\in \R$ hold:
\begin{equation}\label{general a cos-sin decay-estimate}
\|{\rm cos}(t\sqrt {H+a^2})P_{ac}(H)\|_{\ell^p\rightarrow\ell^{p'}}+\left\|\frac{{\rm sin}(t\sqrt {H+a^2})}{t\sqrt {H+a^2}}P_{ac}(H)\right\|_{\ell^p\rightarrow\ell^{p'}}\lesssim|t|^{-\frac{1}{3}(\frac{1}{p}-\frac{1}{p'})},
\end{equation}
 where $t\neq0$, $1<p\le 2$ and $\frac{1}{p}+\frac{1}{p'}=1$.
We remark that when $a = 0$, this estimate remains valid even for the endpoint case $p = 1$, as established in our previous work \cite{HY25} through a direct approach. While the unboundedness of the wave operators on $\ell^1(\Z)$ here is not valid to yield this decay estimate for the endpoint case $p=1$, it nevertheless offers valuable insight for establishing such $\ell^p-\ell^{p'}$ decay estimates for arbitrary $a\in\R$.
Actually, using the idea developed in \cite{HY25}, one can also prove that \eqref{general a cos-sin decay-estimate} is true for $p=1$ and  $a\neq 0$.}
\end{remark}
\subsection{The outline of the proof}\label{subsec outline}
In this subsection, we are devoted to presenting the key ideas of the proof for the above theorems. In view of the relation $W_{-}f=\overline{W_+\overline{f}}$, it suffices to analyze $W_+$ alone. Starting from Stone's formula, we obtain the representation for $W_+$:
\begin{align}\label{expre of Wpm}
W_{+}=I-\frac{2}{\pi i}\int_{0}^{2}\mu^3R^{+}_{V}(\mu^4)V(R^{+}_{0}-R^{-}_0)(\mu^4)d\mu.
\end{align}

The first problem arisen here is to establish the existence of boundary values $R^{\pm}_V(\mu^4)$. It is well-known that the limiting absorption principle (LAP) generally states that the resolvent $(H-z)^{-1}$ may converge in a suitable way as $z$ approaches spectrum points, which plays a fundamental role in spectral and scattering theory. 
For instance, see Agmon's work \cite{Agm75} for the Schr\"{o}dinger operator $-\Delta_{\R^d}+V$ in $\R^{d}$. In the discrete setting, the LAP for discrete Schr\"{o}dinger operators $-\Delta_{\Z^d}+V$ on $\Z^d$ has been extensively studied (cf.\cite{Esk67,SV01,BS98,BS99,KKK06,KKV08,IK12,Man17,PS08} and references therein).

However, to the best of our knowledge, it seems that LAP is open for higher-order Schr\"{o}dinger operators on the lattice $\Z^d$. In our recent work~\cite[Section 2]{HY25}, we addressed this issue for the fourth-order operator $H=\Delta^2+V$ defined in \eqref{bi-Schrodinger} by employing commutator estimates and Mourre theory (cf.\cite{JMP84,Mou81,Mou83}). Specifically, under appropriate conditions on $V$, we proved that  $R^{\pm}_{V}(\mu^4)$ for $H$ exist as elements of some weighted spaces $\B(\ell^{2,s}(\Z),\ell^{2,-s}(\Z))$~(see Lemma \ref{LAP-lemma}).

Throughout the paper, we denote by $K$ the operator with kernel $K(n,m)$, i.e.,
 $$(Kf)(n):=\sum\limits_{m\in\Z}^{}K(n,m)f(m).$$
We can obtain the explicit expression for the kernel of free resolvent $R^{\pm}_0(\mu^4)$~(see Lemma \ref{lemma of kernel R0 boundary}):
\begin{align}\label{kernel of R0}
R^{\pm}_{0}(\mu^4,n,m)
=\frac{1}{4\mu^3}\Big(\frac{\pm ie^{\mp i\theta_{+}|n-m|}}{\sqrt{1-\frac{\mu^2}{4}}}-\frac{e^{b(\mu)|n-m|}}{\sqrt{1+\frac{\mu^2}{4}}}\Big),
\end{align}
where $\theta_+:=\theta_{+}(\mu)$ satisfies $2-2{\rm cos}\theta_{+}=\mu^2$ with $\theta_{+}\in(-\pi,0)$ and $b(\mu)={\rm ln} \big(1+\frac{\mu^2}{2}-\mu(1+\frac{\mu^2}{4})^{\frac{1}{2}}\big)$.
We note that the resolvent $R^\pm_0(\mu^4)$ exhibits singular behavior with order $O(\mu^{-3})$ near $\mu=0$ and $O((2-\mu)^{-1/2})$ near $\mu=2$. Based on this observation, given a sufficiently small fixed positive constant $0<\mu_0\ll1$, we consider the partition of unity:
$$\chi_1(\mu)+\chi_2(\mu)+\chi_3(\mu)=1,$$
where $\chi_1(\mu)\in C^{\infty}_0([0,\mu_0])$, $\chi_2(\mu)\in C^{\infty}_0([\mu_0,2-\mu_0])$ and $\chi_3(\mu)\in C^{\infty}_0([2-\mu_0,2])$. Correspondingly,
\begin{align}\label{expre of W-}
W_+=I-\frac{2}{\pi i}\sum\limits_{j=1}^{3}\int_{0}^{2}\mu^3\chi_j(\mu)R^{+}_{V}(\mu^4)V(R^{+}_{0}-R^{-}_0)(\mu^4)d\mu:=I-\frac{2}{\pi i}\sum\limits_{j=1}^{3}\mcaK_j.
\end{align}
Therefore, the $\ell^p$ boundedness of $W_+$ reduces to establishing the boundedness of each part $\mcaK_j$. Among these parts, the intermediate energy part $\mcaK_2$ is easier to handle 
since the resolvent does not have singularity in the interval $[\mu_0,2-\mu_0]$. Indeed, as shown in Section \ref{sec of W2}, $\mcaK_2\in\B(\ell^p(\Z))$ for all $1\leq p\leq\infty$. Consequently, much effort in this article is devoted to handling the low energy part $\mcaK_1$ and the high energy part $\mcaK_3$, which constitutes the main challenge of this paper.

To overcome this, a key point involves analyzing the asymptotic behaviors of $R^{+}_V(\mu^4)$ near $\mu=0$ and $\mu=2$. To this end, we introduce
\begin{equation*}
M(\mu)=U+vR^{+}_{0}(\mu^4)v,\quad\mu\in(0,2),\quad U={\rm sign}\left(V(n)\right),\quad v(n)=\sqrt{|V(n)|}.
\end{equation*}
As established in Lemma \ref{existence of inverse of Mmu}, $M(\mu)$ is invertible on $\ell^2(\Z)$ and its inverse $M^{-1}(\mu)$ satisfies
 \begin{equation*}
R^{+}_V(\mu^4)V=R^{+}_0(\mu^4)vM^{-1}(\mu)v.
\end{equation*}
This allows us to reformulate $\mcaK_j$ as
\begin{align}\label{mcaKj}
\mcaK_j&=\int_{0}^{2}\mu^3\chi_j(\mu)R^{+}_{0}(\mu^4)v M^{-1}(\mu)v(R^{+}_{0}-R^{-}_0)(\mu^4)d\mu,\quad j=1,3.
\end{align}
Hence, the asymptotic expansions of $M^{-1}(\mu)$ near $\mu=0$ and $\mu=2$ are crucial. These expansions were derived in our recent work \cite[Theorem 1.8]{HY25} and will be restated in Lemma \ref{Asymptotic expansion theorem}. The basic idea behind the expansions of $M^{-1}(\mu)$ is the Neumann expansion, which in turn depends on the expansion of $R^{+}_{0}(\mu^4)$. In this respect, Jensen and Kato initiated their seminal work in \cite{JK79} for Schr\"odinger operator $-\Delta_{\R^3}+V$ on $\R^3$. Since then, the method has been widely applied (cf. \cite{JN01, SWY22}).
When considering the discrete bi-Laplacian $\Delta^2$ on the lattice $\Z$, we will face two distinct difficulties.
Firstly, compared with Laplacian $-\Delta$ on $\Z$, the threshold $0$ now is a {\bf degenerate critical value}~(i.e., $\mcaM(0)=\mcaM^{'}(0)=\mcaM^{''}(0)=0$, where the symbol $\mcaM(x)=(2-2{\rm cos}x)^2$ is defined in \eqref{unitary equivalent}). This degeneracy leads to additional steps to expand the $M^{-1}(\mu)$. Secondly, in contrast to the continuous analogue \cite{SWY22}, we encounter another threshold 16~(i.e., corresponding to $\mu=2$).

With these expansions in hand, the next crucial step is to utilize them to establish the desired boundedness of $\mcaK_1$ and $\mcaK_3$. For simplicity,
we consider the representative case where both $0$ and $16$ are regular points of the operator $H$. The subsequent analysis will be divided into two parts.

\subsubsection{\bf{On the $\ell^p$ boundedness}}
\begin{itemize}
\item \underline{\bf{For the low energy part} ${\bm{\mcaK_1}}.$} We first note that in expression \eqref{kernel of R0},  $\theta_{+}$ and $b(\mu)$ exhibit
\end{itemize}
the following behaviors, respectively:
$$\theta_{+}=-\mu+o(\mu),\quad b(\mu)=-\mu+o(\mu),\quad \mu\rightarrow0^+.$$
This shows that the kernel $R^{\pm}_0(\mu^4,n,m)$ closely resembles its continuous counterpart~(cf. \cite{SWY22}):
$$R^{\pm}_0(\mu^4,x,y)=\frac{1}{4\mu^3}\big(\pm ie^{\pm i\mu|x-y|}-e^{-\mu|x-y|}\big),\quad x,y\in \R.$$
This observation inspires us that we may also try to combine the Taylor expansion of $R^{\pm}_0(\mu^4,n,m)$ and the orthogonality
\begin{align*}
    Qv=0=S_0(v_k),\quad \left<Qf,v\right>=0=\left<S_0f,v_k\right>,\quad v_k(n)=n^kv(n),\quad  k=0,1,
\end{align*}
of the orthogonal projection operators $Q,S_0$ in the expansion \eqref{asy expan of regular 0}
\begin{align}\label{new asy expan regular 0}
\begin{split}
M^{-1}(\mu)&=S_0A_0S_0+\mu QA_1Q+\mu^2(QA^{0}_{21}Q+S_0A^{0}_{22}+A^{0}_{23}S_0)\\
&\quad+\mu^3(QA^{0}_{31}+A^{0}_{32}Q)+\mu^3P_1+\Gamma^{0}_{4}(\mu)
\end{split}
\end{align}
to eliminate the singularity of $\mcaK_1$ near $\mu=0$ as its continuous analogue \cite[Lemma 2.5]{MWY24}. However, unlike the perfect form in the continuous case, the more complex structure of the kernel $R^{\pm}_0(\mu^4,n,m)$ introduces significant technical challenges. To address this issue, we establish a modified cancelation Lemma \ref{cancelation lemma} to derive
\begin{align*}
R^{\pm}_0(\mu^4)vQ=O(\mu^{-2}),\  R^{\pm}_0(\mu^4)vS_0=O(\mu^{-1}),\
QvR^{\pm}_0(\mu^4)=O(\mu^{-2}),\ S_0vR^{\pm}_0(\mu^4)=O(\mu^{-1}).
\end{align*}
By virtue of this property, we can classify the operators in \eqref{new asy expan regular 0} above into the following two groups according to the order of $K_A(n,m)$ with regard to $\mu$:
$$O(\mu):S_0A_0S_0,\mu^2QA^{0}_{21}Q,\mu^2S_0A^{0}_{22},\mu^2A^{0}_{23}S_0,\mu^3QA^{0}_{31},\mu^3A^{0}_{32}Q,\Gamma^0_4(\mu),\quad O(1):\mu QA_1Q,\mu^3P_1,$$
where\begin{equation}\label{KA}
K_A(n,m)=\int_{0}^{2}\mu^3\chi_1(\mu)\big[R^{+}_{0}(\mu^4)vAv(R^{+}_{0}-R^{-}_0)(\mu^4)\big](n,m)d\mu.
\end{equation}
Substituting \eqref{new asy expan regular 0} into \eqref{mcaKj} for $j=1$, we can further express $\mcaK_1$ as the sum
\begin{align*}
\mcaK_1=\sum\limits_{A\in O(\mu)}K_A+\sum\limits_{A\in O(1)}K_A.
\end{align*}

For the operators in class $ O(\mu)$, we can prove that $K_A\in\B(\ell^p(\Z))$ for all $1\leq p\leq\infty$. We shall explain this for $A=S_0A_0S_0$ as a model case. In this case, by means of Lemma \ref{cancelation lemma}

and the variable substitution:
\begin{equation}\label{varible substi 0}
 {\rm cos}\theta_{+}=1-\frac{\mu^2}{2} \Longrightarrow\  \frac{d\mu}{d\theta_+}=-\sqrt{1-\frac{\mu^2}{4}},\quad \theta_+\rightarrow0\ {\rm as}\ \mu\rightarrow0\ {\rm and}\ \theta_+\rightarrow-\pi\ {\rm as}\ \mu\rightarrow2,
\end{equation}
we can rewrite \eqref{KA} as a linear combination of the following functions:
\begin{align*}
K^{\pm,1}_0(n,m)&=\int_{-\pi}^{0}e^{-i\theta_+(|n|\pm|m|)}g_1(\theta_+)\chi_{1}(\mu(\theta_+))L^{\pm,1}_{0}(\theta_+,n,m)d\theta_+,\\
K^{\pm,2}_0(n,m)&=\int_{0}^{2}e^{b(\mu)|n|\pm i\theta_+|m|}g_2(\mu)\chi_{1}(\mu)L^{\pm,2}_{0}(\mu,n,m)d\mu,\\
K^{\pm,3}_0(n,m)&=\int_{-\pi}^{0}e^{\pm i\theta_+(|n|+|m|)}g_3(\theta_+){\chi}_{1}(\mu(\theta_+))L^{\pm,3}_{0}(\theta_+,n,m)d\theta_+,
\end{align*}
where $g_j$ and $L^{\pm,j}_0$ satisfy the following properties for any $k=0,1,2$, respectively:
\begin{align}
&\lim\limits_{\theta_+\rightarrow0}g^{(k)}_j(\theta_+)\quad{\rm and }\quad  \lim\limits_{\mu\rightarrow0}g^{(k)}_2(\mu)\quad {\rm exist},\label{property in mcaK1}\\
\sum\limits_{j\in\{1,3\}}^{}\sup\limits_{\theta_{+}\in(-\pi,0)}\big|(\partial^k_{\theta_+}&L^{\pm,j}_{0})(\theta_+,n,m)\big|+\sup\limits_{\mu\in(0,\mu_0]}\big|e^{b(\mu)|n|}(\partial^k_\mu L^{\pm,2}_{0})(\mu,n,m)\big|\lesssim \|\left<\cdot\right>^{2k+4}V(\cdot)\|_{\ell^1},\notag
\end{align}
uniformly in $n,m\in\Z$.
Moreover, we have $\lim\limits_{\theta_+\rightarrow0}g_j(\theta_+)=0$ and $\lim\limits_{\mu\rightarrow0}g_2(\mu)=0$. This property, combined with
\begin{align*}
 \frac{1}{\big|b'(\mu)|n|\pm i\theta'_+(\mu)|m|\big|^2}\lesssim (|n|+|m|)^{-2},\quad {\rm uniformly\ in}\ (n,m)\neq(0,0)\ {\rm and}\ \mu\in(0,2),
 \end{align*}
enables us to apply integration by parts twice to $K^{\pm,j}_0(n,m)$ obtaining
\begin{equation*}
|K^{\pm,j}_{0}(n,m)|\lesssim \left<|n|\pm|m|\right>^{-2},\quad j=1,2,3.
\end{equation*}
Then the $\ell^p$ boundedness for $1\leq p\leq\infty$ is derived by the Shur test Lemma \ref{shur test}.

In the preceding analysis, we notice that a crucial point is that the limits of $g_j$ vanish as $\theta_+\rightarrow0$ and $\mu\rightarrow0$. However, for operators in class $O(1)$, which lack this vanishing property, the singular terms $O(\left<|n|\pm|m|\right>^{-1})$ will emerge. To solve this problem, we shall appeal to the theory of the discrete singular integrals on the lattice, see Appendix \ref{sec of appendix}. 

Specifically, we consider $A=\mu^3P_1$ as a model case of the operator class $O(1)$. In this case, an analogous argument as above yields that \eqref{KA} can be written as a linear combination of these functions:
\begin{align*}
K^{\pm,1}_{P_1}(n,m)&=\int_{-\pi}^{0}e^{-i\theta_+(|n|\pm|m|)}h_1(\theta_+){\chi}_1(\mu(\theta_+))L^{\pm,1}_{P_1}(\theta_+,n,m)d\theta_+,\\
K^{\pm,2}_{P_1}(n,m)&=\int_{0}^{2}e^{b(\mu)|n|\pm i\theta_+|m|}h_2(\mu){\chi}_1(\mu)L^{\pm,2}_{P_1}(\mu,n,m)d\mu,
\end{align*}
where $h_j$ and $L^{\pm,j}_{P_1}$ satisfy the similar property as in \eqref{property in mcaK1}. Applying integration by parts twice to $K^{\pm,j}_{P_1}(n,m)$, we find that \eqref{KA} equals
$$\frac{i-1}{8}\Big(k^{+}_1(n,m)+k^{-}_1(n,m)+k^{+}_2(n,m)+k^{-}_2(n,m)\Big)+O(\left<|n|\pm|m|\right>^{-2}),$$
where
$$k^{\pm}_1(n,m)=\frac{\phi(||n|\pm|m||^2)}{|n|\pm |m|},\quad k^{\pm}_{2}(n,m)=\frac{\phi(||n|-|m||^2)}{|n|\pm i|m|}$$
with $\phi(s)$ being smooth cut-off functions supported in $\{s:|s|\geq1\}$.
In view that the integral operator $T_{k^{\pm}_{\ell}}$  associated with the kernel $k^{\pm}_{\ell}(x,y)$ in the continuous setting is not a Calder\'{o}n-Zygmund operator, we cannot directly apply Theorem \ref{Tdis lemma}. To overcome this, we utilize the following relation:
\begin{equation*}
({k}^{\pm}_1f)(n)=\big[(\chi_+\widetilde{k}_1\chi_{\mp}-\chi_-\widetilde{k}_1\chi_{\pm})(1+\tau)f\big](n),\
({k}^{\pm}_2f)(n)=\big[(\chi_+\widetilde{k}^{\pm}_2\chi_{+}-\chi_-\widetilde{k}^{\pm}_2\chi_{-})(1+\tau)f\big](n),
\end{equation*}
where $\chi_{\pm}=\chi_{\Z^{\pm}}$ is the characteristic function on $\Z^{\pm}:=\{m\in\Z:\pm m>0\}$, $(\tau f)(n)=f(-n)$ and
$$\widetilde{k}_1(n,m)=\frac{\phi(|n-m|^2)}{n-m},\quad\widetilde{k}^{\pm}_2(n,m)=\frac{\phi(|n-m|^2)}{n\pm im}. $$
This relation allows us to reduce the $\ell^p$ boundedness of $k^{\pm}_{1}$ to $\widetilde{k}_1$, and $k^{\pm}_2$ to $\widetilde{k}^{\pm}_2$, both of which are $\ell^p$ bounded for $1<p<\infty$,
since their continuous analogues $T_{\widetilde{k}_1}$ and $T_{\widetilde{k}^{\pm}_2}$ are Calder\'{o}n-Zygmund operators~(for further details, we refer to \cite[Lemma 3.3]{MWY24}).

\begin{itemize}
\item \underline{\bf{For the high energy part} ${\bm{\mcaK_3}}.$} In this case,
the first distinction from $\mcaK_1$ is that it is not
\end{itemize}
straightforward to utilize the orthogonality
$\widetilde{Q}\tilde{v}=0=\big<\widetilde{Q}f,\tilde{v}\big>$
 of the projection operator $\widetilde{Q}$ in the expansion \eqref{asy expan regular 2}
\begin{align*}
M^{-1}(\mu)=\widetilde{Q}B_0\widetilde{Q}+
(2-\mu)^{\frac{1}{2}}(\widetilde{Q}B^{0}_{11}+B^{0}_{12}\widetilde{Q})+(2-\mu)^{\frac{1}{2}}\widetilde{P}_1+(2-\mu) B^{0}_{21}+\Gamma^{0}_{\frac{3}{2}}(2-\mu)
\end{align*}
 to eliminate the singularity at $\mu=2$, where $\tilde{v}(n)=(Jv)(n):=(-1)^nv(n)$. To this end, we will make use of the unitary transform $J$ above, which satisfies
 $$JR^{\pm}_{-\Delta}(\mu^2)J=-R^{\mp}_{-\Delta}(4-\mu^2),\quad \mu\in(0,2).$$
 This formula together with $J^2=I$ allows us to rewrite \begin{equation}\label{R0vBvR0 }
R^{+}_{0}(\mu^4)v Bv(R^{+}_{0}-R^{-}_0)(\mu^4)=\frac{1}{4\mu^4}J\big(R^{-}_{-\Delta}(4-\mu^2)+JR_{-\Delta}(-\mu^2)J\big)\tilde{v}B\tilde{v}(R^{-}_{-\Delta}-R^{+}_{-\Delta})(4-\mu^2)J.
\end{equation}
Such form motivates us to combine the orthogonality of $\widetilde{Q}$ above and the Taylor expansion of
\begin{align*}
R^{\mp}_{-\Delta}(4-\mu^2,n,m)=\frac{\pm ie^{\pm i\tilde{\theta}_{+}|n-m|}}{2{\rm sin}\tilde{\theta}_+}=\frac{\pm i}{2{\rm sin}\tilde{\theta}_+}\Big(e^{\pm i\tilde{\theta}_{+}|n|}\mp i\tilde{\theta}_{+}m\int_{0}^{1}({\rm sign}(n-\rho m))e^{\pm i\tilde{\theta}_{+}|n-\rho m|}d\rho\Big)
\end{align*}
deriving
$$R^{\mp}_{-\Delta}(4-\mu^2)\tilde{v}\widetilde{Q}=O(1),\quad \widetilde{Q}\tilde{v}R^{\mp}_{-\Delta}(4-\mu^2)=O(1),$$
where $\tilde{\theta}_{+}:=\tilde{\theta}_{+}(\mu)$ satisfies ${\rm cos}\tilde{\theta}_{+}=\frac{\mu^2}{2}-1$ with $ \tilde{\theta}_{+}\in(-\pi,0)$.
Similarly, we can express $\mcaK_3$ as
\begin{align*}
\mcaK_3=\sum\limits_{B\in O(1)}K_B+\sum\limits_{B\in O((2-\mu)^{-\frac{1}{2}})}K_B.
\end{align*}
Here
$$O(1):\widetilde{Q}B_0\widetilde{Q},(2-\mu)^{\frac{1}{2}}\widetilde{Q}B^{0}_{11},(2-\mu)^{\frac{1}{2}}B^{0}_{12}\widetilde{Q},(2-\mu) B^{0}_{21},\Gamma^{0}_{\frac{3}{2}}(2-\mu),\quad O((2-\mu)^{-\frac{1}{2}}):(2-\mu)^{\frac{1}{2}}\widetilde{P}_1$$
and
\begin{equation}\label{KB}
K_B(n,m)=\int_{0}^{2}\mu^3\chi_3(\mu)\big[R^{+}_{0}(\mu^4)vBv(R^{+}_{0}-R^{-}_0)(\mu^4)\big](n,m)d\mu.
\end{equation}
We will prove that $K_B$ is $\ell^p$ bounded for all $1\leq p\leq\infty$ for the operators in the class $O(1)$ and $\ell^p$ bounded for $1< p<\infty$ for the operators in the class $O((2-\mu)^{-\frac{1}{2}})$.

For further explanation, we consider $B=\widetilde{Q}B_0\widetilde{Q}$ as a model case of the class $O(1)$. In this case, by virtue of \eqref{R0vBvR0 } and the following variable substitution:
\begin{align}\label{varible substi2 introduction}
{\rm cos}\tilde{\theta}_{+}=\frac{\mu^2}{2}-1 \Longrightarrow\  \frac{d\mu}{d\tilde{\theta}_+}=\sqrt{1-\frac{\mu^2}{4}},\quad \tilde{\theta}_+\rightarrow-\pi\ {\rm as}\ \mu\rightarrow0\ {\rm and}\ \tilde{\theta}_+\rightarrow0\ {\rm as}\  \mu\rightarrow2,
\end{align}
we can express \eqref{KB} as a linear combination of the following functions:
\begin{align*}
\widetilde{K}^{\pm,1}_{0}(n,m)&=(-1)^{n+m}\int_{-\pi}^{0}e^{i\tilde{\theta}_+(|n|\pm |m|)}\tilde{g}_1(\tilde{\theta}_+){\chi}_3(\mu(\tilde{\theta}_+))\widetilde{L}^{\pm,1}_0(\tilde{\theta}_+,n,m)d\tilde{\theta}_{+},\\
\widetilde{K}^{\pm,2}_{0}(n,m)&=(-1)^m\int_{-\pi}^{0}e^{\pm i\tilde{\theta}_+(|m|\pm i|n|)}\tilde{g}_2(\tilde{\theta}_+){\chi}_3(\mu(\tilde{\theta}_+))\widetilde{L}^{\pm,2}_0(\mu(\tilde{\theta}_+),n,m)d\tilde{\theta}_+,
\end{align*}
where $\tilde{g}_j(\tilde{\theta}_+)$ satisfies the similar property \eqref{property in mcaK1} and
\begin{equation*}
\sup\limits_{\tilde{\theta}_{+}\in(-\pi,0)}\big|(\partial^k_{\tilde{\theta}_+}\widetilde{L}^{\pm,1}_{0})(\tilde{\theta}_+,n,m)\big|+\sup\limits_{\tilde{\theta}_{+}\in[\gamma_1,0)}\big|(\partial^k_{\tilde{\theta}_+}\widetilde{L}^{\pm,2}_{0})(\mu(\tilde{\theta}_+),n,m)\big|\lesssim \|\left<\cdot\right>^{2k+2}V(\cdot)\|_{\ell^1},\ k=0,1,2
\end{equation*}
uniformly in $n,m\in\Z$, where $\gamma_1\in(-\pi,0)$ satisfies ${\rm cos}\gamma_1=\frac{(2-\mu_0)^2}{2}-1$. Thanks to the substitution \eqref{varible substi2 introduction} above, which contributes a factor of $(2-\mu)^{\frac{1}{2}}$, this allows that the limits of $\tilde{g}_j(\tilde{\theta}_+)$ vanish as $\tilde{\theta}_+\rightarrow0$. This constitutes another difference from $\mcaK_1$, where the variable substitution \eqref{varible substi 0} does not alter the singularity near $\mu=0$. Using the argument analogous to $K^{\pm,j}_0$ above, we can also derive
\begin{equation*}
|\widetilde{K}^{\pm,j}_{0}(n,m)|\lesssim \left<|n|\pm|m|\right>^{-2},\quad j=1,2,
\end{equation*}
and so the desired result is obtained. For the operator $(2-\mu)^{\frac{1}{2}}\widetilde{P}_1$, combining the arguments for $K_B$ with ${\widetilde{Q}B_0\widetilde{Q}}$ and $K_A$ with $A=\mu^3P_1$ above, we can demonstrate that \eqref{KB} equals
\begin{align*}
\frac{(-1)^{n+m}}{4}\big(k^{+}_1(n,m)+k^{-}_1(n,m)\big)+O(\left<|n|\pm|m|\right>^{-2}),
\end{align*}
which is $\ell^p$ bounded for $1<p<\infty$.

In summary, we establish $W_+\in\B(\ell^p(\Z))$ for all $1<p<\infty$.
\subsubsection{\bf{Counterexamples of $\ell^1$ and $\ell^{\infty}$ boundedness}}
As demonstrated above, the operators in $O(\mu)$ for $\mcaK_1$ and those in $O(1)$ for $\mcaK_3$ are always $\ell^p$ bounded for all $1\leq p\leq\infty$. Consequently, the boundedness of $W_+$ at the endpoints $p=1,\infty$ reduces
 to analyzing the operators $K_1$, $K_{P_1}$ and $K_{\widetilde{P}_1}$, where $$K_1=K_A~(A=\mu QA_1Q),\quad K_{P_1}=K_A~(A=\mu^3P_1), \quad   K_{\widetilde{P}_1}=K_B~(B=(2-\mu)^{\frac{1}{2}}\widetilde{P}_1).$$
These three operators may fail to be bounded on both $\ell^1(\Z)$ and $\ell^{\infty}(\Z)$. While this observation is not sufficient to disprove such boundedness of $W_+$, we can examine the behavior by taking the characteristic function $f_{N}(n):=\chi_{[-N,N]}(n)$ on the interval $[-N,N]$ with $N\in\N^{+}$ as test function.  Specifically, we can show
\begin{itemize}
 \item$|(K_{P_1}+K_{\widetilde{P}_1})f_{N}(N+2)|\rightarrow+\infty,N\rightarrow+\infty$ and $(K_{P_1}+K_{\widetilde{P}_1})f_{1}\notin\ell^{1}(\Z)$,
 \vskip0.2cm
\item$\sup\limits_{N\in\N^+}\|K_1f_{N}\|_{\ell^{\infty}}<\infty$ and $K_1f_{1}\in\ell^1(\Z)$.
 \end{itemize}
This gives that $W_{+}\notin\B(\ell^{1}(\Z))\cup\B(\ell^{\infty}(\Z))$.

\subsection{Organizations of the paper}
The remainder of this paper is organized as follows. Section \ref{sec of Preliminaries} presents preliminary materials, including the limiting absorption principle and asymptotic expansions of $M^{-1}(\mu)$. Sections \ref{sect of mcaK1}$\sim$\ref{sec of mcaK3} are devoted to the proof of Theorem \ref{main theorem}, while Section \ref{sec of counterexample} contains the proof of Theorem \ref{theorem at endpoints}. In Section \ref{sec of application}, we will apply the previously established $\ell^p$-boundedness of wave operators to derive decay estimates for solutions of the discrete beam equation \eqref{discrete Beam equation}. Finally, Appendix \ref{sec of appendix} provides a review of  discrete Calder\'{o}n-Zygmund operators on the lattice $\mathbb{Z}^d$.

\section{Preliminaries}\label{sec of Preliminaries}
This section is devoted to establishing the limiting absorption principle for the operator $H$ and investigating the asymptotic behaviors of $R^{+}_V(\mu^4)$ near $\mu=0$ and $\mu=2$.
\subsection{Limiting absorption principle}\label{subsec of LAP} We begin by recalling some basics of the resolvents. From the definition of Laplacian $\Delta$ on $\Z$ in \eqref{defin of laplacian}, the bi-Laplacian $\Delta^2$ on $\Z$ is given by
\begin{equation*}
(\Delta^2\phi)(n)=(\Delta(\Delta\phi))(n)=\phi(n+2)-4\phi(n+1)+6\phi(n)-4\phi(n-1)+\phi(n-2).
\end{equation*}
Consider the Fourier transform $\mcaF$: $\ell^2(\Z)\rightarrow L^2(\T), \T=\R/2\pi\Z=[-\pi,\pi]$, defined by
\begin{equation}\label{fourier transform}
(\mcaF\phi)(x):=(2\pi)^{-{\frac{1}{2}}}\sum_{n\in\Z}^{}e^{-inx}\phi(n), \quad\forall\ \phi\in\ell^2(\Z).
\end{equation}
Under this transform, we have
\begin{equation}\label{unitary equivalent}
(\mcaF\Delta^2\phi)(x)=(2-2{\rm cos}x)^2(\mcaF\phi)(x):=\mcaM(x)(\mcaF\phi)(x),\quad  x\in\T=[-\pi,\pi],
\end{equation}
which implies that the spectrum of $\Delta^2$ is purely absolutely continuous and equals $[0,16]$. 
Let
$$\quad R_0(z):=(\Delta^2-z)^{-1},\quad R_V(z):=(H-z)^{-1},\quad z\in\rho(H)~{\rm(the\ resolvent\ set\ of}\ H{\rm)}$$
  be the resolvents of $\Delta^2$ and $H$, respectively and define their boundary values on $(0,16)$ via
  \begin{align*}
  R^{\pm}_{0}(\lambda):=\lim\limits_{\varepsilon\downarrow0}R_{0}(\lambda\pm i\varepsilon ),\quad
  R^{\pm}_{V}(\lambda):=\lim\limits_{\varepsilon\downarrow0}R_{V}(\lambda\pm i\varepsilon ),\quad\lambda\in(0,16).
  \end{align*}
  The existence of $R^{\pm}_0(\lambda)$ in $\B(\ell^{2,s}(\Z),\ell^{2,-s}(\Z))$ for $s>\frac{1}{2}$ follows from the resolvent decomposition
\begin{equation*}
 R_0(z)=\frac{1}{2\sqrt z}\left(R_{-\Delta}(\sqrt z)-R_{-\Delta}(-\sqrt z)\right),\quad\sqrt{z}=\sqrt{|z|}e^{i\frac{arg z}{2}},\  0<argz<2\pi,
 \end{equation*}
   and the known limiting absorption principle for $-\Delta$~(cf.\cite{KKK06})
  $$R^{\pm}_{-\Delta}(\mu):=\lim\limits_{\varepsilon\downarrow0}R_{-\Delta}(\mu\pm i\varepsilon ),\quad \mu\in(0,4),$$
which exists in the operator norm of $\B(\ell^{2,s}(\Z),\ell^{2,-s}(\Z))$ for $s>\frac{1}{2}$, where $R_{-\Delta}(\omega)=(-\Delta-\omega)^{-1}$ is the resolvent of $-\Delta$. 

 Moreover, using the kernel of $R_{-\Delta}(\omega)$~(cf. \cite{KKK06}):
 \begin{equation}\label{kernel of lapl resolvent}
R_{-\Delta}(\omega,n,m)=\frac{-ie^{-i\theta(\omega)|n-m|}}{2{\rm sin}\theta(\omega)},\quad n,m\in\Z,
\end{equation}
where $\theta(\omega)$ is the solution to the equation $2-2{\rm cos}\theta=\omega$
in the domain $$\mcaD:=\left\{\theta(\omega)=a+ib:-\pi\leq a\leq\pi,b<0\right\},$$
we can derive explicit expression for the integral kernel of the free resolvent $R^{\pm}_0(\mu^4)$ as follows. 
\begin{lemma}\label{lemma of kernel R0 boundary}{\rm(\cite[Lemma 2.3]{HY25} )}
{ For $\mu\in(0,2)$, the kernel of $R^{\pm}_0(\mu^4)$ is given by
\begin{align}\label{kernel of R0 boundary}
R^{\pm}_{0}(\mu^4,n,m)
=\frac{1}{4\mu^3}\Big(\pm ia_1(\mu)e^{\mp i\theta_{+}|n-m|}+a_2(\mu)e^{b(\mu)|n-m|}\Big),
\end{align}
where $\theta_+:=\theta_{+}(\mu)$ satisfies $2-2{\rm cos}\theta_{+}=\mu^2$ with $\theta_{+}\in(-\pi,0)$ and
\begin{equation}\label{exp of a1 a2 bu}
a_1(\mu)=\frac{1}{\sqrt{1-\frac{\mu^2}{4}}},\quad a_2(\mu)=\frac{-1}{\sqrt{1+\frac{\mu^2}{4}}},\quad b(\mu)={\rm ln} \big(1+\frac{\mu^2}{2}-\mu(1+\frac{\mu^2}{4})^{\frac{1}{2}}\big).
\end{equation}
}
\end{lemma}
As shown above, the resolvent $R^\pm_0(\mu^4)$ exhibits singular behavior with order $O(\mu^{-3})$ near $\mu=0$ and $O((2-\mu)^{-1/2})$ near $\mu=2$. Indeed, more precise asymptotic expansions of $R^{\pm}_0(\mu^4)$ near these critical points in suitable weighted space $\B(\ell^{2,s}(\mathbb{Z}),\ell^{2,-s}(\mathbb{Z}))$ can be found in \cite[Lemma 5.4]{HY25}.

For the perturbed operator $H = \Delta^2 + V$, we can derive the following LAP.
\begin{lemma}\label{LAP-lemma}{\rm(\cite[Theorem 2.4]{HY25} )}
Let $H=\Delta^2+V$ with $|V(n)|\lesssim \left<n\right>^{-\beta}$ for $\beta>1$ and $\mcaI=(0,16)$. Denote by $[\beta]$ the biggest integer no more than $\beta$. Then the following statements hold.
\begin{itemize}
\item [\rm(i)] The point spectrum $\sigma_p(H)\cap\mcaI$ is discrete, with each eigenvalue has a finite multiplicity, and the singular continuous spectrum $\sigma_{sc}(H)=\varnothing$.
\vskip0.2cm
\item [\rm(ii)] Let $j\in\left\{0,\cdots,[\beta]-1\right\}$ and $j+\frac{1}{2}<s\leq[\beta]$, then 
the following norm limits
    \begin{equation*}
   \frac{d^{j}}{d\lambda^j}(R^{\pm}_V(\lambda))=\lim\limits_{\varepsilon\downarrow0}R^{(j)}_V(\lambda\pm i\varepsilon) \quad {\rm in} \quad \B(\ell^{2,s}(\Z),\ell^{2,-s}(\Z))
    \end{equation*}
are norm continuous from $\mcaI\setminus\sigma_p(H)$ to $\B(\ell^{2,s}(\Z),\ell^{2,-s}(\Z))$, where $R^{(j)}_{V}(z)$ denotes the jth derivative of $R_{V}(z)$.
\end{itemize}
\end{lemma}

We remark that the derivation of this LAP relies on the commutator estimates and Mourre theory. The upper bound of $s$ is closely related to the regularity of $H$. For more details, see \cite[Section 2 and Appendix A]{HY25}.

Throughout the paper, we always assume that $H$ has no positive eigenvalues in $\mcaI$. As a consequence of this lemma, $R^{\pm}_V(\mu^4)$ exists in $\B(\ell^{2,s}(\Z),\ell^{2,-s}(\Z))$ for $\frac{1}{2}<s\leq[\beta]$ and all $\mu\in(0,2)$.

\subsection{Asymptotic expansions of $R^{+}_V(\mu^4)$ } \label{subsec:asymptotics}
In this subsection, we further investigate the asymptotic behaviors of $R^+_V(\mu^4)$ near $\mu=0$ and $\mu=2$. For this purpose, we introduce
\begin{equation*}
M(\mu)=U+vR^{+}_{0}(\mu^4)v,\quad\mu\in(0,2),\quad U={\rm sign}\left(V(n)\right),\quad v(n)=\sqrt{|V(n)|},
\end{equation*}
and denote the inverse by $M^{-1}(\mu)$ as long as it exists. Indeed, building upon Lemma \ref{LAP-lemma}, we can establish such invertibility. Moreover, it allows us to reduce the asymptotic analysis of $R_V^+(\mu^4)$ to studying the behavior of $M^{-1}(\mu)$ near $\mu=0$ and $\mu=2$. Specifically, we have:
\begin{lemma}\label{existence of inverse of Mmu}{\rm (\cite[Corollary 2.5]{HY25})}
Let $H,V$ be as in Lemma \ref{LAP-lemma}. For any $\mu\in(0,2)$, $M(\mu)$ is invertible on $\ell^2(\Z)$ and satisfies the relation below in $\B(\ell^{2,s}(\Z),\ell^{2,-s}(\Z))$ for $\frac{1}{2}<s\leq[\beta]$:
\begin{equation}\label{inverti relation}
R^{+}_V(\mu^4)V=R^{+}_0(\mu^4)vM^{-1}(\mu)v.
\end{equation}
\end{lemma}
Prior to presenting the asymptotic expansions of $M^{-1}(\mu)$, we first recall another characterization of zero and sixteen resonances of $H$ established in \cite[Remark 5.3]{HY25}, which provides a direct approach to compute $M^{-1}(\mu)$ via Von-Neumann series expansion.
Let
$$\left<f,g\right>:=\sum\limits_{m\in\Z}^{}f(m)\overline{g(m)},\quad f,g\in\ell^2(\Z).$$
For simplicity, we denote the kernel of the operator $\mcaT$ restricted to the space $X$ by ${\rm Ker}\mcaT\big|_{X}=\{f\in X:\mcaT f=0\}$.
\begin{definition}\label{definition at zero}\textbf{{\rm(}Zero resonances{\rm)}} Let $H=\Delta^2+V$ with $|V(n)|\lesssim \left<n\right>^{-\beta}$ for some $\beta>0$ and  $G_j$~{\rm(}$j=-1,0,1,3${\rm)} be the integral operator with kernel $G_{j}(n,m)$:
\begin{align*}
G_{-1}(n,m)&=\frac{1}{8}-\frac{1}{2}|n-m|^2,\quad G_{0}(n,m)=\frac{1}{12}\left(|n-m|^3-|n-m|\right),\\
G_{1}(n,m)&=\frac{1}{3}|n-m|^4-\frac{5}{6}|n-m|^2+\frac{3}{16},\\
G_{3}(n,m)&=|n-m|^6-\frac{35}{4}|n-m|^4+\frac{259}{16}|n-m|^2-\frac{225}{64}.
\end{align*}
Let $I$ be the identity operator and define
\begin{equation*}  P:=\left\|V\right\|^{-1}_{\ell^1}\left<\cdot,v\right>v,\quad Q:=I-P,\quad T:=U+vG_0v.
\end{equation*}
Let $S_0$ be the orthogonal projection onto the kernel space $ {\rm Ker}QvG_{-1}vQ\big|_{Q\ell^2(\Z)}$
and denote by $D_0$ the inverse of $QvG_{-1}vQ+S_0$ on $Q\ell^2(\Z)$. We say that
\begin{itemize}
\item[{\rm(i)}] ${\bm{0}}$ is \textbf{a regular point of} ${\bm{H}}$ if $T_0:=S_0TS_0:S_0\ell^2(\Z)\rightarrow S_0\ell^2(\Z)$ is invertible on $S_0\ell^2(\Z)$.
\vskip0.15cm
\item[{\rm(ii)}] Assume that $T_0$ is not invertible on $S_0\ell^2(\Z)$. Let $S_1$ be the orthogonal projection onto the kernel space ${\rm Ker}T_0\big|_{S_0\ell^2(\Z)}$.
We say that ${\bm{0}}$ is \textbf{a first kind resonance of} ${\bm{H}}$ if
\begin{align*}
T_1=S_1vG_1vS_1+\frac{8}{\|V\|_{\ell^1}}S_1vG_{-1}vPvG_{-1}vS_1+64S_1TD_0TS_1
\end{align*}
is invertible on $S_1\ell^2(\Z)$.
\vskip0.15cm
\item[{\rm(iii)}] Assume that $T_1$ is not invertible on $S_1\ell^2(\Z)$. Let $S_2$ be the orthogonal projection onto the kernel space ${\rm Ker}T_1\big|_{S_1\ell^2(\Z)}$ and
denote by $D_2$ the inverse of $T_1+S_2$ on $S_1\ell^2(\Z)$. Then we say that ${\bm{0}}$ is \textbf{a second kind resonance of} ${\bm{H}}$ if
\begin{align*}
T_2&=\frac{1}{6!}\Big(S_2vG_3vS_2-\frac{8\cdot6!}{\|V\|_{\ell^1}}S_2{T}^2S_2-\frac{6!}{64}S_2vG_1vD_0vG_1vS_2\Big)\notag\\
&\quad+\frac{64}{\|V\|^2_{\ell^1}}\big(S_2TvG_{-1}vD_0-\frac{\|V\|_{\ell^1}}{8}S_2vG_1vD_0TD_0\big)D_2\\
&\quad\times \big(D_0vG_{-1}vTS_2-\frac{\|V\|_{\ell^1}}{8}D_0TD_0vG_1vS_2\big)\notag
\end{align*}
is invertible on $S_2\ell^2(\Z)$.
\end{itemize}
\end{definition}

\begin{definition}\label{defini of sixteen}\textbf{{\rm(}Sixteen resonances{\rm)}}
Let $H=\Delta^2+V$ with $|V(n)|\lesssim \left<n\right>^{-\beta}$ for some $\beta>0$ and  $\widetilde{G}_j~(j=0,1,2)$ be the integral operator with kernel $\widetilde{G}_j(n,m)$:
\begin{align*}
    \widetilde{G}_{0}(n,m)&=\frac{1}{32\sqrt2}\big(2\sqrt2|n-m|-{(2\sqrt2-3)}^{|n-m|}\big),\quad \widetilde{G}_{1}(n,m)=2|n-m|^2-\frac{13}{8},\\
  \widetilde{G}_{2}(n,m)&=-\frac{1}{24}|n-m|^3+\frac{5}{48}|n-m|-\big(2\sqrt2-3\big)^{|n-m|}\Big(\frac{\sqrt2}{2}|n-m|-\frac{1}{8}+\frac{15}{256\sqrt2}\Big).
\end{align*}
Define
\begin{equation}\label{J}
\widetilde{P}:=\left\|V\right\|^{-1}_{\ell^1}\left<\cdot,\tilde{v}\right>\tilde{v},\quad \widetilde{Q}:=I-\widetilde{P},\quad \widetilde{T}:=U+\tilde{v}\widetilde{G}_0\tilde{v},\quad \tilde{v}(n)=(Jv)(n):=(-1)^nv(n).
\end{equation}
\begin{itemize}
    \item [{\rm(i)}] ${\bm{16}}$ is \textbf{a regular point of} ${\bm{H}}$ if $\widetilde{T}_0:=\widetilde{Q}\widetilde{T}\widetilde{Q}$ is invertible on $\widetilde{Q}\ell^2(\Z)$.
    \vskip0.15cm
 \item [{\rm(ii)}] Assume that $\widetilde{T}_0$ is not invertible on $\widetilde{Q}\ell^2(\Z)$. Let $\widetilde{S}_0$ be the orthogonal projection onto the kernel space ${\rm Ker}\widetilde{T}_0\big|_{\widetilde{Q}\ell^2(\Z)}$. We say that
 ${\bm{16}}$ is \textbf{a resonance of} ${\bm{H}}$ if
 \begin{equation*}
\widetilde{T}_1=\widetilde{S}_0\tilde{v}\widetilde{G}_1\tilde{v}\widetilde{S}_0+\frac{32}{\|V\|_{\ell^1}}\widetilde{S}_0\widetilde{T}^2\widetilde{S}_0\ \ is\ \  invertible\  on \ \ \widetilde{S}_0\ell^2(\Z).
\end{equation*}
\vskip0.15cm
\item[{\rm(iii)}] Assume that $\widetilde{T}_1$ is not invertible on $\widetilde{S}_0\ell^2(\Z)$. Let $\widetilde{S}_1$ be the orthogonal projection onto the kernel space $ {\rm Ker}\widetilde{T}_1\big|_{\widetilde{S}_0\ell^2(\Z)}$. We say that
${\bm{16}}$ is \textbf{an eigenvalue of} ${\bm{H}}$ if
$\widetilde{T}_2=\widetilde{S}_1\tilde{v}\widetilde{G}_2\tilde{v}\widetilde{S}_1$ is invertible on $\widetilde{S}_1\ell^2(\Z)$.
\end{itemize}
\end{definition}
\begin{remark}
{\rm (1) We point out that $G_0,J\widetilde{G}_0J$ actually are the fundamental solutions of $\Delta^2$ and $\Delta^2-16$, respectively, i.e., $\Delta^2G_0=\delta$ and $(\Delta^2-16)J\widetilde{G}_0J=\delta$.
\vskip0.1cm
(2)~When $\beta>9$, we can obtain concise expressions of the kernel subspaces, i.e., orthogonal projection spaces $S_j\ell^2(\Z)$ and $\widetilde{S}_j\ell^2(\Z)$ above.
\begin{itemize}
    \item \underline{Orthogonal projection spaces $S_j\ell^2(\Z)$.}
\end{itemize}
\begin{align*}
 S_0\ell^2(\Z)&={\rm Ker}QvG_{-1}vQ\big|_{Q\ell^2(\Z)}=\big\{f\in\ell^2(\Z):\left<f,v_{k}\right>=0,\ k=0,1\big\},\quad v_k(n)=n^kv(n),\\
S_1\ell^2(\Z)&={\rm Ker}T_0\big|_{S_0\ell^2(\Z)}=\big\{f\in\ell^2(\Z):\left<f,v_{k}\right>=0,\ k=0,1,\ S_0Tf=0\big\},\\
S_2\ell^2(\Z)&={\rm Ker}T_1\big|_{S_1\ell^2(\Z)}=\big\{f\in\ell^2(\Z):\left<f,v_{k}\right>=0,\ k=0,1,2,\ QTf=0\big\},\\
S_3\ell^2(\Z)&:={\rm Ker}T_2\big|_{S_2\ell^2(\Z)}=\big\{f\in\ell^2(\Z):\left<f,v_{k}\right>=0,\ k=0,1,2,3,\ Tf=0\big\}.
\end{align*}
\begin{itemize}
    \item \underline{Orthogonal projection spaces $\widetilde{S}_j\ell^2(\Z)$.}
\end{itemize}
\begin{align*}
\widetilde{S}_0\ell^2(\Z)&={\rm Ker}\widetilde{T}_0\big|_{\widetilde{Q}\ell^2(\Z)}=\big\{f\in\ell^2(\Z):\left<f,\tilde{v}\right>=0,\ \widetilde{Q}\widetilde{T}f=0\big\},\quad \tilde{v}_k=Jv_k,\\
\widetilde{S}_1\ell^2(\Z)&={\rm Ker}\widetilde{T}_1\big|_{\widetilde{S}_0\ell^2(\Z)}=\big\{f\in\ell^2(\Z):\left<f,\tilde{v}_k\right>=0,\ k=0,1,\ \widetilde{T}_0f=0\big\},\\
\widetilde{S}_2\ell^2(\Z)&:={\rm Ker}\widetilde{T}_2\big|_{\widetilde{S}_1\ell^2(\Z)}=\big\{f\in\ell^2(\Z):\left<f,\tilde{v}_k\right>=0,\ k=0,1,\ \widetilde{T}_2f=0\big\}.
\end{align*}
We note that these spaces have the following inclusion relations:
\begin{align*}
&S_3\ell^2(\Z)\subseteq S_2\ell^2(\Z)\subseteq S_1\ell^2(\Z)\subseteq S_0\ell^2(\Z)\subseteq Q\ell^2(\Z),\\
&\widetilde{S}_2\ell^2(\Z)\subseteq\widetilde{S}_1\ell^2(\Z)\subseteq \widetilde{S}_0\ell^2(\Z) \subseteq \widetilde{Q}\ell^2(\Z).
\end{align*}
Particularly, we can show that $S_3\ell^2(\Z)=\{0\}$ and $\widetilde{S}_2\ell^2(\Z)=\{0\}$, which completes the entire inversion process of $M^{-1}(\mu)$. For a detailed proof of these facts, we refer to \cite[Lemma 5.2]{HY25}.

 In other word, Definitions \ref{definition at zero}$\sim$\ref{defini of sixteen} can be simply expressed as:
\begin{itemize}
\item[{\rm(i)}] $0$ is a regular point of $H$ if and only if $S_1\ell^2(\Z)=\{0\}$.
\item[{\rm(ii)}] $0$ is a first kind resonance of $H$ if and only if $S_1\ell^2(\Z)\neq\{0\}$ and $S_2\ell^2(\Z)=\{0\}$.
\item[{\rm(iii)}] $0$ is a second kind resonance of $H$ if and only if $S_2\ell^2(\Z)\neq\{0\}$.
\item[{\rm(iv)}] $16$ is a regular point of $H$ if and only if $\widetilde{S}_0\ell^2(\Z)=\{0\}$.
\item[{\rm(v)}] $16$ is a resonance of $H$ if and only if $\widetilde{S}_0\ell^2(\Z)\neq\{0\}$ and $\widetilde{S}_{1}\ell^2(\Z)=\{0\}$.
\item[{\rm(vi)}] $16$ is an eigenvalue of $H$ if and only if $\widetilde{S}_{1}\ell^2(\Z)\neq\{0\}$.
\end{itemize}

(3) Moreover, we remark that these orthogonal projection spaces $S_j\ell^2(\Z)$ (resp. $\widetilde{S}_j\ell^2(\Z)$) are intimately linked to the solutions of difference equation $H\phi=0$~(resp. $H\phi=16\phi$) in suitable weighted space $W_{\sigma}(\Z)$. More precisely, we have
\begin{lemma}{\rm(\cite[Lemma 6.1]{HY25})}\label{charac of solution and space}
Let $H=\Delta^2+V$ on $\Z$ and $|V(n)|\lesssim \left<n\right>^{-\beta}$ with $\beta>9$, then
\begin{itemize}
\item [{\rm(i)}]$f\in S_1\ell^2(\Z)\Longleftrightarrow \exists\ \phi\in W_{3/2}(\Z)$ such that $H\phi=0$. Moreover, $f=Uv\phi$ and $\phi(n)=-(G_0vf)(n)+c_1n+c_2$,
    where
    \begin{equation*}
    c_1=\frac{\left<Tf,v'\right>}{\|v'\|^2_{\ell^2}},\quad c_2=\frac{\left<Tf,v\right>}{\|V\|_{\ell^1}}-\frac{\left<v_1,v\right>}{\|V\|_{\ell^1}}c_1,\quad v'=Q(v_1)=v_1-\frac{\left<v_1,v\right>}{\|V\|_{\ell^1}}v.
    \end{equation*}
\item [{\rm(ii)}]$f\in S_2\ell^2(\Z)\Longleftrightarrow \exists\ \phi\in W_{1/2}(\Z)$ such that $H\phi=0$. Moreover, $f=Uv\phi$ and
    $$\phi=-G_0vf+\frac{\left<Tf,v\right>}{\|V\|_{\ell^1}}.$$
\item [{\rm(iii)}] $f\in S_3\ell^2(\Z)\Longleftrightarrow \exists\ \phi\in \ell^2(\Z)$ such that $H\phi=0$. Moreover, $f=Uv\phi$ and $\phi=-G_0vf$.
\item [{\rm(iv)}] $f\in \widetilde{S}_0\ell^2(\Z)\Longleftrightarrow \exists\ \phi\in W_{1/2}(\Z)$ such that $H\phi=16\phi$. Moreover, $f=Uv\phi$ and
    \begin{equation*}
    \phi=-J\widetilde{G}_{0}\tilde{v}f+J\frac{\big<\widetilde{T}f,\tilde{v}\big>}{\|V\|_{\ell^1}}.
    \end{equation*}
\item [{\rm(v)}] $f\in \widetilde{S}_1\ell^2(\Z)\Longleftrightarrow \exists\ \phi\in \ell^2(\Z)$ such that $H\phi=16\phi$. Moreover, $f=Uv\phi$ and $\phi=-J\widetilde{G}_{0}\tilde{v}f$.
\end{itemize}
\end{lemma}
This lemma indicates that
\begin{align*}
S_1\ell^2(\Z)&=\{0\}\Leftrightarrow H\phi=0\ {\rm has\  only\  zero\ solution\ in\ }W_{3/2}(\Z),\\
S_2\ell^2(\Z)&=\{0\}\Leftrightarrow H\phi=0\ {\rm has\  only\  zero\ solution\ in\ }W_{1/2}(\Z),\\
S_3\ell^2(\Z)&=\{0\}\Leftrightarrow H\phi=0\ {\rm has\  only\  zero\ solution\ in\ }\ell^2(\Z),\\
\widetilde{S}_0\ell^2(\Z)&=\{0\}\Leftrightarrow H\phi=16\phi\ {\rm has\  only\  zero\ solution\ in\ }W_{1/2}(\Z),\\
\widetilde{S}_1\ell^2(\Z)&=\{0\}\Leftrightarrow H\phi=16\phi\ {\rm has\  only\  zero\ solution\ in\ }\ell^2(\Z).
\end{align*}
}
\end{remark}

We now give the asymptotic expansions of $M^{-1}(\mu)$ as follows. Let
\begin{equation*}
P_1:=-\frac{2(1+i)}{\left\|V\right\|_{\ell^1}}P,\quad \widetilde{P}_1:=-\frac{32i}{\left\|V\right\|_{\ell^1}}\widetilde{P}.
\end{equation*}
We say that an integral operator $K\in\B(\ell^2(\Z))$ is absolutely bounded if its associated absolute value integral operator $|K|$, defined by the kernel $|K(n,m)|$, is also bounded on $\ell^2(\Z)$.

\begin{lemma}\label{Asymptotic expansion theorem}{\rm(\cite[Theorem 1.8]{HY25} )}
 Let $H=\Delta^2+V$ with $|V(n)|\lesssim \left<n\right>^{-\beta}$ for some $\beta>0$. 
Then we have the following asymptotic expansions on $\ell^2(\Z)$ for $0<\mu<\mu_0$:
\begin{itemize}
\item[{\rm(i)}] if $0$ is a regular point of $H$ and $\beta>15$, then
\begin{align}\label{asy expan of regular 0}
M^{-1}(\mu)&=S_0A_{0}S_0+\mu QA_{1}Q+\mu^2(QA^{0}_{21}Q+S_0A^{0}_{22}+A^{0}_{23}S_0)\notag\\
&\quad+\mu^3(QA^{0}_{31}+A^{0}_{32}Q)+\mu^3P_1+\Gamma^{0}_{4}(\mu),
\end{align}
\item [{\rm(ii)}] if $0$ is a first kind resonance of $H$ and $\beta>19$, then
\begin{align}\label{asy expan of 1st 0}
\begin{split}
M^{-1}(\mu)&=\mu^{-1}S_{1}A_{-1}S_{1}+\big(S_0A^{1}_{01}Q+QA^{1}_{02}S_0\big)+\mu\big(S_0A^{1}_{11}+A^{1}_{12}S_0+QA^{1}_{13}Q\big)\\
&\quad+\mu^2\big(QA^{1}_{21}+A^{1}_{22}Q\big)+\mu^3(QA^{1}_{31}+A^{1}_{32}Q)+\mu^3P_1+\mu^3A^{1}_{33}+\Gamma^{1}_{4}(\mu),
\end{split}
\end{align}
\vskip0.2cm
\item [{\rm(iii)}] if $0$ is a second kind resonance of $H$ and $\beta>27$, then
\begin{align}\label{asy expan 2nd 0}
M^{-1}(\mu)&=\frac{S_{2}A_{-3}S_{2}}{\mu^3}+\frac{S_2A_{-2,1}S_0+S_0A_{-2,2}S_2}{\mu^2}
+\frac{S_2A_{-1,1}Q+QA_{-1,2}S_2+S_0A_{-1,3}S_0}{\mu}\notag\\
 &\quad+\big(S_2A^{2}_{01}+A^{2}_{02}S_2+QA^{2}_{03}S_0+S_0A^{2}_{04}Q\big)+\mu\big(S_0A^{2}_{11}+A^{2}_{12}S_0+QA^{2}_{13}Q\big)\notag\\
&\quad+\mu^2\big(QA^{2}_{21}+A^{2}_{22}Q\big)+\mu^3(QA^{2}_{31}+A^{2}_{32}Q)+\mu^3P_1+\mu^3A^{2}_{33}+\Gamma^{2}_{4}(\mu),
\end{align}
\item[{\rm(iv)}] if $16$ is a regular point of $H$ and $\beta>9$, then
\begin{align}\label{asy expan regular 2}
M^{-1}(2-\mu)=\widetilde{Q}B_{0}\widetilde{Q}+
\mu^{\frac{1}{2}}(\widetilde{Q}B^{0}_{11}+B^{0}_{12}\widetilde{Q})+\mu^{\frac{1}{2}}\widetilde{P}_1+\mu B^{0}_{21}+\Gamma^{0}_{\frac{3}{2}}(\mu),
\end{align}
\item [{\rm(v)}] if $16$ is a resonance of $H$ and $\beta>13$, then
\begin{align}\label{asy expan resonance 2}
M^{-1}(2-\mu)&=\mu^{-\frac{1}{2}}\widetilde{S}_0B_{-1}\widetilde{S}_0+\big(\widetilde{S}_0B^{1}_{01}+B^{1}_{02}\widetilde{S}_0+\widetilde{Q}B^{1}_{03}\widetilde{Q}\big)+
\mu^{\frac{1}{2}}(\widetilde{Q}B^{1}_{11}+B^{1}_{12}\widetilde{Q})\notag\\
&\quad+\mu^{\frac{1}{2}}\widetilde{P}_1+\mu B^{1}_{21}+\Gamma^{1}_{\frac{3}{2}}(\mu),
\end{align}
\item [{\rm(vi)}] if $16$ is an eigenvalue of $H$ and $\beta>17$, then
\begin{align}\label{asy expan eigenvalue 2}
M^{-1}(2-\mu)&=\mu^{-1}\widetilde{S}_1B_{-2}\widetilde{S}_1+\mu^{-\frac{1}{2}}\big(\widetilde{S}_0B_{-1,1}\widetilde{Q}+\widetilde{Q}B_{-1,2}\widetilde{S}_0\big)+\big(\widetilde{Q}B^{2}_{01}+B^{2}_{02}\widetilde{Q}\big)\notag\\
&\quad +\mu^{\frac{1}{2}}(\widetilde{Q}B^{2}_{11}+B^{2}_{12}\widetilde{Q})+\mu^{\frac{1}{2}}\widetilde{P}_1 +\mu B^{2}_{21}+\Gamma^{2}_{\frac{3}{2}}(\mu),
\end{align}
\end{itemize}
where $A_0,A_1,A_{-1},A_{-3},A^{i}_{jk},A_{j,k},B_0,B_{-1},B_{-2},B^{i}_{jk},B_{j,k}$ are $\mu$-independent bounded operators on $\ell^2(\Z)$ and $\Gamma^{i}_\ell(\mu)$ are $\mu$-dependent bounded operators on $\ell^2(\Z)$ such that all the operators appeared in the right hand sides of \eqref{asy expan of regular 0}$\sim$\eqref{asy expan eigenvalue 2} are absolutely bounded. Moreover, $\Gamma^{i}_\ell(\mu)$ satisfies the following estimates:
   \begin{equation}\label{estimate of Gamma}
\|\Gamma^{i}_{\ell}(\mu)\|_{\ell^2\rightarrow\ell^2}+\mu\big\|\partial_{\mu}(\Gamma^{i}_{\ell}(\mu))\big\|_{\ell^2\rightarrow\ell^2}\lesssim\mu^{\ell}.
\end{equation}
\end{lemma}
\begin{remark}
{\rm We note that in \cite[Theorem 1.8]{HY25}, the precise information of the $\mu^3$ term in $M^{-1}(\mu)$ near $\mu = 0$ is not required. However, for our analysis of $\ell^p$ boundedness, this detailed information becomes essential and can be extracted from the proof given in \cite[Section 5]{HY25}. Furthermore, we require more terms in the expansion of $M^{-1}(\mu)$ around $\mu=2$, which can also be obtained by following the analogous arguments in \cite[Section 5]{HY25}.
}
\end{remark}

\section{The  low energy part $\mcaK_1$}\label{sect of mcaK1}
This section aims to establish the $\ell^p$ boundedness of the low energy part $\mcaK_1$. Namely, 
\begin{theorem}\label{theorem of W1}
Let $H=\Delta^2+V$ with $|V(n)|\lesssim \left<n\right>^{-\beta}$ for some $\beta>0$. Suppose that $H$ has no positive eigenvalues in the interval $\rm{(}0,16\rm{)}$. If
\begin{align*}
 \beta>\left\{\begin{aligned}&15,\ 0\ is\ a\ regular\ point\ of\ H,\\
&19,\ 0\ is\ a\ first\ kind\ resonance\ of\ H,\\
&27,\ 0\ is\ a\ second\ kind\ resonance\ of\ H,\end{aligned}\right.
\end{align*}
then $\mcaK_1\in\B(\ell^p(\Z))$ for all $1<p<\infty$.
\end{theorem}
Before proceeding the proof, we present a crucial lemma, which plays a key role in eliminating the singularity of $\mcaK_1$ near $\mu=0$, as detailed below.
\begin{lemma}\label{cancelation lemma}{\rm (\cite[Lemma 4.2]{HY25})}
Let $Q,S_j~(j=0,1,2)$ be the operators defined in Definition \ref{definition at zero}. Then for any $f\in\ell^2(\Z)$, the following statements hold.
\vskip0.2cm
\noindent{\rm(1)}~$(R^{\pm}_0(\mu^4)vQf)(n)=\frac{1}{4\mu^3}\sum\limits_{m\in\Z}\int_{0}^{1}({\rm sign}(n-\rho m))\big({\bm {b_1(\mu)}}e^{\mp i\theta_{+}|n-\rho m|}+{\bm {b_2(\mu)}}e^{b(\mu)|n-\rho m|}\big)d\rho$
\vskip0.2cm
    \qquad \qquad \qquad \qquad\quad $\times v_1(m)(Qf)(m)$,
    \vskip0.2cm
    \qquad \qquad \qquad \quad\ \ $:=\frac{1}{4\mu^3}\sum\limits_{m\in\Z}\mcaB^{\pm}(\mu,n,m)(Qf)(m)$,
    \vskip0.25cm
\noindent{\rm(2)}~$(R^{\pm}_0(\mu^4)vS_jf)(n)=\frac{1}{4\mu^3}\sum\limits_{m\in\Z}\Big[\int_{0}^{1}(1-\rho)\big({\bm{c^{\pm}_1(\mu)}}e^{\mp i\theta_{+}|n-\rho m|}+{\bm{c_2(\mu)}}e^{b(\mu)|n-\rho m|}\big)d\rho\cdot v_2(m)$
\vskip0.2cm
    \qquad \qquad \qquad \qquad\quad $+{\bm{c_3(\mu)}}|n-m|v(m)\Big](S_jf)(m),$
    \vskip0.2cm
    \qquad \qquad \qquad \quad\ \ $:=\frac{1}{4\mu^3}\sum\limits_{m\in\Z}\mcaC^{\pm}(\mu,n,m)(S_jf)(m),$
    \vskip0.25cm
\noindent{\rm(3)}~$(R^{\pm}_0(\mu^4)vS_2f)(n)=\frac{1}{8\mu^3}\sum\limits_{m\in\Z}\Big[\int_{0}^{1}(1-\rho)^2({\rm sign}(n-\rho m))^3\big({\bm{d_1(\mu)}}e^{\mp i\theta_{+}|n-\rho m|}+{\bm{d_2(\mu)}}e^{b(\mu)|n-\rho m|}\big)d\rho$
\vskip0.2cm
\qquad \qquad \qquad \qquad\quad$\times v_3(m)+{\bm{d_3(\mu)}}|n-m|v(m)\Big](S_2f)(m)$,
\vskip0.2cm
    \qquad \qquad \qquad \quad\ \ $:=\frac{1}{8\mu^3}\sum\limits_{m\in\Z}\mcaD^{\pm}(\mu,n,m)(S_2f)(m)$,
    \vskip0.25cm
\noindent{\rm(4)}~$Q(vR^{\pm}_0(\mu^4)f\big)=Qf^{\pm},\quad S_j\big(vR^{\pm}_0(\mu^4)f\big)=S_jf^{\pm}_j,$ where $j=0,1,2$ and
\vskip0.25cm
\begin{itemize}
\item $a_1(\mu)=\frac{1}{\sqrt{1-\frac{\mu^2}{4}}}$,\quad$a_2(\mu)=\frac{-1}{\sqrt{1+\frac{\mu^2}{4}}}$,\quad$b(\mu)={\rm ln} \big(1+\frac{\mu^2}{2}-\mu(1+\frac{\mu^2}{4})^{\frac{1}{2}}\big)$,
    \vskip0.2cm
\item $b_1(\mu)=-\theta_{+}a_1(\mu), \quad b_2(\mu)=-b(\mu)a_2(\mu)$,
\vskip0.2cm
\item $c^{\pm}_1(\mu)=\mp i\theta^2_{+}a_1(\mu), \quad c_2(\mu)=(b(\mu))^2a_2(\mu),\quad c_3(\mu)=\theta_{+}a_1(\mu)+b(\mu)a_2(\mu)$,
    \vskip0.2cm
\item $d_1(\mu)=\theta^3_{+}a_1(\mu), \quad d_2(\mu)=-(b(\mu))^3a_2(\mu),\quad d_3(\mu)=2c_3(\mu),$
\vskip0.2cm
\item $f^{\pm}(n)=\frac{1}{4\mu^3}\sum\limits_{m\in\Z}\mcaB^{\pm}(\mu,m,n)f(m),\quad f^{\pm}_j(n)=\frac{1}{4\mu^3}\sum\limits_{m\in\Z}\mcaC^{\pm}(\mu,m,n)f(m),\ j=0,1,$
\vskip0.2cm
\item $f^{\pm}_2(n)=\frac{1}{8\mu^3}\sum\limits_{m\in\Z}\mcaD^{\pm}(\mu,m,n)f(m)$.
\end{itemize}
\end{lemma}
\begin{remark}\label{remark of cancelation lemma}
{\rm(1) Noting that $\theta_{+}$, $b(\mu)$ and $c_3(\mu)$ exhibit the following behaviors, respectively:
    $$\theta_{+}=-\mu+o(\mu),\quad b(\mu)=-\mu+o(\mu),\quad c_3(\mu)=-\frac{1}{3}\mu^3-\frac{1}{8}\mu^4+O(\mu^5),\quad \mu\rightarrow0^+.$$
This indicates that, compared to the free resolvent $R^{\pm}_0(\mu^4)=O(\mu^{-3})$~(here $O(\mu^{-3})$ refers to the order of the kernel $R^{\pm}_{0}(\mu^4,n,m)$ with respect to $\mu$ and the same convention applies to the following operators unless otherwise specified), the operators considered in this lemma can decrease the singularity near $\mu=0$. Precisely, we have
\begin{align}\label{order of operators in cancel lemma 0}
\begin{split}
R^{\pm}_0(\mu^4)vQ&=O(\mu^{-2}),\quad R^{\pm}_0(\mu^4)vS_j=O(\mu^{-1})~(j=0,1),\quad  R^{\pm}_0(\mu^4)vS_2=O(1);\\
QvR^{\pm}_0(\mu^4)&=O(\mu^{-2}),\quad S_jvR^{\pm}_0(\mu^4)=O(\mu^{-1})~(j=0,1),\quad  S_2vR^{\pm}_0(\mu^4)=O(1).
\end{split}
\end{align}
\vskip0.15cm
\noindent(2)~However, the form in our discrete setting is far more intricate than its continuous counterpart \cite[Lemma 2.5]{MWY24}. Specifically, compared \eqref{kernel of R0 boundary} with the kernel on the line:
$$R^{\pm}_0(\mu^4,x,y)=\frac{1}{4\mu^3}\big(\pm ie^{\pm i\mu|x-y|}-e^{-\mu|x-y|}\big),\quad x,y\in \R,$$
  we observe that the continuous analogue of $(\theta_+,b(\mu),a_1(\mu),a_2(\mu))$ is $(-\mu,-\mu,1,-1)$. This means that the corresponding ${b}_j(\mu)$, ${c}^{\pm}_1(\mu)$, ${c}_j(\mu)$, ${d}_j(\mu)$ in the continuous case are the polynomials of $\mu$. In particular, $c_3(\mu)$ vanishes identically. We remark that such discrepancy will introduce some additional technical challenges in establishing the $\ell^p$ boundedness of $W_+$ in our discrete setting.}
\end{remark}
\subsection{$0$ is a regular point of $H$}\label{subsec of W1 regular}
In this subsection, we prove the $\ell^p$ boundedness for $\mcaK_1$ when $0$ is a regular point of $H$. First recall from \eqref{mcaKj} that
\begin{equation}\label{kernel of mcaK1}
\mcaK_1=\int_{0}^{2}\mu^3\chi_1(\mu)\big[R^{+}_{0}(\mu^4)v M^{-1}(\mu)v(R^{+}_{0}-R^{-}_0)(\mu^4)\big]d\mu,
\end{equation}
and the expansion \eqref{asy expan of regular 0} of $M^{-1}(\mu)$:
\begin{align*}
\begin{split}
M^{-1}(\mu)=S_0A_0S_0+\mu QA_1Q+\mu^2(QA^{0}_{21}Q+S_0A^{0}_{22}+A^{0}_{23}S_0)+\mu^3(QA^{0}_{31}+A^{0}_{32}Q)+\mu^3P_1+\Gamma^{0}_{4}(\mu),
\end{split}
\end{align*}
then $\mcaK_1$ can be written as a finite sum of the following integral operators:
\begin{equation}\label{decom1 of W1 regular}
\mcaK_1=\sum\limits_{A\in\mcaA_0}^{}K_A+K_{1}+K_{P_1}+K^0_4,
\end{equation}
where $\mcaA_0=\{S_0A_0S_0,\mu^2QA^{0}_{21}Q,\mu^2S_0A^{0}_{22},\mu^2A^{0}_{23}S_0,\mu^3QA^{0}_{31},\mu^3A^{0}_{32}Q\}$ and
\begin{align}
K_A(n,m)&=\int_{0}^{2}\mu^3\chi_1(\mu)\big[R^{+}_{0}(\mu^4)vAv(R^{+}_{0}-R^{-}_0)(\mu^4)\big](n,m)d\mu,\quad A\in\mcaA_0,\notag\\
K_1(n,m)& =\int_{0}^{2}\mu^4\chi_1(\mu)\big[R^{+}_{0}(\mu^4)vQA_1Qv(R^{+}_{0}-R^{-}_0)(\mu^4)\big](n,m)d\mu,\label{kernel of K110}\\
K_{P_1}(n,m)&=\int_{0}^{2}\mu^6\chi_1(\mu)\big[R^{+}_{0}(\mu^4)v P_1v(R^{+}_{0}-R^{-}_0)(\mu^4)\big](n,m)d\mu,\label{kernel of KP1}\\
K^{0}_4(n,m)&=\int_{0}^{2}\mu^3\chi_1(\mu)\big[R^{+}_{0}(\mu^4)v \Gamma^0_4(\mu)v(R^{+}_{0}-R^{-}_0)(\mu^4)\big](n,m)d\mu.\label{kernel of K40}
\end{align}
 Based on \eqref{order of operators in cancel lemma 0}, we can classify the operators in \eqref{decom1 of W1 regular} into the following two groups according to the order of their kernels with respect to $\mu$ as $\mu\rightarrow0^{+}$:
\begin{align*}
O(1):K~(K\in\{K_1,K_{P_1}\}),\quad O(\mu):K~(K\in\{K^0_{4}\}\cup\{K_{A}:A\in\mcaA_0\}).
\end{align*}
  The $\ell^p$ boundedness of $\mcaK_1$ consequently reduces to proving the boundedness of these two operator classes. We will establish this through three propositions.
\vskip0.2cm
To begin with, we deal with the operators in the class $O(\mu)$. 
Prior to this, we give the following Schur test lemma, which will often be used to establish the $\ell^p$-boundedness of integral operators.
\begin{lemma}\label{shur test}
If the kernel $K(n,m)$ satisfies $$\sup_{n\in\Z}\sum\limits_{m\in\Z}|K(n,m)|+\sup_{m\in\Z}\sum\limits_{n\in\Z}|K(n,m)|<\infty,$$ then ${K}\in \B(\ell^{p}(\Z))$ for all $1\leq p\leq \infty.$
\end{lemma}

In particular, as a sufficient condition of this lemma $|K(n,m)|\lesssim \left<|n|-|m|\right>^{-\gamma}$ with $\gamma>1$ will be used often in the proof.
\begin{proposition}\label{proposition good1 regular}
Let $H=\Delta^2+V$ with $|V(n)|\lesssim \left<n\right>^{-\beta}$ for $\beta>15$. Suppose that $H$ has no positive eigenvalues in the interval $\rm{(}0,16\rm{)}$ and $0$ is a regular point of $H$. Let $\mcaA_0$ be defined in \eqref{decom1 of W1 regular}, then for any $A\in\mcaA_0$, $K_A\in\B(\ell^p(\Z))$ for all $1\leq p\leq\infty$.
\end{proposition}
\begin{proof}
\underline{{\bf(1)}}~For $A=S_0A_0S_0$, denote
\begin{align*}
K_{0}(n,m)=\int_{0}^{2}\mu^3\chi_1(\mu)\big[R^{+}_{0}(\mu^4)vS_0A_0S_0v(R^{+}_{0}-R^{-}_0)(\mu^4)\big](n,m)d\mu.
\end{align*}
By virtue of Lemma \ref{cancelation lemma}, it can be further expressed as
\begin{align}\label{kernel of S0A01S0}
K_{0}(n,m)=\frac{1}{16}\sum\limits_{j=1}^{3}(K^{+,j}_{0}+K^{-,j}_{0})(n,m),
\end{align}
where $N_1=n-\rho_1m_1$, $M_2=m-\rho_2m_2$ and
\begin{align*}
K^{\pm,1}_{0}(n,m)&=\int_{0}^{2}\mu^{-3}(c^{+}_1(\mu))^2\chi_1(\mu)\sum\limits_{m_1,m_2\in\Z}\int_{[0,1]^2}^{}(1-\rho_1)(1-\rho_2)e^{-i\theta_+(|N_1|\pm|M_2|)}d\rho_1d\rho_2\notag\\
&\quad\times(v_2S_0A_0S_0v_2)(m_1,m_2)d\mu,\\
K^{\pm,2}_{0}(n,m)&=\int_{0}^{2}\mu^{-3}c^{+}_1(\mu)c_2(\mu)\chi_1(\mu)\sum\limits_{m_1,m_2\in\Z}\int_{[0,1]^2}^{}(1-\rho_1)(1-\rho_2)e^{b(\mu)|N_1|\pm i\theta_+|M_2|}d\rho_1d\rho_2\notag\\
&\quad\times(v_2S_0A_0S_0v_2)(m_1,m_2)d\mu,\\
K^{\pm,3}_{0}(n,m)&=\int_{0}^{2}\mu^{-3}c^{+}_1(\mu)c_3(\mu)\chi_1(\mu)\sum\limits_{m_1,m_2\in\Z}^{}\int_{0}^{1}(1-\rho_2)e^{\pm i\theta_+|M_2|}d\rho_2\cdot|n-m_1|\\
&\quad\times(vS_0A_0S_0v_2)(m_1,m_2)d\mu.
\end{align*}
Next we establish the following estimates:
\begin{equation}\label{estimate of Kj S0A01S0}
|K^{\pm,j}_{0}(n,m)|\lesssim \left<|n|\pm|m|\right>^{-2},\quad j=1,2,3,
\end{equation}
which combined with Lemma \ref{shur test} and \eqref{kernel of S0A01S0}, yield that $K_{0}\in\B(\ell^p(\Z))$ for any $1\leq p\leq\infty$.
\vskip0.3cm
{\underline{$\bm {Case\ j=1.}$}} Decomposing
\begin{equation}\label{decompostion of eitheta}
e^{-i\theta_+(|N_1|\pm|M_2|)}=e^{-i\theta_+(|n|\pm|m|)}e^{-i\theta_+\big(|N_1|-|n|\pm(|M_2|-|m|)\big)}
\end{equation}
and employing the following variable substitution:
\begin{equation}\label{varible substi1}
 {\rm cos}\theta_{+}=1-\frac{\mu^2}{2} \Longrightarrow\  \frac{d\mu}{d\theta_+}=\frac{{\rm sin}\theta_+}{\mu},\quad \theta_+\rightarrow0\ {\rm as}\ \mu\rightarrow0\ {\rm and}\ \theta_+\rightarrow-\pi\ {\rm as}\ \mu\rightarrow2,
\end{equation}
we can rewrite $K^{\pm,1}_{0}(n,m)$ as
\begin{align}\label{kernel of K01pm1}
K^{\pm,1}_{0}(n,m)&=\int_{0}^{2}e^{-i\theta_+(|n|\pm|m|)}\mu^{-3}\theta^4_+\chi_{11}(\mu)L^{\pm,1}_{0}(\theta_+,n,m)d\mu\notag\\
&=\int_{-\pi}^{0}e^{-i\theta_+(|n|\pm|m|)}g(\theta_+)\chi_{11}(\mu(\theta_+))L^{\pm,1}_{0}(\theta_+,n,m)d\theta_+\\
&:=\int_{-\pi}^{0}e^{-i\theta_+(|n|\pm|m|)}G^{\pm,1}_{0}(\theta_+,n,m)d\theta_+\notag,
\end{align}
where $\chi_{11}(\mu)=-\chi_1(\mu)(1-\frac{\mu^2}{4})^{-1}$, $g(\theta_+)=-\Big(\frac{\theta^2_+}{2(1-{\rm cos}\theta_+)}\Big)^2{\rm sin}\theta_+$ and
\begin{align*}
L^{\pm,1}_{0}(\theta_+,n,m)&=\sum\limits_{m_1,m_2\in\Z}\int_{[0,1]^2}^{}(1-\rho_1)(1-\rho_2)e^{-i\theta_+\big(|N_1|-|n|\pm(|M_2|-|m|)\big)}d\rho_1d\rho_2\\
&\quad\times(v_2S_0A_0S_0v_2)(m_1,m_2).
\end{align*}
First, for each $k=0,1,2$, we have the following estimate:
\begin{equation}\label{estimate of Lpm11}
\sup\limits_{\theta_{+}\in(-\pi,0)}\big|(\partial^k_{\theta_+}L^{\pm,1}_{0})(\theta_+,n,m)\big|\lesssim \|\left<\cdot\right>^{2k+4}V(\cdot)\|_{\ell^1},\quad {\rm uniformly\ in}\ n,m\in\Z. 
\end{equation}
This estimate combined with the facts that $supp\chi_1(\mu)\subseteq[0,\mu_0]$, $\lim\limits_{\theta_+\rightarrow0}g(\theta_+)=0$ immediately yields
\begin{equation}\label{uniform boundedness of Kpm11}
|K^{\pm,1}_{0}(n,m)|\lesssim 1,\quad {\rm uniformly\ in}\ n,m\in\Z.
\end{equation}
Moreover, applying integration by parts twice to $K^{\pm,1}_{0}(n,m)$ with $||n|\pm|m||\geq1$, we obtain
\begin{align*}
&K^{\pm,1}_{0}(n,m)=\Big(\frac{e^{-i\theta_+(|n|\pm|m|)}}{-i(|n|\pm|m|)}G^{\pm,1}_{0}(\theta_+,n,m)\Big)\Big|^0_{\theta_+=-\pi}-\int_{-\pi}^{0}\frac{e^{-i\theta_+(|n|\pm|m|)}}{-i(|n|\pm|m|)}(\partial_{\theta_+}G^{\pm,1}_{0})(\theta_+,n,m)d\theta_+\\
&=\frac{1}{i(|n|\pm|m|)}\int_{-\pi}^{0}e^{-i\theta_+(|n|\pm|m|)}(\partial_{\theta_+}G^{\pm,1}_{0})(\theta_+,n,m)d\theta_+\\
&=\frac{1}{(|n|\pm|m|)^2}\Big(\lim\limits_{\theta_+\rightarrow0}e^{-i\theta_+(|n|\pm|m|)}(\partial_{\theta_+}G^{\pm,1}_{0})(\theta_+,n,m)-\int_{-\pi}^{0}e^{-i\theta_+(|n|\pm|m|)}(\partial^2_{\theta_+}G^{\pm,1}_{0})(\theta_+,n,m)d\theta_+\Big)\\
&=O((|n|\pm|m|)^{-2}),
\end{align*}
where the second equality follows from the support condition of $\chi_1(\mu)$ and $\lim\limits_{\theta_+\rightarrow0}g(\theta_+)=0$. The fourth equality is obtained by combining the support of $\chi_1(\mu)$, the estimate \eqref{estimate of Lpm11} and the existence of limits $\lim\limits_{\theta_+\rightarrow0}g^{(k)}(\theta_+)$ for $k=1,2$. Therefore, for any $n,m\in\Z$, one has
\begin{align*}
|K^{\pm,1}_{0}(n,m)|\lesssim \left<|n|\pm|m|\right>^{-2}.
\end{align*}
\vskip0.3cm
{\underline{$\bm {Case\ j=2.}$}} We consider the decomposition
\begin{equation}\label{decompostion of ebmu}
e^{b(\mu)|N_1|\pm i\theta_+|M_2|}=e^{b(\mu)|n|\pm i\theta_+|m|}e^{b(\mu)(|N_1|-|n|)\pm i\theta_+(|M_2|-|m|)},
\end{equation}
then $K^{\pm,2}_{0}(n,m)$ can be expressed as
\begin{align}\label{kernel of K01pm2}
K^{\pm,2}_{0}(n,m)&=\int_{0}^{2}e^{b(\mu)|n|\pm i\theta_+|m|}\mu^{-3}(b(\mu))^2\theta^2_+\chi_{12}(\mu)L^{\pm,2}_{0}(\mu,n,m)d\mu\notag\\
&:=\int_{0}^{2}e^{b(\mu)|n|\pm i\theta_+|m|}G^{\pm,2}_{0}(\mu,n,m)d\mu,
\end{align}
where $\chi_{12}(\mu)=i\chi_1(\mu)(1-\frac{\mu^4}{16})^{-\frac{1}{2}}$ and
\begin{align*}
L^{\pm,2}_{0}(\mu,n,m)&=\sum\limits_{m_1,m_2\in\Z}\int_{[0,1]^2}^{}(1-\rho_1)(1-\rho_2)e^{b(\mu)(|N_1|-|n|)\pm i\theta_+(|M_2|-|m|)}d\rho_1d\rho_2\notag\\
&\quad\times(v_2S_0A_0S_0v_2)(m_1,m_2).
\end{align*}
Noting that $b(\mu)<0$, $\theta'_{+}(\mu)=-(1-\frac{\mu^2}{4})^{-\frac{1}{2}}$ and
$$b'(\mu)=-(2+\mu^2)^{-1}\big((4+\mu^2)^{\frac{1}{2}}+\mu^2(4+\mu^2)^{-\frac{1}{2}}\big)<0,\quad\mu\in(0,2),$$
we can verify that for any $k=0,1,2$,
\begin{align}\label{estimate of Lpm12}
\sup\limits_{\mu\in(0,\mu_0]}\big|e^{b(\mu)|n|}(\partial^k_\mu L^{\pm,2}_{0})(\mu,n,m)\big|\lesssim \|\left<\cdot\right>^{2k+4}V(\cdot)\|_{\ell^1},\quad {\rm uniformly\ in}\ n,m\in\Z.
\end{align}
This immediately yields that $K^{\pm,2}_{0}(n,m)$ is uniformly bounded on $\Z^2$ by combining $supp\chi_1(\mu)\subseteq[0,\mu_0]$ and the existence of the limits
 $\lim\limits_{\mu\rightarrow0^+}\frac{b(\mu)}{\mu}$ and $ \lim\limits_{\mu\rightarrow0^+}\frac{\theta_+}{\mu}.$ On the other hand, assuming that $||n|\pm|m||\geq1$, we apply integration by parts twice to $K^{\pm,2}_{0}(n,m)$ obtaining that
\begin{align*}
&K^{\pm,2}_{0}(n,m)=\Big(e^{b(\mu)|n|\pm i\theta_+|m|}\frac{G^{\pm,2}_{0}(\mu,n,m)}{\alpha^{\pm}(\mu,n,m)}\Big)\Big|^2_{\mu=0}-\int_{0}^{2}e^{b(\mu)|n|\pm i\theta_+|m|}\Big(\frac{G^{\pm,2}_{0}}{\alpha^{\pm}}\Big)'(\mu,n,m)d\mu\\
&\xlongequal[]{\widetilde{G}^{\pm,2}_{0}:=\Big(\frac{G^{\pm,2}_{0}}{\alpha^{\pm}}\Big)'}-\int_{0}^{2}e^{b(\mu)|n|\pm i\theta_+|m|}\widetilde{G}^{\pm,2}_{0}(\mu,n,m)d\mu\\
&=\lim\limits_{\mu\rightarrow0^+}e^{b(\mu)|n|\pm i\theta_+|m|}\Big(\frac{\widetilde{G}^{\pm,2}_{0}}{\alpha^{\pm}}\Big)(\mu,n,m)+\int_{0}^{2}e^{b(\mu)|n|\pm i\theta_+|m|}\Big(\frac{\widetilde{G}^{\pm,2}_{0}}{\alpha^{\pm}}\Big)'(\mu,n,m)d\mu\\
&=O((|n|\pm|m|)^{-2}),
 \end{align*}
where
\begin{equation}\label{defi of alphapm}
\alpha^{\pm}(\mu,n,m):=b'(\mu)|n|\pm i\theta'_+(\mu)|m|,
 \end{equation}
 and in the second equality we used the facts that $supp\chi_1(\mu)\subseteq[0,\mu_0]$, $\lim\limits_{\mu\rightarrow0^+}\frac{(b(\mu))^2\theta^2_+}{\mu^3}=0$ and \eqref{estimate of Lpm12}. To verify the fourth equality, first, we compute
 \begin{align*}
\Big(\frac{\widetilde{G}^{\pm,2}_{0}}{\alpha^{\pm}}\Big)(\mu,n,m)&=\frac{1}{(\alpha^{\pm})^2}\Big(\frac{-(\alpha^{\pm})'}{\alpha^{\pm}}G^{\pm,2}_{0}+(G^{\pm,2}_{0})'\Big)(\mu,n,m),\\
\Big(\frac{\widetilde{G}^{\pm,2}_{0}}{\alpha^{\pm}}\Big)'(\mu,n,m)&=\frac{1}{(\alpha^{\pm})^2}\Big[\Big(-\frac{(\alpha^{\pm})^{(2)}}{\alpha^{\pm}}+3\Big(\frac{(\alpha^{\pm})'}{\alpha^{\pm}}\Big)^{2}\Big)G^{\pm,2}_{0}-3\frac{(\alpha^{\pm})'}{\alpha^{\pm}}(G^{\pm,2}_{0})'+(G^{\pm,2}_{0})^{(2)}\Big](\mu,n,m).
 \end{align*}
 Notice that \begin{align*}
 \frac{1}{|\alpha^{\pm}(\mu,n,m)|^2}\lesssim (|n|+|m|)^{-2},\quad {\rm uniformly\ in}\ (n,m)\neq(0,0)\ {\rm and}\ \mu\in(0,2),
 \end{align*}
 and for any $k=1,2$,
$$\lim\limits_{\mu\rightarrow0^+}\Big(\frac{b(\mu)}{\mu}\Big)^{(k)}\ {\rm and}\  \lim\limits_{\mu\rightarrow0^+}\Big(\frac{\theta_+}{\mu}\Big)^{(k)}\ {\rm exist} ,$$
$$\Big|\frac{(\partial^{k}_\mu\alpha^{\pm})(\mu,n,m)}{\alpha^{\pm}(\mu,n,m)}\Big|\lesssim 1,\quad {\rm uniformly\ in}\ (n,m)\neq(0,0)\ {\rm and}\ \mu\in(0,\mu_0].$$
These facts together with \eqref{estimate of Lpm12} establish the fourth equality, which combined with the uniform boundedness of $K^{\pm,2}_{0}(n,m)$ gives
\begin{align*}
|K^{\pm,2}_{0}(n,m)|\lesssim \left<|n|\pm|m|\right>^{-2},\quad \forall\ n,m\in\Z.
\end{align*}
\vskip0.3cm
{\underline{$\bm {Case\ j=3.}$}} Considering
$$e^{\pm i\theta_+|M_2|}=e^{\pm i\theta_+(|n|+|m|)}e^{\pm i\theta_+(|M_2|-(|n|+|m|))},$$
which allows us to rewrite $K^{\pm,3}_{0}(n,m)$ as
\begin{align}\label{kernel of K01pm3}
K^{\pm,3}_{0}(n,m)&=\int_{0}^{2}e^{\pm i\theta_+(|n|+|m|)}\mu^{-3}c^{+}_1(\mu)c_3(\mu){\chi}_1(\mu)L^{\pm,3}_{0}(\theta_+,n,m)d\mu\notag\\
&\xlongequal[]{{\rm by}\ \eqref{varible substi1}}\int_{-\pi}^{0}e^{\pm i\theta_+(|n|+|m|)}{\chi}_{13}(\mu(\theta_+))\frac{\theta^2_{+}}{2(1-{\rm cos}\theta_+)}L^{\pm,3}_{0}(\theta_+,n,m)d\theta_+\notag\\
&:=\int_{-\pi}^{0}e^{\pm i\theta_+(|n|+|m|)}G^{\pm,3}_{0}(\theta_+,n,m)d\theta_+,
\end{align}
where ${\chi}_{13}(\mu)=-i\chi_1(\mu)\frac{c_3(\mu)}{\mu}$ and
\begin{align*}
L^{\pm,3}_{0}(\theta_+,n,m)=\sum\limits_{m_1,m_2\in\Z}^{}\int_{0}^{1}(1-\rho_2)e^{\pm i\theta_+(|M_2|-|m|-|n|)}d\rho_2\cdot|n-m_1|\big(vS_0A_0S_0v_2)(m_1,m_2).
\end{align*}
In view that
$$c_3(\mu)=-\frac{1}{3}\mu^3-\frac{1}{8}\mu^4+O(\mu^5),\quad\mu\rightarrow0^+,$$
and for any $k=0,1,2$,
$$\lim\limits_{\mu\rightarrow0^+}\Big(\frac{c_3(\mu)}{\mu}\Big)^{(k)}\quad {\rm and}\ \lim\limits_{\theta_+\rightarrow0}\Big(\frac{\theta^2_{+}}{2(1-{\rm cos}\theta_+)}\Big)^{(k)}\ {\rm exist}.$$
According to the argument for $K^{\pm,1}_{0}$, it suffices to establish the following estimate for any $k=0,1,2$:
\begin{equation}\label{estimate of L2}
\sup\limits_{\theta_{+}\in(-\pi,0)}\big|(\partial^k_{\theta_+}L^{\pm,3}_{0})(\theta_+,n,m)\big|\lesssim 1,\quad {\rm uniformly\ in}\ n,m\in\Z. 
\end{equation}
To see this, {\bf{for ${\bm{k=0}}$}}, using the orthogonality $\left<S_0f,v\right>=0$, we have
$$L^{\pm,3}_{0}(\theta_+,n,m)=\sum\limits_{m_1\in\Z}^{}e^{\mp i\theta_+|n|}(|n-m_1|-|n|)v(m_1)(S_0A_0S_0(h^{\pm}(\theta_+,m,\cdot)))(m_1)$$
with
$$h^{\pm}(\theta_+,m,m_2)=v_2(m_2)\int_{0}^{1}(1-\rho_2)e^{\pm i\theta_+(|M_2|-|m|)}d\rho_2.$$
By the triangle inequality and H\"{o}lder's inequality, we obtain
\begin{equation}
\sup\limits_{\theta_{+}\in(-\pi,0)}|L^{\pm,3}_{0}(\theta_+,n,m)|\lesssim \|\left<\cdot\right>^4V(\cdot)\|_{\ell^1},\quad {\rm uniformly\ in}\ n,m\in\Z.
\end{equation}
{\bf{For ${\bm{k=1}}$}}, it is crucial to show that
$$\tilde{L}^{\pm,3}_{0}(\theta_+,n,m):=\sum\limits_{m_1\in\Z}e^{\mp i\theta_+|n|}|n|\cdot|n-m_1|v(m_1)(S_0A_0S_0(h^{\pm}(\theta_+,m,\cdot)))(m_1)$$ is uniformly bounded in $n,m,\theta_+$. Using $\left<S_0f,v\right>=\left<S_0f,v_1\right>=0$, we rewrite it as
\begin{align*}
\tilde{L}^{\pm,3}_{0}(\theta_+,n,m)=\sum\limits_{m_1\in\Z}e^{\mp i\theta_+|n|}\underbrace{(|n|\cdot|n-m_1|-n^2+nm_1)}_{J_1(n,m_1)}v(m_1)(S_0A_0S_0(h^{\pm}(\theta_+,m,\cdot)))(m_1),
\end{align*}
which together with the following fact:
\begin{align*}
|J_1(n,m_1)|=\big||n|(|n-m_1|-|n|)+nm_1\big|=\Big|\frac{|n|m^2_1+nm_1(|n-m_1|-|n|)}{|n-m_1|+|n|}\Big|\lesssim \left<m_1\right>^2.
\end{align*}
gives the desired uniform boundedness.  {\bf{For ${\bm{k=2}}$}}, it is key to verify the uniform boundedness of
\begin{align*}
\tilde{\tilde{L}}^{\pm,3}_{0}(\theta_+,n,m)&:=\sum\limits_{m_1\in\Z}e^{\mp i\theta_+|n|}|n|^2\cdot|n-m_1|v(m_1)(S_0A_0S_0(h^{\pm}(\theta_+,m,\cdot)))(m_1)\\
&=\sum\limits_{m_1\in\Z}e^{\mp i\theta_+|n|}\underbrace{(|n|^2\cdot|n-m_1|-n^2|n|+|n|nm_1)}_{J_2(n,m_1)}v(m_1)(S_0A_0S_0(h^{\pm}(\theta_+,m,\cdot)))(m_1).
\end{align*}
This can be obtained by the fact that
$$|J_2(n,m_1)|\lesssim \left<m_1\right>^3,$$
which can be verified through the following computation:
\begin{align*}
J_2(n,m_1)&=n^2(|n-m_1|-|n|)+|n|nm_1=\frac{n^2m^2_1+n|n|m_1(|n-m_1|-|n|)}{|n-m_1|+|n|}\\
&=\Big(\frac{n^2m^2_1}{|n-m_1|+|n|}-\frac{1}{2}|n|m^2_1\Big)+\Big(\frac{n|n|m_1(|n-m_1|-|n|)}{|n-m_1|+|n|}+\frac{1}{2}|n|m^2_1\Big)\\
&=\frac{|n|m^2_1(|n|-|n-m_1|)}{2(|n-m_1|+|n|)}+\frac{n^2m^2_1(|n-m_1|-|n|)+\frac{1}{2}|n|m^4_1}{(|n-m_1|+|n|)^2}.
\end{align*}
To sum up, the desired estimate \eqref{estimate of Kj S0A01S0} is obtained.
\vskip0.3cm
{\bf(2)}~For any $A\in\mcaA_0\setminus\{S_0A_{0}S_0\}$, denote
\begin{equation}\label{mcaKA munm}
\mcaK_A(\mu,n,m)=16\mu^{3}\big[R^{+}_{0}(\mu^4)vAv(R^{+}_{0}-R^{-}_0)(\mu^4)\big](n,m),
\end{equation}
then it follows from Lemma \ref{cancelation lemma} that
\begin{align}\label{mcaKA regular 0}
\mcaK_A(\mu,n,m)=
\left\{\begin{aligned}&\sum\limits_{m_1,m_2}\int_{[0,1]^2}\mcaM^0_{21}(N_1,M_2,m_1,m_2)(f_{21}^{+}-f_{21}^{-})(\mu,N_1,M_2)d\rho_1d\rho_2,\quad A=\mu^2 QA^{0}_{21}Q,\\
&\sum\limits_{m_1,m_2}\Big[\int_{0}^{1}\mcaM^0_{22}(\rho_1,m_1,m_2)\big(f^{+,1}_{22}+f^{-,1}_{22}\big)(\mu,N_1,\widetilde{M}_2)d\rho_1+\\
&\qquad\qquad\big(f^{+,2}_{22}+f^{-,2}_{22}\big)(\mu,\widetilde{M}_2,n,m_1,m_2)\Big],\qquad\qquad\quad A= \mu^2S_0A^{0}_{22},\\
&\sum\limits_{m_1,m_2}\int_{0}^{1}\mcaM^0_{23}(\rho_2,m_1,m_2)(f^+_{23}+f^-_{23})(\mu,\widetilde{N}_1,M_2)d\rho_2,\quad A=\mu^2 A^{0}_{23}S_0,\\
&\sum\limits_{m_1,m_2}\int_{0}^{1}\mcaM^0_{31}(N_1,m_1,m_2)(f^+_{31}+f^-_{31})(\mu,{N}_1,\widetilde{M}_2)d\rho_1,\quad A= \mu^3QA^{0}_{31},\\
&\sum\limits_{m_1,m_2}\int_{0}^{1}\mcaM^0_{32}(M_2,m_1,m_2)(f^+_{32}-f^-_{32})(\mu,\widetilde{{N}}_1,{M}_2)d\rho_2,\quad A= \mu^3A^{0}_{32}Q,\end{aligned}\right.
\end{align}
where $N_1=n-\rho_1m_1$, $\widetilde{N}_1=n-m_1$, $M_2=m-\rho_2m_2$, $\widetilde{M}_2=m-m_2$,
\begin{align}
\Phi^{\pm}_1(\mu,X,Y)&=e^{-i\theta_+(|X|\pm|Y|)},\quad \Phi^{\pm}_2(\mu,X,Y)=e^{b(\mu)|X|\pm i\theta_+|Y|},\notag\\
f^{\pm}_{21}(\mu,N_1,M_2)&=\mu^{-1}\theta^2_+a_{11}(\mu)\Phi^{\pm}_1(\mu,N_1,M_2)-\mu^{-1}\theta_+b(\mu)a_{12}(\mu)\Phi^{\pm}_2(\mu,N_1,M_2),\notag\\
f^{\pm,1}_{22}(\mu,N_1,\widetilde{M}_2)&=\mu^{-1}\theta^2_+a_{11}(\mu)\Phi^{\pm}_1(\mu,N_1,\widetilde{M}_2)+i\mu^{-1}(b(\mu))^2a_{12}(\mu)\Phi^{\pm}_2(\mu,N_1,\widetilde{M}_2),\notag\\
f^{\pm,2}_{22}(\mu,\widetilde{M}_2,n,m_1,m_2)&=i\mu^{-1}a_1(\mu)c_3(\mu)e^{\pm i\theta_+|\widetilde{M}_2|}|n-m_1|v(m_1)(S_0A^0_{22}v)(m_1,m_2),\notag\\
f^{\pm}_{23}(\mu,\widetilde{N}_1,M_2)&=\mu^{-1}\theta^2_+a_{11}(\mu)\Phi^{\pm}_1(\mu,\widetilde{N}_1,M_2)-i\mu^{-1}\theta^2_+a_{12}(\mu)\Phi^{\pm}_2(\mu,\widetilde{N}_1,M_2),\notag\\
f^{\pm}_{31}(\mu,N_1,\widetilde{M}_2)&=-i\theta_+a_{11}(\mu)\Phi^{\pm}_1(\mu,{N}_1,\widetilde{M}_2)-ib(\mu)a_{12}(\mu)\Phi^{\pm}_2(\mu,{N}_1,\widetilde{M}_2),\notag\\
f^{\pm}_{32}(\mu,\widetilde{N}_1,M_2)&=-i{\theta_+}a_{11}(\mu)\Phi^{\pm}_1(\mu,\widetilde{{N}}_1,{M}_2)+{\theta_+}a_{12}(\mu)\Phi^{\pm}_2(\mu,\widetilde{{N}}_1,{M}_2),\label{expre of fijpm}
\end{align}
with $a_{11}(\mu)=(a_1(\mu))^2$,
$a_{12}(\mu)=a_1(\mu)a_2(\mu)$, and
\begin{itemize}
\item $\mcaM^0_{21}(N_1,M_2,m_1,m_2)=({\rm sign}(N_1))({\rm sign}(M_2))(v_1QA^0_{21}Qv_1)(m_1,m_2)$, \vskip0.2cm
    \item $\mcaM^0_{22}(\rho_1,m_1,m_2)=(1-\rho_1)(v_2S_0A^0_{22}v)(m_1,m_2)$,
\vskip0.2cm
    \item $\mcaM^0_{23}(\rho_2,m_1,m_2)=(1-\rho_2)(vA^0_{23}S_0v_2)(m_1,m_2)$,
    \vskip0.2cm
    \item
$\mcaM^0_{31}(N_1,m_1,m_2)=({\rm sign}(N_1))(v_1QA^0_{31}v)(m_1,m_2)$,
\vskip0.2cm
    \item $\mcaM^0_{32}(M_2,m_1,m_2)=({\rm sign}(M_2))(vA^0_{32}Qv_1)(m_1,m_2)$.
\end{itemize}
From \eqref{expre of fijpm}, we note that for any operator $A\in\mcaA_0\setminus\{S_0A_0S_0\}$, the estimates of $K_A(n,m)$ can be reduced to the three fundamental cases presented in \eqref{kernel of K01pm1}, \eqref{kernel of K01pm2} and \eqref{kernel of K01pm3}. Using analogous arguments to those employed previously, we can obtain
$$|K_A(n,m)|\lesssim \left<|n|\pm|m|\right>^{-2},\quad A\in\mcaA_0\setminus\{S_0A_0S_0\},$$
which gives $K_A\in\B(\ell^p(\Z))$ for all $1\leq p\leq\infty$. This completes the whole proof.
\end{proof}
\begin{proposition}\label{proposition good2 regular}
Under the assumptions in Proposition \ref{proposition good1 regular}, let $K^0_4$ be the operator with kernel defined in \eqref{kernel of K40}, then $K^0_4\in\B(\ell^p(\Z))$ for any $1\leq p\leq\infty$.
\end{proposition}
\begin{proof} It follows from \eqref{kernel of K40} and \eqref{kernel of R0 boundary} that
\begin{align*}
K^0_4(n,m)&=\int_{0}^{2}\mu^3\chi_1(\mu)\big[R^{+}_{0}(\mu^4)v \Gamma^0_4(\mu)v(R^{+}_{0}-R^{-}_0)(\mu^4)\big](n,m)d\mu=\frac{1}{16}\sum\limits_{j=1}^{2}(K^{+}_{4j}+K^{-}_{4j})(n,m),
\end{align*}
where $N_1=n-m_1$, $M_2=m-m_2$, $\widetilde{\Gamma}^0_4(\mu)=\frac{\Gamma^0_4(\mu)}{\mu^4}$ and
\begin{align}
K^{\pm}_{41}(n,m)&=-\int_{0}^{2}e^{-i\theta_+(|n|\pm|m|)}\mu\underbrace{\chi_1(\mu)a_{11}(\mu)}_{\widetilde{\chi}_1(\mu)}\sum\limits_{m_1,m_2\in\Z}e^{-i\theta_+\big(|N_1|-|n|\pm(|M_2|-|m|)\big)}\notag\\
&\quad\times(v\widetilde{\Gamma}^0_4(\mu)v)(m_1,m_2)d\mu:=\int_{0}^{2}e^{-i\theta_+(|n|\pm|m|)}\mu\widetilde{\chi}_1(\mu)L^{\pm}_{41}(\mu,n,m)d\mu,\label{kernel of Kpm1 4}\\
K^{\pm}_{42}(n,m)&=i\int_{0}^{2}e^{b(\mu)|n|\pm i\theta_+|m|}\mu\underbrace{\chi_1(\mu)a_{12}(\mu)}_{\widetilde{\widetilde{\chi}}_1(\mu)}\sum\limits_{m_1,m_2\in\Z}e^{b(\mu)(|N_1|-|n|)\pm i\theta_+(|M_2|-|m|)}\notag\\
&\quad\times(v\widetilde{\Gamma}^0_4(\mu)v)(m_1,m_2)d\mu:=\int_{0}^{2}e^{b(\mu)|n|\pm i\theta_+|m|}\mu\widetilde{\widetilde{\chi}}_1(\mu)L^{\pm}_{42}(\mu,n,m)d\mu.\label{kernel of Kpm2 4}
\end{align}
We consider the following homogeneous dyadic partition of unity $\{{\varphi_{N}}\}_{N\in \Z}$ on $(0,\infty)$: $\varphi\in C^{\infty}_{0}(\R^{+})$, $0\leq\varphi\leq1$, $supp~\varphi\subset[\frac{1}{4},1],\ \varphi_{N}(\mu)=\varphi(2^{-N}\mu)$,
	$supp~\varphi_{N}\subset\left[2^{N-2},2^{N}\right]$,$$ \sum_{N\in\Z}{}\varphi_{N}(\mu)=1, \quad\mu>0.$$
Let $N_0=[log_2(4\mu_0)]$, then $\widetilde{\chi}_1(\mu)$ and $\widetilde{\widetilde{\chi}}_1(\mu)$ can be decomposed as follows:
\begin{align}
\widetilde{\chi}_1(\mu)=\sum\limits_{N\in\Z}\widetilde{\chi}_1(\mu)\varphi_{N}(\mu)=\sum\limits_{N=-\infty}^{N_0}\widetilde{\chi}_1(\mu)\varphi_{N}(\mu):=\sum\limits_{N=-\infty}^{N_0}\tilde{\phi}_N(\mu),\label{decomp of tutachi1}\\
\widetilde{\widetilde{\chi}}_1(\mu)=\sum\limits_{N\in\Z}\widetilde{\widetilde{\chi}}_1(\mu)\varphi_{N}(\mu)=\sum\limits_{N=-\infty}^{N_0}\widetilde{\widetilde{\chi}}_1(\mu)\varphi_{N}(\mu):=\sum\limits_{N=-\infty}^{N_0}\tilde{\tilde{\phi}}_N(\mu).\label{decomp of 2tutachi1}
\end{align}
It immediately concludes that for any $s\in\N$,
\begin{align}\label{estimate of phiN}
\big|\big(\tilde{\phi}_N\big)^{(s)}(\mu)\big|+\Big|\Big(\tilde{\tilde{\phi}}_N\Big)^{(s)}(\mu)\Big|\leq c(s)2^{-Ns},
\end{align}
where $c(s)$ is a constant depending on $s$. Taking \eqref{decomp of tutachi1} into \eqref{kernel of Kpm1 4} and \eqref{decomp of 2tutachi1} into \eqref{kernel of Kpm2 4}, we have
\begin{align}\label{kernel of K4jN}
\begin{split}
K^{\pm}_{41}(n,m)&=\sum\limits_{N=-\infty}^{N_0}\int_{0}^{2}e^{-i\theta_+(|n|\pm|m|)}\mu\tilde{\phi}_N(\mu)L^{\pm}_{41}(\mu,n,m)d\mu:=\sum\limits_{N=-\infty}^{N_0}K^{\pm,N}_{41}(n,m),\\
K^{\pm}_{42}(n,m)&=\sum\limits_{N=-\infty}^{N_0}\int_{0}^{2}e^{b(\mu)|n|\pm i\theta_+|m|}\mu\tilde{\tilde{\phi}}_N(\mu)L^{\pm}_{42}(\mu,n,m)d\mu:=\sum\limits_{N=-\infty}^{N_0}K^{\pm,N}_{42}(n,m).
\end{split}
\end{align}
Next we show that for any $N\leq N_0$, the following estimates hold:
\begin{align}\label{estimate of K4jN}
|K^{\pm,N}_{4j}(n,m)|\lesssim \min\{2^{2N},\left<|n|\pm|m|\right>^{-2}\},\quad\forall\ n,m\in\Z,\quad j=1,2,
\end{align}
from which
\begin{align*}
|K^{\pm,N}_{4j}(n,m)|\lesssim 2^{2N(1-t)}\left<|n|\pm|m|\right>^{-2t},\quad t\in[0,1],\quad\forall\ n,m\in\Z,\quad j=1,2.
\end{align*}
By choosing $t=\frac{3}{4}$, we obtain
\begin{align*}
|K^{\pm}_{4j}(n,m)|\leq\sum\limits_{N=-\infty}^{N_0}|K^{\pm,N}_{4j}(n,m)|\lesssim \left<|n|\pm|m|\right>^{-\frac{3}{2}}\sum\limits_{N=-\infty}^{N_0}2^{\frac{N}{2}}\lesssim \left<|n|\pm|m|\right>^{-\frac{3}{2}}, \quad j=1,2,
\end{align*}
which together with Lemma \ref{shur test} gives the desired result. To derive \eqref{estimate of K4jN}, we first note that for any $k=0,1,2$,
\begin{align}\label{estimate of Lpm4j}
\sup\limits_{\mu\in(0,\mu_0]}\mu^k\Big(\big|(\partial^k_\mu L^{\pm}_{41})(\mu,n,m)\big|+\big|e^{b(\mu)|n|}(\partial^k_\mu L^{\pm}_{42})(\mu,n,m)\big|\Big)\lesssim \|\left<\cdot\right>^{2k}V(\cdot)\|_{\ell^1},
\end{align}
uniformly in $n,m\in\Z$.

On one hand, combining \eqref{kernel of K4jN}, the support of $\varphi_N$ and \eqref{estimate of Lpm4j}, one has
\begin{align}\label{uniform bound of KpmN4j}
|K^{\pm,N}_{4j}(n,m)|\lesssim \int_{supp\varphi_{N}}\mu d\mu=\int_{2^{N-2}}^{2^N}\mu d\mu\lesssim 2^{2N}\lesssim2^{2N_0}\lesssim1,\quad j=1,2.
\end{align}

On the other hand, for any $N\leq N_0$, denote
$$G^{\pm,N}_{41}(\mu,n,m)=\mu\tilde{\phi}_N(\mu)L^{\pm}_{41}(\mu,n,m).$$
Assuming that $||n|\pm|m||\ge1$ and applying integration by parts twice to $K^{\pm,N}_{41}(n,m)$, we obtain
\begin{align*}
&K^{\pm,N}_{41}(n,m)=\Big(\frac{e^{-i\theta_+(|n|\pm|m|)}}{-i\theta'_+(\mu)(|n|\pm|m|)}G^{\pm,N}_{41}(\mu,n,m)\Big)\Big|^2_{0}-\int_{0}^{2}\frac{e^{-i\theta_+(|n|\pm|m|)}}{-i(|n|\pm|m|)}\big(\widetilde{G}^{\pm,N}_{41}\big)'(\mu,n,m)d\mu\\
&=\int_{0}^{2}\frac{e^{-i\theta_+(|n|\pm|m|)}}{i(|n|\pm|m|)}\big(\widetilde{G}^{\pm,N}_{41}\big)'(\mu,n,m)d\mu\\
&=-\frac{1}{(|n|\pm|m|)^2}\int_{0}^{2}e^{-i\theta_+(|n|\pm|m|)}\Big(g(\mu)\big(\widetilde{G}^{\pm,N}_{41}\big)'(\mu,n,m)\Big)'d\mu=O((|n|\pm|m|)^{-2}),
\end{align*}
where $$\theta'_+(\mu)=-(1-\frac{\mu^2}{4})^{\frac{-1}{2}},\quad \widetilde{G}^{\pm,N}_{41}(\mu,n,m)=(\theta'_+(\mu))^{-1}G^{\pm,N}_{41}(\mu,n,m):=g(\mu)G^{\pm,N}_{41}(\mu,n,m),$$ and in both second and third equalities we used the facts that $supp\chi_1\subseteq[0,\mu_0]$ and $\varphi(0)=0$. For the fourth equality, first we can compute that
\begin{align*}
\Big(g(\mu)\big(\widetilde{G}^{\pm,N}_{41}\big)'(\mu,n,m)\Big)'&=\Big[\Big(g(\mu)g^{(2)}(\mu)+(g'(\mu))^2\Big)G^{\pm,N}_{41}(\mu,n,m)+3g'(\mu)g(\mu)\big(G^{\pm,N}_{41}\big)'(\mu,n,m)\\
&\quad\quad+(g(\mu))^2\big(G^{\pm,N}_{41}\big)^{(2)}(\mu,n,m)\Big].
\end{align*}
For any $\mu\in(0,\mu_0]\cap[2^{N-2},2^N]$, it follows from \eqref{estimate of Lpm4j} and \eqref{estimate of phiN} that
\begin{align*}
\big|(\partial^{k}_{\mu}G^{\pm,N}_{41})(\mu,n,m)\big|\lesssim 2^{N(1-k)},\quad k=0,1,2,\quad{\rm uniformly\ in}\ n,m\in\Z,
\end{align*}
which together with the smoothness of $g(\mu)$ on $(0,\mu_0]$ gives the desired fourth equality. Through an analogous argument and using the properties of $\alpha^{\pm}(\mu,n,m)$ defined in \eqref{defi of alphapm}, we can verify that the same bound also holds for $K^{\pm,N}_{42}(n,m)$. Therefore, the desired \eqref{estimate of K4jN} is obtained.
\end{proof}
Next, we turn to establish the boundedness of operators in the class $O(1)$. For such integral operators, we shall need the following key lemma.

\begin{lemma}\label{C-Z lemma}
Let $\phi \in C^{\infty}(\R,\R)$ be such that $\phi(s)=0$ for $0\leq s\leq1$ and $\phi(s)=1$ for $s\geq2$. Define $k^{\pm}_{j}$ be the integral operator with the following kernel $k^{\pm}_j(n,m)$:
\begin{equation*}
k^{\pm}_1(n,m)=\frac{\phi(||n|\pm|m||^2)}{|n|\pm |m|},\quad k^{\pm}_{2}(n,m)=\frac{\phi(||n|-|m||^2)}{|n|\pm i|m|}
\end{equation*}
 then $k^{\pm}_1,k^{\pm}_2\in\B(\ell^p(\Z))$ for $1< p<\infty$.
\end{lemma}
The proof of this lemma will be postponed at the end of this section.
\begin{proposition}\label{proposition K110KP1 regular}
Under the assumptions in Proposition \ref{proposition good1 regular}, let $K_{1}$ and $K_{P_1}$ be the operators with kernels defined in \eqref{kernel of K110} and \eqref{kernel of KP1}, respectively. Then $K_1,K_{P_1}\in\B(\ell^p(\Z))$ for all $1< p<\infty$. 
\end{proposition}
\begin{proof}
\underline{{\bf{(1) For ${\bm{K_{P_1}}}$}}}, it follows from \eqref{kernel of KP1} and \eqref{kernel of R0 boundary} that
\begin{align}\label{expre of KP1}
K_{P_1}(n,m)&=\frac{1}{16}\sum\limits_{j=1}^{2}\int_{0}^{2}ia_1(\mu)\chi_1(\mu)\sum\limits_{m_1,m_2\in\Z}({I}^+_j+{I}^-_j) (\mu,N_1,M_2)(vP_1v)(m_1,m_2)d\mu,\notag\\
&:=\frac{1}{16}\sum\limits_{j=1}^{2}(K^{+,j}_{P_1}+K^{-,j}_{P_1})(n,m),
\end{align}
where $N_1=n-m_1$, $M_2=m-m_2$ and
$${I}^{\pm}_1(\mu,N_1,M_2)=ia_1(\mu)e^{-i\theta_+(|N_1|\pm|M_2|)},\quad{I}^{\pm}_2(\mu,N_1,M_2)=a_2(\mu)e^{b(\mu)|N_1|\pm i\theta_+|M_2|}.$$
By virtue of \eqref{decompostion of eitheta}, \eqref{varible substi1} and \eqref{decompostion of ebmu}, $K^{\pm,j}_{P_1}(n,m)$ can be written as follows, respectively:
\begin{align*}
K^{\pm,1}_{P_1}(n,m)&=\int_{-\pi}^{0}e^{-i\theta_+(|n|\pm|m|)}\tilde{\chi}_1(\mu(\theta_+))L^{\pm,1}_{P_1}(\theta_+,n,m)d\theta_+:=\int_{-\pi}^{0}e^{-i\theta_+(|n|\pm|m|)}G^{\pm,1}_{P_1}(\theta_+,n,m)d\theta_+,\\
K^{\pm,2}_{P_1}(n,m)&=\int_{0}^{2}e^{b(\mu)|n|\pm i\theta_+|m|}\tilde{\tilde{\chi}}_1(\mu)L^{\pm,2}_{P_1}(\mu,n,m)d\mu\notag:=\int_{0}^{2}e^{b(\mu)|n|\pm i\theta_+|m|}G^{\pm,2}_{P_1}(\mu,n,m)d\mu,
\end{align*}
where $\tilde{\chi}_1(\mu)=-(1-\frac{\mu^2}{4})^{-\frac{1}{2}}\chi_1(\mu)$, $\tilde{\tilde{\chi}}_1(\mu)=-i(1-\frac{\mu^4}{16})^{-\frac{1}{2}}\chi_1(\mu)$ and
\begin{align*}
L^{\pm,1}_{P_1}(\theta_+,n,m)&=\sum\limits_{m_1,m_2\in\Z}e^{-i\theta_+\big(|N_1|-|n|\pm(|M_2|-|m|)\big)}(vP_1v)(m_1,m_2),\\
L^{\pm,2}_{P_1}(\mu,n,m)&=\sum\limits_{m_1,m_2\in\Z}e^{b(\mu)(|N_1|-|n|)\pm i\theta_+(|M_2|-|m|)}(vP_1v)(m_1,m_2).
\end{align*}
Similarly, for any $k=0,1,2$, we can establish the following estimates:
\begin{align*}
\sup\limits_{\theta_{+}\in(-\pi,0)}\big|(\partial^k_{\theta_+}L^{\pm,1}_{P_1})(\theta_+,n,m)\big|+\sup_{\mu\in(0,\mu_0]}\big|e^{b(\mu)|n|}(\partial^k_\mu L^{\pm,2}_{P_1})(\mu,n,m)\big|\lesssim \|\left<\cdot\right>^{2k}V(\cdot)\|_{\ell^1},
\end{align*}
uniformly in $n,m\in\Z$. This immediately yields the uniform boundedness of $K^{\pm,j}_{P_1}(n,m)$ on $\Z^2$ for $j=1,2$. 
We consider decomposing $K^{\pm,j}_{P_1}(n,m)$ as follows:
\begin{equation}\label{decom of KpmP1}
K^{\pm,j}_{P_1}(n,m)=\left\{
\begin{aligned}
&\phi_{\pm}K^{\pm,1}_{P_1}(n,m)+(1-\phi_{\pm})K^{\pm,1}_{P_1}(n,m), &if\ j=1,\\
&\phi_{-}K^{\pm,2}_{P_1}(n,m)+(1-\phi_{-})K^{\pm,2}_{P_1}(n,m),&if\  j=2,
\end{aligned}
\right.
\end{equation}
where $\phi_{\pm}:=\phi\left(||n|\pm|m||^2\right)$ with $\phi$ as defined in Lemma \ref{C-Z lemma}. For the second terms in \eqref{decom of KpmP1}, the boundedness of $K^{\pm,j}_{P_1}(n,m)$ combined with the support of $\phi_{\pm}$ implies
\begin{align}\label{1-phipmKpmP1}
(1-\phi_{\pm})K^{\pm,1}_{P_1}(n,m)=O(\left<|n|\pm|m|\right>^{-2}),\quad (1-\phi_{-})K^{\pm,2}_{P_1}(n,m)=O(\left<|n|\pm|m|\right>^{-2}).
\end{align}
For the first terms, using the method for $K^{\pm,1}_{0}$ and $K^{\pm,2}_{0}$ in Proposition \ref{proposition good1 regular}, respectively, we obtain
\begin{align}
\phi_{\pm}K^{\pm,1}_{P_1}(n,m)&=\phi_{\pm}\lim\limits_{\theta_+\rightarrow0}\frac{ie^{-i\theta_+(|n|\pm|m|)}}{|n|\pm|m|}G^{\pm,1}_{P_1}(\theta_+,n,m)+\int_{-\pi}^{0}\frac{\phi_{\pm}e^{-i\theta_+(|n|\pm|m|)}}{i(|n|\pm|m|)}\frac{\partial{G^{\pm,1}_{P_1}}}{\partial\theta_+}(\theta_+,n,m)d\theta_+\notag\\
&=2(i-1)k^{\pm}_1(n,m)+O(\left<|n|\pm|m|\right>^{-2}),\label{kernel of phiKpm1P1 after integ}\\
\phi_-K^{\pm,2}_{P_1}(n,m)&=-\lim\limits_{\mu\rightarrow0}\frac{\phi_-e^{b(\mu)|n|\pm i\theta_+|m|}}{\alpha^{\pm}(\mu,n,m)}G^{\pm,2}_{P_1}(\mu,n,m)-\int_{0}^{2}\phi_-e^{b(\mu)|n|\pm i\theta_+|m|}\Big(\frac{G^{\pm,2}_{P_1}}{\alpha^{\pm}}\Big)'(\mu,n,m)d\mu,\notag\\
&=2(i-1)k^{\pm}_2(n,m)+O(\left<|n|\pm|m|\right>^{-2}),\label{kernel of phiKpm2P1 after integ}
\end{align}
where $\alpha^{\pm}(\mu,n,m)$ and $k^{\pm}_{\ell}(n,m)$ are defined in \eqref{defi of alphapm} and Lemma \ref{C-Z lemma}, respectively, and we used
$$\sum\limits_{m_1,m_2}(vP_1v)(m_1,m_2)=-2(i+1).$$
Therefore, combining \eqref{decom of KpmP1}$\sim$\eqref{kernel of phiKpm2P1 after integ} and \eqref{expre of KP1}, we derive
\begin{align}\label{expres of Kp1}
K_{P_1}(n,m)=\frac{i-1}{8}\Big(k^{+}_1(n,m)+k^{-}_1(n,m)+k^{+}_2(n,m)+k^{-}_2(n,m)\Big)+O(\left<|n|\pm|m|\right>^{-2}),
\end{align}
which together with Lemmas \ref{shur test} and \ref{C-Z lemma} gives that $K_{P_1}\in\B(\ell^p(\Z))$ for all $1<p<\infty$.
\vskip0.3cm
\underline{{\bf{(2) For ${\bm{K_1}}$}}}, by \eqref{kernel of K110} and Lemma \ref{cancelation lemma}, we have
\begin{align}\label{expre of K110}
K_1(n,m)=\frac{1}{16}\sum\limits_{j=1}^{2}(K^{+,j}_{1}-K^{-,j}_{1})(n,m),
\end{align}
where $N_1=n-\rho_1m_1$, $M_2=m-\rho_2m_2$ and 
\begin{align}
 K^{\pm,1}_{1}(n,m)&=\sum\limits_{m_1\in\Z}\int_{0}^{1}({\rm sign}(N_1))\int_{0}^{2}\chi_1(\mu)\left(\frac{\theta_+}{{\rm sin}\theta_+}\right)^2\sum\limits_{m_2\in\Z}\int_{0}^{1}({\rm sign}(M_2))e^{-i\theta_+(|N_1|\pm|M_2|)}d\rho_2\notag\\
  &\quad\times (v_1QA_{1}Qv_1)(m_1,m_2)d\mu d\rho_1:=\sum\limits_{m_1\in\Z}\int_{0}^{1}({\rm sign}(N_1))K^{\pm,1}_{\rho_1,m_1}(n,m)d\rho_1,\label{kernel of Kpmjrho1m1}\\
 K^{\pm,2}_{1}(n,m)&=\sum\limits_{m_1\in\Z}\int_{0}^{1}({\rm sign}(N_1))\int_{0}^{2}{\chi}_1(\mu)\frac{\theta_+}{{\rm sin}\theta_+}\frac{b(\mu)}{\mu}a_2(\mu)\sum\limits_{m_2\in\Z}\int_{0}^{1}({\rm sign}(M_2))e^{b(\mu)|N_1|\pm i\theta_+|M_2|}d\rho_2\notag\\
 &\quad\times (v_1QA_{1}Qv_1)(m_1,m_2)d\mu d\rho_1:=\sum\limits_{m_1\in\Z}\int_{0}^{1}({\rm sign}(N_1))K^{\pm,2}_{\rho_1,m_1}(n,m)d\rho_1.\notag
 \end{align}
For any fixed parameters $(\rho_1,m_1)\in[0,1]\times\Z$, we shall establish the following estimates:
\begin{align}\label{estimate for kpmjrho1m1}
\begin{split}
K^{\pm,1}_{\rho_1,m_1}(n,m)&=ik^{\pm}_1(n,m)C_{1}(m_1,m)+O\Big(\mcaM_{1}(m_1)\left<|n|\pm|m|\right>^{-2}\Big),\\
K^{\pm,2}_{\rho_1,m_1}(n,m)&=k^{\pm}_2(n,m)C_{1}(m_1,m)+O\Big(\mcaM_{1}(m_1)\left<|n|\pm|m|\right>^{-2}\Big),
\end{split}
\end{align}
where $\mcaM_{1}(m_1)=\left<m_1\right>^3|v(m_1)|\big(|QA_1Q|(\left<\cdot\right>^3|v(\cdot)|)\big)(m_1)$ and
\begin{align*}
C_{1}(m_1,m)=\sum\limits_{m_2\in\Z}(v_1QA_{1}Qv_1)(m_1,m_2)\int_{0}^{1}({\rm sign}(M_2))d\rho_2.
\end{align*}
Once this is established, noting that $|C_{1}(m_1,m)|\leq \mcaM_{1}(m_1)$ uniformly in $m\in\Z$, which combined with Lemmas \ref{shur test} and \ref{C-Z lemma} and triangle inequality, yields that $K^{\pm,j}_{\rho_1,m_1}$ is $\ell^p$ bounded for any $1<p<\infty$ and satisfies
\begin{align}\label{estimate for Kpmjrho1m1}
\|K^{\pm,j}_{\rho_1,m_1}\|_{\ell^p\rightarrow\ell^p}\lesssim \mcaM_{1}(m_1),\quad j=1,2.
\end{align}
Then for any $j=1,2$ and $1<p<\infty$, we have
\begin{align*}
\|K^{\pm,j}_{1}f\|^p_{\ell^p}&=\sum\limits_{n\in\Z}\Big|\sum\limits_{m\in\Z}\sum\limits_{m_1\in\Z}\int_{0}^{1}({\rm sign}(N_1))K^{\pm,j}_{\rho_1,m_1}(n,m)d\rho_1f(m)\Big|^p\\
&\leq\sum\limits_{n\in\Z}\Big(\sum\limits_{m_1\in\Z}\int_{0}^{1}|(K^{\pm,j}_{\rho_1,m_1}f)(n)|d\rho_1 \Big)^p,
\end{align*}
which together with the Minkowski's inequality and \eqref{estimate for Kpmjrho1m1} concludes that
\begin{align*}
\|K^{\pm,j}_{1}f\|_{\ell^p}&\leq\sum\limits_{m_1\in\Z}\int_{0}^{1}\|K^{\pm,j}_{\rho_1,m_1}f\|_{\ell^p}d\rho_1\leq\|f\|_{\ell^p}\sum\limits_{m_1\in\Z}\int_{0}^{1}\|K^{\pm,j}_{\rho_1,m_1}\|_{\ell^p\rightarrow\ell^p}d\rho_1\\
&\lesssim \|f\|_{\ell^p}\sum\limits_{m_1\in\Z}\left<m_1\right>^3|v(m_1)|\big(|QA_{1}Q|(\left<\cdot\right>^3|v(\cdot)|)\big)(m_1)\lesssim\|f\|_{\ell^p}.
\end{align*}
This result combined with \eqref{expre of K110} establishes that $K_{1}\in\B(\ell^p(\Z))$ for $1<p<\infty$.
\vskip0.15cm
To obtain \eqref{estimate for kpmjrho1m1}, it follows from \eqref{kernel of Kpmjrho1m1} that for any given $(\rho_1,m_1)\in[0,1]\times\Z$,
\begin{align}
K^{\pm,1}_{\rho_1,m_1}(n,m)&=\int_{0}^{2}e^{-i\theta_+(|n|\pm|m|)}\chi_1(\mu)\left(\frac{\theta_+}{{\rm sin}\theta_+}\right)^2L^{\pm,1}_{\rho_1,m_1}(\theta_+,n,m)d\mu\notag\\
&\xlongequal[]{{\rm by}\ \eqref{varible substi1}}\int_{-\pi}^{0}e^{-i\theta_+(|n|\pm|m|)}\bar{\chi}_1(\mu(\theta_+))\frac{\theta^2_{+}}{2(1-{\rm cos}\theta_+)}L^{\pm,1}_{\rho_1,m_1}(\theta_+,n,m)d\theta_+\notag\\
&:=\int_{-\pi}^{0}e^{-i\theta_+(|n|\pm|m|)}G^{\pm,1}_{\rho_1,m_1}(\theta_+,n,m)d\theta_+,
\end{align}
and
\begin{align}
K^{\pm,2}_{\rho_1,m_1}(n,m)&=\int_{0}^{2}e^{b(\mu)|n|\pm i\theta_+|m|}\bar{\bar{\chi}}_{1}(\mu)\frac{\theta_+}{{\rm sin}\theta_+}\frac{b(\mu)}{\mu}L^{\pm,2}_{\rho_1,m_1}(\mu,n,m)d\mu\notag\\
&:=\int_{0}^{2}e^{b(\mu)|n|\pm i\theta_+|m|}G^{\pm,2}_{\rho_1,m_1}(\mu,n,m)d\mu,
\end{align}
where $\bar{\chi}_1(\mu)=\chi_1(\mu)(1-\frac{\mu^2}{4})^{-\frac{1}{2}}$, $\bar{\bar{\chi}}_{1}(\mu)=-\chi_1(\mu)(1+\frac{\mu^2}{4})^{-\frac{1}{2}}$ and
\begin{align*}
L^{\pm,1}_{\rho_1,m_1}(\theta_+,n,m)&=\sum\limits_{m_2\in\Z}\int_{0}^{1}({\rm sign}(M_2))e^{-i\theta_+\big(|N_1|-|n|\pm(|M_2|-|m|)\big)}d\rho_2(v_1QA_{1}Qv_1)(m_1,m_2),\\
L^{\pm,2}_{\rho_1,m_1}(\mu,n,m)&=\sum\limits_{m_2\in\Z}\int_{0}^{1}({\rm sign}(M_2))e^{b(\mu)(|N_1|-|n|)\pm i\theta_+(|M_2|-|m|)}d\rho_2(v_1QA_{1}Qv_1)(m_1,m_2).
\end{align*}
Then we have the following estimates for any $k=0,1,2$:
\begin{align}
\sup\limits_{\theta_{+}\in(-\pi,0)}\big|(\partial^k_{\theta_+}L^{\pm,1}_{\rho_1,m_1})(\theta_+,n,m)\big|+\sup_{\mu\in(0,\mu_0]}\big|e^{b(\mu)|n|}(\partial^k_\mu L^{\pm,2}_{\rho_1,m_1})(\mu,n,m)\big|\lesssim \mcaM_{1}(m_1),
\end{align}
uniformly in $n,m\in\Z$ and $\rho_1\in[0,1]$. Moreover, noting that
$$\lim\limits_{\theta_+\rightarrow0}\Big(\frac{\theta^2_{+}}{2(1-{\rm cos}\theta_+)}\Big)^{(k)},\ \lim\limits_{\mu\rightarrow0^+}\Big(\frac{b(\mu)}{\mu}\Big)^{(k)}\ {\rm and}\ \lim\limits_{\mu\rightarrow0^+}\Big(\frac{\theta_+}{{\rm sin}\theta_+}\Big)^{(k)}\ {\rm exist\ for}\ k=0,1,2.$$
Following an analogous argument to that used for the operators $K^{\pm,j}_{P_1}$, we then obtain the claimed estimates in \eqref{estimate for kpmjrho1m1}. This completes the entire proof.
\end{proof}
Hence, combining Propositions \ref{proposition good1 regular}, \ref{proposition good2 regular} and \ref{proposition K110KP1 regular} and \eqref{decom1 of W1 regular}, Theorem \ref{theorem of W1} holds for the regular case.
\subsection{$0$ is a first kind resonance of $H$}\label{subsec of K1 1st}In this subsection, we consider the case where $0$ is a first kind resonance of $H$. 
As before, taking the expansion \eqref{asy expan of 1st 0} 
\begin{align*}
\begin{split}
M^{-1}(\mu)&=\mu^{-1}S_{1}A_{-1}S_{1}+\big(S_0A^{1}_{01}Q+QA^{1}_{02}S_0\big)+\mu\big(S_0A^{1}_{11}+A^{1}_{12}S_0+QA^{1}_{13}Q\big)\\
&\quad+\mu^2\big(QA^{1}_{21}+A^{1}_{22}Q\big)+\mu^3(QA^{1}_{31}+A^{1}_{32}Q)+\mu^3P_1+\mu^3A^{1}_{33}+\Gamma^{1}_{4}(\mu)
\end{split}
\end{align*}
into \eqref{kernel of mcaK1}, we obtain
\begin{equation}\label{decom1 of mcaK1 1st}
\mcaK_1=\sum\limits_{A\in\mcaA_{11}\cup\mcaA_{12}}^{}K_A+K_{P_1}+K^1_4,
\end{equation}
where $\mcaA_{11}=\{\mu^{-1}S_{1}A_{-1}S_{1},S_0A^{1}_{01}Q,QA^{1}_{02}S_0,\mu S_0A^{1}_{11},\mu A^{1}_{12}S_0,\mu QA^{1}_{13}Q,\mu^2 QA^{1}_{21},\mu^2A^{1}_{22}Q,\mu^3A^{1}_{33}\},$
\vskip0.1cm
\noindent$\mcaA_{12}=\{\mu^3QA^{1}_{31},\mu^3A^{1}_{32}Q\}$, $K_{A},K_{P_1}$ are defined in \eqref{kernel of KP1} and
$$K^{1}_4(n,m)=\int_{0}^{2}\mu^3\chi_1(\mu)\big[R^{+}_{0}(\mu^4)v \Gamma^1_4(\mu)v(R^{+}_{0}-R^{-}_0)(\mu^4)\big](n,m)d\mu.$$
Similarly, based on \eqref{order of operators in cancel lemma 0}, we can classify these integral operators into two groups:
\begin{align*}
O(1):K~(K\in\{K_A:A\in\mcaA_{11}\}\cup\{K_{P_1}\}),\quad O(\mu):K~(K\in\{K^1_{4}\}\cup\{K_{A}:A\in\mcaA_{12}\}).
\end{align*}
Recalling the results in Propositions \ref{proposition good1 regular}, \ref{proposition good2 regular} and \ref{proposition K110KP1 regular}, we have derived the following $\ell^p$ boundedness:
\begin{itemize}
\item $K\in\B(\ell^p(\Z))\  {\rm for\  all\ } 1\leq p\leq\infty$ with $K\in\{K^1_{4}\}\cup\{K_{A}:A\in\mcaA_{12}\}$,
\vskip0.15cm
\item $K_{A},K_{P_1}\in\B(\ell^p(\Z))\  {\rm for\  all\ } 1< p<\infty$ with $A=\mu QA^1_{13}Q$.
\end{itemize}
Therefore, in this case, it suffices to establish the boundedness for $\{K_A:A\in\mcaA_{11}\setminus\{\mu QA^1_{13}Q\}\}$.
\begin{proposition}\label{proposition 1st}
Let $H=\Delta^2+V$ with $|V(n)|\lesssim \left<n\right>^{-\beta}$ for $\beta>19$. Suppose that $H$ has no positive eigenvalues in the interval $\rm{(}0,16\rm{)}$ and $0$ is a first kind resonance of $H$. Let $\mcaA_{11}$ be defined in \eqref{decom1 of mcaK1 1st}. Then for any $A\in\mcaA_{11}\setminus\{\mu QA^1_{13}Q\}$, $K_A\in\B(\ell^p(\Z))$ for all $1< p<\infty$ and therefore $\mcaK_1\in\B(\ell^p(\Z))$ for all $1< p<\infty$.
\end{proposition}
\begin{proof}
\underline{{\bf(1)}}~For $A=\mu^{-1}S_{1}A_{-1}S_{1}$, denote
\begin{align*}
\bm{K_{-1}(n,m)}=\int_{0}^{2}\mu^2\chi_1(\mu)\big[R^{+}_{0}(\mu^4)vS_1A_{-1}S_1v(R^{+}_{0}-R^{-}_0)(\mu^4)\big](n,m)d\mu.
\end{align*}
Since $S_1$ has the same cancelation as $S_0$, it follows directly from \eqref{kernel of S0A01S0}, \eqref{kernel of K01pm1}, \eqref{kernel of K01pm2}, \eqref{kernel of K01pm3} that
\begin{align}\label{kernel of S1A-1S1}
K_{-1}(n,m)=\frac{1}{16}\big(K^{(1)}_{-1}+K^{(2)}_{-1}\big)(n,m),
\end{align}
where 
$\widetilde{\chi}_{13}(\mu)=-i\chi_1(\mu)\frac{c_3(\mu)}{\mu^2}$ and 
\begin{align*}
K^{(1)}_{-1}(n,m)&=\frac{-i}{4}C_{-1}(k^+_1+k^-_1-k^+_2-k^-_2)(n,m)+O(\left<|n|\pm|m|\right>^{-2}),\\
K^{(2)}_{-1}(n,m)&=\int_{-\pi}^{0}e^{\pm i\theta_+(|n|+|m|)}\widetilde{\chi}_{13}(\mu(\theta_+))\frac{\theta^2_+}{2(1-{\rm cos}\theta_+)}L^{\pm}_{-1}(\theta_+,n,m)d\theta_+,
\end{align*}
with
\begin{equation}\label{C-1}
C_{-1}=\sum\limits_{m_1,m_2\in\Z}(v_2S_1A_{-1}S_1v_2)(m_1,m_2),
\end{equation}
\begin{align*}
L^{\pm}_{-1}(\theta_+,n,m)=\sum\limits_{m_1,m_2\in\Z}^{}\int_{0}^{1}(1-\rho_2)e^{\pm i\theta_+(|M_2|-|m|-|n|)}d\rho_2\cdot|n-m_1|\big(vS_1A_{-1}S_1v_2)(m_1,m_2).
\end{align*}
From Lemmas \ref{shur test} and \ref{C-Z lemma}, it is clear that $K^{(1)}_{-1}\in\B(\ell^p(\Z))$ for all $1<p<\infty$. As for $K^{(2)}_{-1}$, since
$$c_3(\mu)=-\frac{1}{3}\mu^3-\frac{1}{8}\mu^4+O(\mu^5),\quad \mu\rightarrow0^+,$$
we can apply the similar method for $K^{\pm,3}_{0}(n,m)$ in Proposition \ref{proposition good1 regular} to obtain

$$|K^{(2)}_{-1}(n,m)|\lesssim\left<|n|+|m|\right>^{-2},\quad {\rm for\ any\ }n,m\in\Z.$$
Thus, $K^{(2)}_{-1}\in\B(\ell^p(\Z))$ for any $1\leq p\leq\infty$ and we derive that $K_{-1}\in\B(\ell^p(\Z))$ for all $1<p<\infty$.
\vskip0.3cm
\underline{{\bf(2)}}~Let $A\in\mcaA_{11}\setminus\{\mu^{-1}S_{1}A_{-1}S_{1},\mu QA^1_{13}Q\}$. We first compute the expression of $\mcaK_A(\mu,n,m)$ in \eqref{mcaKA munm}. Combining the definition of $\mcaA_{11}$ and \eqref{mcaKA regular 0}, it remains to consider such expression for $A=S_0A^1_{01}Q,QA^1_{02}S_0,\mu^3A^1_{33}$. From Lemma \ref{cancelation lemma}, we obtain
\begin{align*}
\mcaK_A(\mu,n,m)=
\left\{\begin{aligned}&\sum\limits_{m_1,m_2\in\Z}\Big[\int_{[0,1]^2}^{}(1-\rho_1)({\rm sign}(M_2))\big(f^{+,1}_{01}-f^{-,1}_{01}\big)(\mu,N_1,{M}_2)d\rho_1d\rho_2\mcaM^1_{01}(m_1,m_2)+\\
&\int_{0}^{1}({\rm sign}(M_2))\big(f^{+,2}_{01}-f^{-,2}_{01}\big)(\mu,\widetilde{N}_1,{M}_2)d\rho_2(vS_0A^1_{01}Qv_1)(m_1,m_2)\Big],\ \  A=S_0A^1_{01}Q,\\
&\sum\limits_{m_1,m_2\in\Z}\int_{[0,1]^2}(1-\rho_2)({\rm sign}(N_1))(f_{02}^{+}+f_{02}^{-})(\mu,N_1,M_2)d\rho_1d\rho_2\\
&\quad\times\mcaM^1_{02}(m_1,m_2),\qquad\qquad\qquad\qquad\qquad\qquad A= QA^1_{02}S_0,\\
&\sum\limits_{m_1,m_2\in\Z}(vA^1_{33}v)(m_1,m_2)(f^+_{33}+f^-_{33})(\mu,\widetilde{N}_1,\widetilde{M}_2)d\rho_2,\quad A=\mu^3A^1_{33},\end{aligned}\right.
\end{align*}
where $N_1=n-\rho_1m_1$, $\widetilde{N}_1=n-m_1$, $M_2=m-\rho_2m_2$, $\widetilde{M}_2=m-m_2$,
\begin{align*}
\mcaM^1_{01}(m_1,m_2)&=(v_2S_0A^1_{01}Qv_1)(m_1,m_2),\quad \mcaM^1_{02}(m_1,m_2)=(v_1QA^1_{02}S_0v_2)(m_1,m_2),\\
f^{\pm,1}_{01}(\mu,N_1,M_2)&=i\mu^{-3}\theta^3_+a_{11}(\mu)\Phi^{\pm}_1(\mu,N_1,M_2)+\mu^{-3}\theta_+(b(\mu))^2a_{12}(\mu)\Phi^{\pm}_2(\mu,N_1,M_2),\\
f^{\pm,2}_{01}(\mu,\widetilde{N}_1,M_2)&=\mu^{-3}\theta_+c_3(\mu)|\widetilde{N}_1|e^{\pm i\theta_+|M_2|},\\
f^{\pm}_{02}(\mu,N_1,M_2)&=i\mu^{-3}\theta^3_+a_{11}(\mu)\Phi^{\pm}_1(\mu,N_1,M_2)+i\mu^{-3}\theta^2_+b(\mu)a_{12}(\mu)\Phi^{\pm}_2(\mu,N_1,M_2),\\
f^{\pm}_{33}(\mu,\widetilde{N}_1,\widetilde{M}_2)&=-a_{11}(\mu)\Phi^{\pm}_1(\mu,\widetilde{N}_1,\widetilde{M}_2)+ia_{12}(\mu)\Phi^{\pm}_2(\mu,\widetilde{N}_1,\widetilde{M}_2),
\end{align*}
with $a_{11}(\mu)$, $a_{12}(\mu)$, $\Phi^{\pm}_{j}(\mu,X,Y)$ defined in \eqref{expre of fijpm}. We notice that all these expressions together with \eqref{mcaKA regular 0} allow $K_A(n,m)$ to reduce to the operators types in Proposition \ref{proposition K110KP1 regular} and $K^{(2)}_{-1}$ in Proposition \ref{proposition 1st}. Consequently, using an analogous argument, we can derive
\begin{align}\label{kernel of KA 1st}
K_A(n,m)=
\left\{\begin{aligned}
&\frac{1}{32}h_{1,1}(n,m)C_{01}(m)+r(n,m):=\bm{K_{01}(n,m)},\quad A=S_0A^1_{01}Q,\\
&\frac{C_{02}(n)}{32}g_{1,i}(n,m)+r(n,m):=\bm{K_{02}(n,m)},\quad   A=QA^1_{02}S_0,\\
&\frac{C_{11}}{32}g_{i,-1}(n,m)+r(n,m):=\bm {K_{11}(n,m)},\quad  A=\mu S_0A^1_{11},\\
&\frac{C_{12}}{32}g_{i,i}(n,m)+r(n,m):=\bm{K_{12}(n,m)},\quad   A=\mu A^1_{12}S_0,\\
&\frac{-C_{21}(n)}{16}g_{1,i}(n,m)+r(n,m):=\bm{K_{21}(n,m)},\quad   A=\mu^2 QA^1_{21},\\
&\frac{1}{16}h_{-1,1}(n,m)C_{22}(m)+r(n,m):=\bm{K_{22}(n,m)},\quad  A=\mu^2 A^1_{22}Q,\\
&\frac{-C_{33}}{16}g_{i,i}(n,m)+r(n,m):=\bm{K_{33}(n,m)},\quad  A=\mu^3A^1_{33},\end{aligned}\right.
\end{align}
where $r(n,m)=O(\left<|n|\pm|m|\right>^{-2})$,
\begin{equation}\label{gab hab}
g_{a,b}(n,m)=\big(a(k^+_1+k^{-}_1)+b(k^+_2+k^{-}_2)\big)(n,m),\  h_{a,b}(n,m)=\big(a(k^+_1-k^{-}_1)+b(k^+_2-k^{-}_2)\big)(n,m),
\end{equation}
and
\begin{align}
C_{01}(m)&=\sum\limits_{m_1,m_2\in\Z}\int_{0}^{1}({\rm sign}(M_2))d\rho_2\cdot\mcaM^1_{01}(m_1,m_2),\quad C_{11}=\sum\limits_{m_1,m_2\in\Z}(v_2S_0A^1_{11}v)(m_1,m_2),\notag\\
C_{02}(n)&=\sum\limits_{m_1,m_2\in\Z}\int_{0}^{1}({\rm sign}(N_1))d\rho_1\cdot\mcaM^1_{02}(m_1,m_2),\quad C_{12}=\sum\limits_{m_1,m_2\in\Z}(vA^1_{12}S_0v_2)(m_1,m_2),\notag\\
C_{21}(n)&=\sum\limits_{m_1,m_2\in\Z}\int_{0}^{1}({\rm sign}(N_1))d\rho_1\cdot(v_1QA^1_{21}v)(m_1,m_2),\label{kernel of 1st}\\
C_{22}(m)&=\sum\limits_{m_1,m_2\in\Z}\int_{0}^{1}({\rm sign}(M_2))d\rho_2(vA^1_{22}Qv_1)(m_1,m_2),\quad
C_{33}=\sum\limits_{m_1,m_2\in\Z}(vA^1_{33}v)(m_1,m_2).\notag
\end{align}
Therefore, combining the uniform boundedness of $C_{01}(m),C_{02}(n),C_{21}(n)$ and $C_{22}(m)$ and Lemmas \ref{shur test} and \ref{C-Z lemma}, we obtain $K_{A}\in\B(\ell^p(\Z))$ for all $1<p<\infty$ and any $A\in\mcaA_{11}\setminus\{\mu^{-1}S_{1}A_{-1}S_{1},\mu QA^1_{13}Q\}$. This completes the whole proof together with {\bf(1)}.
\end{proof}

\subsection{$0$ is a second kind resonance of $H$}\label{subsec of K1 2nd}
In this subsection, we handle the case where $0$ is a second kind resonance of $H$. Compared to the previous two cases, this scenario exhibits some subtle behavior in boundedness analysis. First, as before, taking the expansion \eqref{asy expan 2nd 0}
\begin{align*}
M^{-1}(\mu)&=\frac{S_{2}A_{-3}S_{2}}{\mu^3}+\frac{S_2A_{-2,1}S_0+S_0A_{-2,2}S_2}{\mu^2}
+\frac{S_2A_{-1,1}Q+QA_{-1,2}S_2+S_0A_{-1,3}S_0}{\mu}\notag\\
 &\quad+\big(S_2A^{2}_{01}+A^{2}_{02}S_2+QA^{2}_{03}S_0+S_0A^{2}_{04}Q\big)+\mu\big(S_0A^{2}_{11}+A^{2}_{12}S_0+QA^{2}_{13}Q\big)\notag\\
&\quad+\mu^2\big(QA^{2}_{21}+A^{2}_{22}Q\big)+\mu^3(QA^{2}_{31}+A^{2}_{32}Q)+\mu^3P_1+\mu^3A^{2}_{33}+\Gamma^{2}_{4}(\mu)
\end{align*}
into \eqref{kernel of mcaK1}, we obtain
\begin{equation}\label{decom1 of mcaK1 2nd}
\mcaK_1=\sum\limits_{A\in\mcaA_{21}\cup\mcaA_{22}}^{}K_A+K_{P_1}+K^2_4,
\end{equation}
where $\mcaA_{22}=\{\mu^3QA^{2}_{31},\mu^3A^{2}_{32}Q\}$, $\mcaA_{21}=\mcaA^{(1)}_{21}\cup\mcaA^{(2)}_{21}$ with
\begin{align*}
\mcaA^{(1)}_{21}&=\{\mu^{-3}S_{2}A_{-3}S_{2},\mu^{-2}S_2A_{-2,1}S_0,\mu^{-2}S_0A_{-2,2}S_2,\mu^{-1}S_2A_{-1,1}Q,\mu^{-1}QA_{-1,2}S_2,S_2A^2_{01},A^2_{02}S_2\}\\
\mcaA^{(2)}_{21}&=\{\mu^{-1}S_{0}A_{-1,3}S_{0},QA^{2}_{03}S_0,S_0A^{2}_{04}Q,\mu S_0A^{2}_{11},\mu A^{2}_{12}S_0,\mu QA^{2}_{13}Q,\mu^2 QA^{2}_{21},\mu^2A^{2}_{22}Q,\mu^3A^{2}_{33}\}
\end{align*}
and $K_{A},K_{P_1}$ are defined in \eqref{decom1 of W1 regular} and
$$K^{2}_4(n,m)=\int_{0}^{2}\mu^3\chi_1(\mu)\big[R^{+}_{0}(\mu^4)v \Gamma^2_4(\mu)v(R^{+}_{0}-R^{-}_0)(\mu^4)\big](n,m)d\mu.$$
Similarly, based on \eqref{order of operators in cancel lemma 0}, we can classify these integral operators into two groups:
\begin{align*}
O(1):K~(K\in\{K_A:A\in\mcaA_{21}\}\cup\{K_{P_1}\}),\quad O(\mu):K~(K\in\{K^2_{4}\}\cup\{K_{A}:A\in\mcaA_{22}\}).
\end{align*}
Recalling the established results in Proposition \ref{proposition 1st}, we have derived the following $\ell^p$ boundedness:
\begin{itemize}
\item $K\in\B(\ell^p(\Z))\  {\rm for\  all\ } 1\leq p\leq\infty$ with $K\in\{K^2_{4}\}\cup\{K_{A}:A\in\mcaA_{22}\}$,
\vskip0.15cm
\item $K_{A},K_{P_1}\in\B(\ell^p(\Z))\  {\rm for\  all\ } 1< p<\infty$ and all  $A\in\mcaA^{(2)}_{21}$.
\end{itemize}
Therefore, it suffices to deal with the $\ell^p$ boundedness of the operators $\{K_A:A\in\mcaA^{(1)}_{21}\}$.
\begin{proposition}\label{proposition 2nd}
Let $H=\Delta^2+V$ with $|V(n)|\lesssim \left<n\right>^{-\beta}$ for $\beta>27$. Suppose that $H$ has no positive eigenvalues in the interval $\rm{(}0,16\rm{)}$ and $0$ is a second kind resonance of $H$. Let $\mcaA^{(1)}_{21}$ be defined in \eqref{decom1 of mcaK1 2nd}. Then for any $A\in\mcaA^{(1)}_{21}$, $K_A\in\B(\ell^p(\Z))$ for all $1< p<\infty$ and therefore $\mcaK_1\in\B(\ell^p(\Z))$ for all $1< p<\infty$.
\end{proposition}
\begin{proof}
 (1) For $A=\mu^{-3}S_{2}A_{-3}S_{2}$, denote
 \begin{align*}
    \bm{K_{-3}(n,m)}:=\int_{0}^{2}\chi_1(\mu)\big[R^{+}_{0}(\mu^4)vS_2A_{-3}S_2v(R^{+}_{0}-R^{-}_0)(\mu^4)\big](n,m)d\mu.
 \end{align*}
By Lemma \ref{cancelation lemma}, it can  further be expressed as
\begin{equation*}
K_{-3}(n,m)=\frac{1}{64}\sum\limits_{j=1}^{3}(K^{+,j}_{-3}-K^{-,j}_{-3})(n,m),
\end{equation*}
where $N_1=n-\rho_1m_1$, $M_2=m-\rho_2m_2$ , $a_{1j}(\mu),\Phi^{\pm}_{j}(\mu,N_1,M_2)$ are defined in \eqref{expre of fijpm}
 and
\begin{align*}
K^{\pm,1}_{-3}(n,m)&=\int_{0}^2\chi_1(\mu)\mu^{-6}\theta^6_+a_{11}(\mu)\sum\limits_{m_1,m_2\in\Z}\int_{[0,1]^2}(1-\rho_1)^2(1-\rho_2)^2({\rm sign}(N_1))({\rm sign}(M_2))\\
&\quad\times\Phi^{\pm}_{1}(\mu,N_1,M_2)d\rho_1d\rho_2(v_3S_2A_{-3}S_2v_3)(m_1,m_2)d\mu,\\ K^{\pm,2}_{-3}(n,m)&=\int_{0}^2\chi_1(\mu)\mu^{-6}\theta^3_+(b(\mu))^3a_{12}(\mu)\sum\limits_{m_1,m_2\in\Z}\int_{[0,1]^2}(1-\rho_1)^2(1-\rho_2)^2({\rm sign}(N_1))({\rm sign}(M_2))\\
&\quad\times\Phi^{\pm}_{2}(\mu,N_1,M_2)d\rho_1d\rho_2(v_3S_2A_{-3}S_2v_3)(m_1,m_2)d\mu,\\   K^{\pm,3}_{-3}(n,m)&=-2\int_{0}^2\chi_1(\mu)\mu^{-6}\theta^3_+c_3(\mu)a_1(\mu)\sum\limits_{m_1,m_2\in\Z}\int_{0}^1(1-\rho_2)^2({\rm sign}(M_2))e^{\pm i\theta_+|M_2|}d\rho_2\\
&\quad\times|n-m_1|(vS_2A_{-3}S_2v_3)(m_1,m_2)d\mu.
\end{align*}
For the first two terms, 
using the method for $K_1$ in Proposition \ref{proposition K110KP1 regular}, we can obtain
\begin{align*}
    K^{\pm,j}_{-3}(n,m)=\sum\limits_{m_1\in\Z}\int_{0}^1(1-\rho_1)^2({\rm sign}(N_1))\widetilde{K}^{\pm,j}_{-3}(m_1,n,m)d\rho_1, \quad j=1,2,
\end{align*}
where
\begin{align*}
 \widetilde{K}^{\pm,j}_{-3}(m_1,n,m)=i^jk^{\pm}_j(n,m)C_{-3}(m_1,m)+O(\mcaM_{-3}(m_1)\left<|n|\pm|m|\right>^{-2})
\end{align*}
with $C_{-3}(m_1,m)=\sum\limits_{m_2\in\Z} \int_{0}^1(1-\rho_2)^2({\rm sign}(M_2))d\rho_2(v_3S_2A_{-3}S_2v_3)(m_1,m_2)$ and
\begin{align*}
  \mcaM_{-3}(m_1)=\left<m_1\right>^5|v(m_1)|\big|S_2A_{-3}S_2\big|(\left<\cdot\right>^5|v(\cdot)|)(m_1).
\end{align*}
This establishes that $K^{\pm,j}_{-3}$ is $\ell^p$ bounded for all $1<p<\infty$ for $j=1,2$.
As for the third term, considering the decomposition
\begin{equation*}
    e^{\pm i\theta_+|M_2|}=  e^{\pm i\theta_+(|m|\pm i|n|)}e^{\pm i\theta_+(|M_2|-|m|)}e^{\theta_+|n|},
\end{equation*}
we further have
\begin{align*}
    K^{\pm,3}_{-3}(n,m)&=-2\int_{0}^2e^{\pm i\theta_+(|m|\pm i|n|)}\chi_1(\mu)\mu^{-6}\theta^3_+c_3(\mu)a_1(\mu)L^{\pm}_{-3}(\theta_+,n,m)d\mu,
\end{align*}
where
$$L^{\pm}_{-3}(\theta_+,n,m)=\sum\limits_{m_1,m_2}e^{\theta_+|n|}|n-m_1|\int_{0}^1(1-\rho_2)^2({\rm sign}(M_2))e^{\pm i\theta_+(|M_2|-|m|)}d\rho_2(vS_2A_{-3}S_2v_3)(m_1,m_2).$$
By an analogous argument as $L^{\pm,3}_0(\theta_+,n.m)$ in \eqref{kernel of K01pm3}, the following estimates also hold for any $k=0,1,2$:
\begin{equation*}
\sup\limits_{\theta_{+}\in(-\pi,0)}\big|(\partial^k_{\theta_+}L^{\pm}_{-3})(\theta_+,n,m)\big|\lesssim 1,\quad {\rm uniformly\ in}\ n,m\in\Z. 
\end{equation*}
Then applying the method for $K_{P_1}$, we derive
\begin{align*}
   K^{\pm,3}_{-3}(n,m)&=\frac{2}{3}\sum\limits_{m_1\in\Z}|n-m_1|v(m_1)k^{\mp}_2(n,m)(S_2A_{-3}S_2\varphi_m)(m_1)+O(\left<|n|\pm|m|\right>^{-2}) \\
   &:=K^{\pm}(n,m)+O(\left<|n|\pm|m|\right>^{-2}),
\end{align*}
where $$\varphi_m(m_2)=v_3(m_2)\int_{0}^{1}(1-\rho_2)^2({\rm sign}(M_2))d\rho_2.$$
Notice that a distinction from the previous two cases lies in the occurrence of the singular term $K^{\pm}(n,m)$. 
To deal with such term, we first use the orthogonality $\big<S_2f,v\big>=0$ to rewrite
\begin{equation}\label{phinm1}
 K^{\pm}(n,m)=\frac{2}{3}\sum\limits_{m_1\in\Z}\underbrace{(|n-m_1|-|n|)v(m_1)}_{:=\phi(n,m_1)}k^{\mp}_2(n,m)(S_2A_{-3}S_2\varphi_m)(m_1).
\end{equation}
For any $1<p<\infty$ and $f\in\ell^p(\Z)$, by Minkowski's inequality and Lemma \ref{C-Z lemma}, we have
\begin{align*}
 \|K^{\pm}f\|_{\ell^p}&\lesssim \|f\|_{\ell^p}\sum\limits_{m_1\in\Z}\|\phi(\cdot,m_1)\|_{\ell^{\infty}}\|k^{\mp}_2\|_{\ell^p-\ell^p}\Big(\big|S_2A_{-3}S_2\big|(|v_3|)\Big)(m_1)\\
 &\lesssim \|f\|_{\ell^p}\sum\limits_{m_1\in\Z}|v_1(m_1)|\Big(\big|S_2A_{-3}S_2\big|(|v_3|)\Big)(m_1)\\
 &\lesssim \|f\|_{\ell^p},
\end{align*}
where in the last inequality we used the absolute boundedness of $S_2A_{-3}S_2$ and H\"{o}lder's inequality.
Thus, $K^{\pm,3}_{-3}$ is $\ell^p$ bounded for all $1<p<\infty$ and this proves that
$K_{-3}\in\B(\ell^p(\Z))$ for $1<p<\infty$.
\vskip0.3cm
(2) For any $A\in\mcaA^{(1)}_{21}\setminus\{\mu^{-3}S_{2}A_{-3}S_{2}\}$, let $\phi(n,m_1)$ be as in \eqref{phinm1}, by a similar analysis to $K_{-3}$, we can derive
\begin{align}\label{kernel of KA 2nd}
K_A(n,m)=
\left\{\begin{aligned}
&\frac{C_{-2,1}(n)}{32}g_{-1,i}(n,m)+{\bm{\frac{1}{96}\big<(S_2A_{-2,1}S_0v_2)(\cdot),\phi(n,\cdot)\big>g_{0,-i}(n,m)}}+r(n,m)\\
&:=\bm{K_{-2,1}(n,m)},\qquad\qquad\qquad\qquad\qquad\qquad\qquad \quad\  A=\mu^{-2}S_2A_{-2,1}S_0,\\
&\frac{-1}{32}h_{1,1}(n,m)C_{-2,2}(m)+r(n,m):=\bm{K_{-2,2}(n,m)},\quad   A=\mu^{-2}S_0A_{-2,2}S_2,\\
&\frac{1}{32}\sum\limits_{m_1\in\Z}\int_{0}^1(1-\rho_1)^2({\rm sign}(N_1))h_{-i,1}(n,m)C_{-1,1}(m_1,m)d\rho_1+{\bm{\frac{1}{48} h_{0,1}(n,m)}}\\
&{\bm{\times\big<(S_2A_{-1,1}Q\widetilde{\varphi}_m)(\cdot),\phi(n,\cdot)\big>}}+r(n,m):=\bm{K_{-1,1}(n,m)},\quad  A=\mu^{-1} S_2A_{-1,1}Q,\\
&\frac{-1}{32}\sum\limits_{m_1\in\Z}\int_{0}^1({\rm sign}(N_1))h_{i,1}(n,m)C_{-1,2}(m_1,m)d\rho_1+r(n,m):=\bm{K_{-1,2}(n,m)},\\
&\qquad\qquad\qquad\qquad\qquad\qquad\qquad\qquad\qquad\qquad\qquad\qquad \quad\   A=\mu^{-1} QA_{-1,2}S_2,\\
&\frac{C^{(2)}_{01}(n)}{32}g_{1,-i}(n,m)+{\bm{\frac{1}{48}\big<(S_2A^2_{01}v)(\cdot),\phi(n,\cdot)\big>g_{0,i}(n,m)}}+r(n,m):=\bm{K^{(2)}_{01}(n,m)},\\
&\qquad\qquad\qquad\qquad\qquad\qquad\qquad\qquad\qquad\qquad\qquad\qquad \quad\ A=S_2 A^2_{01},\\
&\frac{1}{32}h_{1,-1}(n,m)C^2_{02}(m)+r(n,m):=\bm{K^{(2)}_{02}(n,m)},\quad  A=A^2_{02}S_2,\end{aligned}\right.
\end{align}
where $r(n,m)=O(\left<|n|\pm|m|\right>^{-2})$, $\widetilde{\varphi}_{m}(m_2)=v_1(m_2)\int_{0}^1({\rm sign}(M_2))d\rho_2$ and
\begin{align*}
 C_{-2,1}(n)&=\frac{1}{2}\sum\limits_{m_1,m_2\in\Z}\int_{0}^{1}(1-\rho_1)^2({\rm sign}(N_1))d\rho_1 (v_3S_2A_{-2,1}S_0v_2)(m_1,m_2),\\
 C_{-2,2}(m)&=\frac{1}{2}\sum\limits_{m_1,m_2\in\Z}\int_{0}^{1}(1-\rho_2)^2({\rm sign}(M_2))d\rho_2 (v_2S_0A_{-2,2}S_2v_3)(m_1,m_2),\\
 C_{-1,1}(m_1,m)&=\sum\limits_{m_2\in\Z}\int_{0}^1({\rm sign}(M_2))d\rho_2(v_3S_2A_{-1,1}Qv_1)(m_1,m_2),\\
C_{-1,2}(m_1,m)&=\sum\limits_{m_2\in\Z}\int_{0}^1(1-\rho_2)^2({\rm sign}(M_2))d\rho_2(v_1QA_{-1,2}S_2v_3)(m_1,m_2),\\
    C^{(2)}_{01}(n)&=\sum\limits_{m_1,m_2\in\Z}\int_{0}^{1}(1-\rho_1)^2({\rm sign}(N_1))d\rho_1 (v_3S_2A^2_{01}v)(m_1,m_2),\\
    C^2_{02}(m)&=\sum\limits_{m_1,m_2\in\Z}\int_{0}^{1}(1-\rho_2)^2({\rm sign}(M_2))d\rho_2 (vA^2_{02}S_2v_3)(m_1,m_2).
\end{align*}
Therefore, we prove that $K_A\in\B(\ell^p(\Z))$ for all $1<p<\infty$ and any $A\in\mcaA^{(1)}_{21}$ and complete the whole proof.
\end{proof}
Therefore, combining Subsections \ref{subsec of W1 regular}, \ref{subsec of K1 1st} and \ref{subsec of K1 2nd}, Theorem \ref{theorem of W1} is derived. Finally, we end this section with the proof of Lemma \ref{C-Z lemma}.
\begin{proof}[Proof of Lemma \ref{C-Z lemma}]
 Let $\chi_{\pm}=\chi_{\Z^{\pm}}$ be the characteristic function on $\Z^{\pm}:=\{m\in\Z:\pm m>0\}$ and define $(\tau f)(n)=f(-n)$. We introduce the kernel functions:
\begin{align*}
\widetilde{k}_1(n,m)=\frac{\phi(|n-m|^2)}{n-m},\quad \widetilde{k}^{\pm}_2(n,m)=\frac{\phi(|n-m|^2)}{n\pm im}.
\end{align*}
The $\ell^p$ boundedness of the operators $k^{\pm}_{1}$ and $k^{\pm}_{2}$ can be reduced to that of $\widetilde{k}_1$ and $\widetilde{k}^{\pm}_2$ through the following relations:
\begin{equation}\label{relation kpm1 and ktuta1}
({k}^{\pm}_1f)(n)=\big[(\chi_+\widetilde{k}_1\chi_{\mp}-\chi_-\widetilde{k}_1\chi_{\pm})(1+\tau)f\big](n),
\end{equation}
\begin{align}\label{relation kpm2 and kpmtuta2}
({k}^{\pm}_2f)(n)=\big[(\chi_+\widetilde{k}^{\pm}_2\chi_{+}-\chi_-\widetilde{k}^{\pm}_2\chi_{-})(1+\tau)f\big](n).
\end{align}
Indeed, noting that $k^{\pm}_{1}(n,m)=\widetilde{k}_1(|n|,\mp|m|)$ and $\widetilde{k}_1(-n,-m)=-\widetilde{k}_1(n,m)$, we have
\begin{align*}
k^{\pm}_{1}(n,m)&=(\chi_{+}(n)+\chi_{-}(n))k^{\pm}_{1}(n,m)(\chi_{+}(m)+\chi_{-}(m))\\
&=\chi_{+}(n)\widetilde{k}_1(n,\mp m)\chi_{+}(m)+\chi_{+}(n)\widetilde{k}_1(n,\pm m)\chi_{-}(m)\\
&\quad-\chi_{-}(n)\widetilde{k}_1(n,\pm m)\chi_{+}(m)-\chi_{-}(n)\widetilde{k}_1(n,\mp m)\chi_{-}(m).
\end{align*}
 Then equation \eqref{relation kpm1 and ktuta1} follows by making the change of variable $m\mapsto-m$ in the first and fourth terms for the $``+"$ case, and the second and third terms for the $``-"$ case, respectively. Similarly, we can obtain
\begin{align*}
k^{\pm}_2(n,m)&=\chi_{+}(n)\widetilde{k}^{\pm}_2(n,m)\chi_{+}(m)+\chi_{+}(n)\widetilde{k}^{\pm}_2(n,-m)\chi_{-}(m)\\
&\quad-\chi_{-}(n)\widetilde{k}^{\pm}_2(n,-m)\chi_{+}(m)-\chi_{-}(n)\widetilde{k}^{\pm}_2(n,m)\chi_{-}(m).
\end{align*}
Applying the variable substitution $m\mapsto-m$ in the second and third terms yields \eqref{relation kpm2 and kpmtuta2}. Since \cite[Lemma 3.3]{MWY24} has established that $T_{\widetilde{k}_1}$ and $T_{\widetilde{k}^{\pm}_2}$ are Calder\'{o}n-Zygmund operators, thus by Theorem \ref{Tdis lemma}, it follows that $\widetilde{k}_1$ and $\widetilde{k}^{\pm}_2$ are $\ell^p$ bounded for $1<p<\infty$. We then get the desired result.
\end{proof}
\section{The intermediate energy part $\mcaK_2$}\label{sec of W2}
\begin{theorem}\label{theorem for middle part}
Let $H=\Delta^2+V$ with $|V(n)|\lesssim \left<n\right>^{-\beta}$ for $\beta>3$. Suppose that $H$ has no positive eigenvalues in the interval $\rm{(}0,16\rm{)}$, then $\mcaK_{2}\in\B(\ell^p(\Z))$ for $1\leq p\leq\infty$.
\end{theorem}
\begin{proof}
Recall from \eqref{expre of W-} and by virtue of the identity
\begin{equation*}
R^+_V(\mu^4)=R^{+}_{0}(\mu^4)-R^{+}_{0}(\mu^4)VR^{+}_{V}(\mu^4),
\end{equation*}
the kernel of $\mcaK_2$ is given by
\begin{align*}
\mcaK_2(n,m)=(\mcaK_{21}-\mcaK_{22})(n,m),
\end{align*}
where
\begin{align}\label{kernel of W2j}
\begin{split}
\mcaK_{21}(n,m)&=\int_{0}^{2}\mu^3\chi_2(\mu)\big[R^{+}_{0}(\mu^4)V(R^{+}_{0}-R^{-}_0)(\mu^4)\big](n,m)d\mu,\\
\mcaK_{22}(n,m)&=\int_{0}^{2}\mu^3\chi_2(\mu)\big[R^{+}_{0}(\mu^4)VR^{+}_{V}(\mu^4)V(R^{+}_{0}-R^{-}_0)(\mu^4)\big](n,m)d\mu.
\end{split}
\end{align}
Next, we claim that both kernels $\mcaK_{2j}(j=1,2)$ satisfy the estimate
\begin{align}\label{estimate for mcaK2j}
|\mcaK_{2j}(n,m)|\lesssim \left<|n|\pm|m|\right>^{-2}.
\end{align}
Combining this with Lemma \ref{shur test}, we conclude that $\mcaK_2\in\B(\ell^p(\Z))$ for all $1\leq p\leq\infty$.
\vskip0.3cm
\underline{{\bf{(1) For ${\bm{j=1}}$}}}, it follows from \eqref{kernel of W2j} and \eqref{kernel of R0 boundary} that
\begin{align}\label{expr of W21}
\mcaK_{21}(n,m)=\frac{1}{16}\sum\limits_{j=1}^{2}(\mcaK^{+,j}_{21}+\mcaK^{-,j}_{21})(n,m),
\end{align}
where $N_1=n-m_1$, $M_1=m-m_1$ and
\begin{align*}
\mcaK^{\pm,1}_{21}(n,m)&=-\int_{0}^{2}\mu^{-3}(a_1(\mu))^2\chi_2(\mu)\sum\limits_{m_1\in\Z}e^{-i\theta_+(|N_1|\pm|M_1|)}V(m_1)d\mu,\\
\mcaK^{\pm,2}_{21}(n,m)&=i\int_{0}^{2}\mu^{-3}a_1(\mu)a_2(\mu)\chi_2(\mu)\sum\limits_{m_1\in\Z}e^{b(\mu)|N_1|\pm i\theta_+|M_1|}V(m_1)d\mu.
\end{align*}
By applying the argument for $K^{\pm,j}_0$ in Proposition \ref{proposition good1 regular} to $\mcaK^{\pm,j}_{21}$, while noting that $supp\chi_2(\mu)\subseteq[\mu_0,2-\mu_0]$, it is not difficult to obtain
\begin{align*}
|\mcaK^{\pm,j}_{21}(n,m)|\lesssim \left<|n|\pm|m|\right>^{-2},\quad j=1,2,
\end{align*}
which establishes \eqref{estimate for mcaK2j} for $\mcaK_{21}$.

\vskip0.3cm
{\bf{\underline{{(2) For ${\bm{j=2}}$}}}}, by \eqref{kernel of W2j} and \eqref{kernel of R0 boundary}, we have
\begin{align*}
\mcaK_{22}(n,m)=\frac{1}{16}\sum\limits_{j=1}^{2}(\mcaK^{+,j}_{22}+\mcaK^{-,j}_{22})(n,m),
\end{align*}
where $N_1=n-m_1$, $M_2=m-m_2$ and
\begin{align*}
\mcaK^{\pm,1}_{22}(n,m)&=-\int_{0}^{2}e^{-i\theta_+(|n|\pm|m|)}\mu^{-3}(a_1(\mu))^2\chi_2(\mu)L^{\pm,1}_{22}(\mu,n,m)d\mu,\\
\mcaK^{\pm,2}_{22}(n,m)&=i\int_{0}^{2}e^{b(\mu)|n|\pm i\theta_+|m|}\mu^{-3}a_1(\mu)a_2(\mu)\chi_2(\mu)L^{\pm,2}_{22}(\mu,n,m)d\mu,
\end{align*}
with
\begin{align*}
L^{\pm,1}_{22}(\mu,n,m)&=\sum\limits_{m_1,m_2\in\Z}e^{-i\theta_+(|N_1|-|n|\pm(|M_2|-|m|))}(VR^+_{V}(\mu^4)V)(m_1,m_2),\\
L^{\pm,2}_{22}(\mu,n,m)&=\sum\limits_{m_1,m_2\in\Z}e^{b(\mu)(|N_1|-|n|)\pm i\theta_+(|M_2|-|m|)}
(VR^+_{V}(\mu^4)V)(m_1,m_2).
\end{align*}
We shall show that for any $k=0,1,2$, the following estimates hold:
\begin{align}\label{estimate of Lpm12 22}
\sup\limits_{\mu\in[\mu_0,2-\mu_0]}\Big|\big(\partial^k_\mu L^{\pm,1}_{22}\big)(\mu,n,m)\Big|+\sup\limits_{\mu\in[\mu_0,2-\mu_0]}\Big|e^{b(\mu)|n|}\big(\partial^k_\mu L^{\pm,2}_{22}\big)(\mu,n,m)\Big|\lesssim1,
\end{align}
uniformly in $n,m\in\Z$. With this established, using $supp\chi_2(\mu)\subseteq[\mu_0,2-\mu_0]$ and applying the method used for $K^{\pm,j}_{0}$ to $\mcaK^{\pm,j}_{22}$, we can derive
 \begin{align*}
|\mcaK^{\pm,j}_{22}(n,m)|\lesssim \left<|n|\pm|m|\right>^{-2},\quad j=1,2,
\end{align*}
which gives \eqref{estimate for mcaK2j}. To obtain \eqref{estimate of Lpm12 22}, we focus on $L^{\pm,2}_{22}$~(the case for $L^{\pm,1}_{22}$ being similar). For $k=0,1,2$,
\begin{align}\label{derivate of Lpm2 22}
\big(\partial^k_\mu L^{\pm,2}_{22}\big)(\mu,n,m)&=\sum\limits_{k_1=0}^{k}\sum\limits_{m_1,m_2\in\Z}C^{k_1}_{k}\partial^{k-k_1}_{\mu}\big(e^{b(\mu)(|N_1|-|n|)\pm i\theta_+(|M_2|-|m|)}\big)(V\partial^{k_1}_\mu(R^+_{V}(\mu^4))V)(m_1,m_2)\notag\\
&=\sum\limits_{k_1=0}^{k}\sum\limits_{m_1,m_2\in\Z}C^{k_1}_{k}\partial^{k-k_1}_{\mu}\big(e^{b(\mu)(|N_1|-|n|)\pm i\theta_+(|M_2|-|m|)}\big)V(m_1)\left<m_1\right>^{\varepsilon_{k_1}}\\
&\quad\times\big(\left<\cdot\right>^{-\varepsilon_{k_1}}\partial^{k_1}_\mu(R^+_{V}(\mu^4))\left<\cdot\right>^{-\varepsilon_{k_1}}\big)(m_1,m_2)\left<m_2\right>^{\varepsilon_{k_1}}V(m_2),\notag
\end{align}
where $\varepsilon_{k_1}$ is a positive constant depending on $k_1$ specified later.
Noting that both $b'(\mu)$ and $\theta'_+(\mu)$ are smooth on $[\mu_0,2-\mu_0]$, and $b(\mu),b'(\mu)<0$, we have 
$$\sup\limits_{N_1,M_2,m}\sup\limits_{\mu\in[\mu_0,2-\mu_0]}\big|e^{b(\mu)|N_1|\pm i\theta_+(|M_2|-|m|)}\big|\lesssim 1.$$
These facts together with Lemma \ref{LAP-lemma} (taking $k_1+\frac{1}{2}<\varepsilon_{k_1}<\beta-\frac{1}{2}+k_1-k$ in \eqref{derivate of Lpm2 22}) yields 
\begin{align*}
\sup\limits_{n,m\in\Z}\sup\limits_{\mu\in[\mu_0,2-\mu_0]}\Big|e^{b(\mu)|n|}\big(\partial^k_\mu L^{\pm,2}_{22}\big)(\mu,n,m)\Big|\lesssim1,\quad k=0,1,2.
\end{align*}
Therefore, \eqref{estimate of Lpm12 22} is obtained and this completes the whole proof.
\end{proof}
\section{The high energy part $\mcaK_3$}\label{sec of mcaK3}
\begin{theorem}\label{theorem for W3}
Let $H=\Delta^2+V$ with $|V(n)|\lesssim \left<n\right>^{-\beta}$ for some $\beta>0$. Suppose that $H$ has no positive eigenvalues in the interval $\rm{(}0,16\rm{)}$. If
 \begin{align*}
 \beta>\left\{\begin{aligned}&9,\ 16\ is\ a\ regular\ point\ of\ H,\\
&13,\ 16\ is\ a\ resonance\ of\ H,\\
&17,\ 16\ is\ an\ eigenvalue\ of\ H,\end{aligned}\right.
\end{align*}
then $\mcaK_3\in\B(\ell^p(\Z))$ for all $1<p<\infty$.
\end{theorem}
Prior to the proof, we first recall that
\begin{align}\label{kernel of W3}
\mcaK_3&=\int_{0}^{2}\mu^3\chi_3(\mu)\big[R^{+}_{0}(\mu^4)v M^{-1}(\mu)v(R^{+}_{0}-R^{-}_0)(\mu^4)\big]d\mu.
\end{align}

We remark that an important difference from Section \ref{sect of mcaK1} is that it is not straightforward to utilize the cancelation properties of the projection operators $\widetilde{Q}$, $\widetilde{S}_0$, $\widetilde{S}_1$ in the expansions of $M^{-1}(\mu)$ to eliminate the singularity at $\mu=2$. To overcome this difficulty, we resort to the unitary operator $J$~(defined in \eqref{J}), which can transfer the operator $R^{+}_{0}(\mu^4)v Bv(R^{+}_{0}-R^{-}_0)(\mu^4)$ to the form
\begin{equation}\label{R0vBvR0}
R^{+}_{0}(\mu^4)v Bv(R^{+}_{0}-R^{-}_0)(\mu^4)=\frac{1}{4\mu^4}J\big(R^{-}_{-\Delta}(4-\mu^2)+JR_{-\Delta}(-\mu^2)J\big)\tilde{v}B\tilde{v}(R^{-}_{-\Delta}-R^{+}_{-\Delta})(4-\mu^2)J
\end{equation}
via the relation $J^2=I$ and the formulas
$$ JR^{\pm}_{-\Delta}(\mu^2)J=-R^{\mp}_{-\Delta}(4-\mu^2),\quad R^{\pm}_{0}(\mu^4)=\frac{1}{2\mu^2}\left(R^{\pm}_{-\Delta}(\mu^2)-R_{-\Delta}(-\mu^2)\right),\quad\mu\in(0,2).$$
This form indicates that one can turn to establish the following lemma to eliminate the singularity.
\begin{lemma}\label{cancelation lemma16}{\rm(\cite[Lemma 4.9 and Lemma 4.14]{HY25})}
{ Let $\widetilde{Q}$, $\widetilde{S}_0$, $\widetilde{S}_1$ be as in Definition \ref{defini of sixteen}. For any $f\in\ell^{2}(\Z)$, then we have
\vskip0.15cm
\noindent{\rm(1)} $\big(R^{\mp}_{-\Delta}(4-\mu^2)\tilde{v}Wf\big)(n)=(2{\rm sin}\tilde{\theta}_+)^{-1}\sum\limits_{m\in\Z}\int_{0}^{1}{\bm {\tilde{\theta}_+}}({\rm sign}(n-\rho m))e^{\pm i\tilde{\theta}_{+}|n-\rho m|}d\rho\cdot \tilde{v}_1(m)(Wf)(m)$,
     \vskip0.15cm
    \qquad \qquad \qquad \quad\quad\qquad $:=(2{\rm sin}\tilde{\theta}_+)^{-1}\sum\limits_{m\in\Z}\widetilde{\mcaB}^{\pm}(\tilde{\theta}_+,n,m)(Wf)(m)$,\quad $W=\widetilde{Q}$, $\widetilde{S}_0$, $\widetilde{S}_1$,
    \vskip0.2cm
\noindent{\rm(2)} $\big[\big((R^{-}_{-\Delta}-R^{+}_{-\Delta})(4-\mu^2)\big)\tilde{v}\widetilde{S}_1f\big](n)=(2{\rm sin}\tilde{\theta}_+)^{-1}\sum\limits_{m\in\Z}\int_{0}^{1}i{\bm{\tilde{\theta}^2_+}}(\rho-1)(e^{i\tilde{\theta}_+|n-\rho m|}+e^{-i\tilde{\theta}_+|n-\rho m|})d\rho$
\vskip0.15cm
    \qquad \qquad \qquad\quad\qquad\qquad \qquad\qquad\qquad \qquad\qquad$\times\tilde{v}_2(m)(\widetilde{S}_1f)(m)$,
\vskip0.15cm
    \qquad \qquad \qquad\quad\qquad \qquad\qquad\qquad $:=(2{\rm sin}\tilde{\theta}_+)^{-1}\sum\limits_{m\in\Z}\widetilde{\mcaC}(\tilde{\theta}_+,n,m)(\widetilde{S}_1f)(m)$,
    \vskip0.2cm
\noindent{\rm(3)} $W\big(\tilde{v}R^{\mp}_{-\Delta}(4-\mu^2)f\big)=W\widetilde{f}^{\pm}$, \quad $\widetilde{S}_1\big(\tilde{v}\big((R^{-}_{-\Delta}-R^{+}_{-\Delta})(4-\mu^2)\big)f\big)=\widetilde{S}_1\widetilde{f}_1$,

\noindent where $\tilde{\theta}_{+}:=\tilde{\theta}_{+}(\mu)$ satisfies ${\rm cos}\tilde{\theta}_{+}=\frac{\mu^2}{2}-1$ with $ \tilde{\theta}_{+}\in(-\pi,0)$ and
\begin{equation*}
\quad \widetilde{f}^{\pm}(n)=(2{\rm sin}\tilde{\theta}_+)^{-1}\sum\limits_{m\in\Z}\widetilde{\mcaB}^{\pm}(\tilde{\theta}_+,m,n)f(m),\quad \widetilde{f}_1(n)=(2{\rm sin}\tilde{\theta}_+)^{-1}\sum\limits_{m\in\Z}\widetilde{\mcaC}(\tilde{\theta}_+,m,n)f(m).
\end{equation*}
}
\end{lemma}
\begin{remark}
{\rm Noting that $\tilde{\theta}_{+}=O((2-\mu)^{\frac{1}{2}})$ as $\mu\rightarrow2$, compared with $R^{\mp}_{-\Delta}(4-\mu^2)=O((2-\mu)^{-\frac{1}{2}})$, this lemma indicates that these operators can eliminate the singularity of $R^{\mp}_{-\Delta}(4-\mu^2)$. Precisely,
\begin{align}\label{order in cancel lemma 16}
\begin{split}
R^{\mp}_{-\Delta}(4-\mu^2)\tilde{v}W&=O(1), \quad\big((R^{-}_{-\Delta}-R^{+}_{-\Delta})(4-\mu^2)\big)\tilde{v}\widetilde{S}_1=O((2-\mu)^{\frac{1}{2}}),\quad W=\widetilde{Q},\widetilde{S}_0,\widetilde{S}_1,\\
W\tilde{v}R^{\mp}_{-\Delta}(4-\mu^2)&=O(1),\quad\widetilde{S}_1\big(\tilde{v}\big((R^{-}_{-\Delta}-R^{+}_{-\Delta})(4-\mu^2)\big)\big)=O((2-\mu)^{\frac{1}{2}}).
\end{split}
\end{align}
}
\end{remark}
To prove Theorem \ref{theorem for W3}, we will address each case individually in the following three subsections.
\subsection{$16$ is a regular point of $H$}\label{subsec of W3 regular}In this subsection, we prove the $\ell^p$ boundedness for $\mcaK_3$ when $16$ is a regular point of $H$. Recall the expansion \eqref{asy expan regular 2} of $M^{-1}(\mu)$ as $\mu\rightarrow2$:~
\begin{align*}
M^{-1}(\mu)=\widetilde{Q}B_0\widetilde{Q}+
(2-\mu)^{\frac{1}{2}}(\widetilde{Q}B^{0}_{11}+B^{0}_{12}\widetilde{Q})+(2-\mu)^{\frac{1}{2}}\widetilde{P}_1+(2-\mu) B^{0}_{21}+\Gamma^{0}_{\frac{3}{2}}(2-\mu)
\end{align*}
and substitute it into \eqref{kernel of W3}, then $\mcaK_3$ can be expressed as the sum of six integral operators:
\begin{equation}\label{decom1 of W3 regular}
\mcaK_3=\sum\limits_{B\in\mcaB_0}^{}K_B+K_{\widetilde{P}_1}+K^0_{r},
\end{equation}
where $\mcaB_0=\{\widetilde{Q}B_0\widetilde{Q},(2-\mu)^{\frac{1}{2}}\widetilde{Q}B^{0}_{11},(2-\mu)^{\frac{1}{2}}B^{0}_{12}\widetilde{Q},(2-\mu) B^{0}_{21}\}$ and
\begin{align}
K_B(n,m)&=\int_{0}^{2}\mu^3\chi_3(\mu)\big[R^{+}_{0}(\mu^4)vBv(R^{+}_{0}-R^{-}_0)(\mu^4)\big](n,m)d\mu,\quad B\in\mcaB_0\notag\\
K_{\widetilde{P}_1}(n,m)&=\int_{0}^{2}(2-\mu)^{\frac{1}{2}}\mu^3\chi_3(\mu)\big[R^{+}_{0}(\mu^4)v \widetilde{P}_1v(R^{+}_{0}-R^{-}_0)(\mu^4)\big](n,m)d\mu,\label{kernel of KPtuta1}\\
K^{0}_r(n,m)&=\int_{0}^{2}\mu^3\chi_3(\mu)\big[R^{+}_{0}(\mu^4)v \Gamma^0_{\frac{3}{2}}(2-\mu)v(R^{+}_{0}-R^{-}_0)(\mu^4)\big](n,m)d\mu.\label{kernel of Kr0}
\end{align}
 Based on \eqref{R0vBvR0} and \eqref{order in cancel lemma 16}, we can further classify these operators as the following three types according to their order in $(2-\mu)$ as $\mu\rightarrow2$:
\begin{align*}
O(1):K_B(B\in\mcaB_0),\quad O((2-\mu)^{-\frac{1}{2}}):K_{\widetilde{P}_1} ,\quad O((2-\mu)^{\frac{1}{2}}):K^0_r.
\end{align*}

\begin{proposition}\label{proposition good1 regular 2}
Let $H=\Delta^2+V$ with $|V(n)|\lesssim \left<n\right>^{-\beta}$ for $\beta>9$. Suppose that $H$ has no positive eigenvalues in the interval $\rm{(}0,16\rm{)}$ and $16$ is a regular point of $H$. Let $\mcaB_0$ be defined in \eqref{decom1 of W3 regular}. Then
\vskip0.15cm
{\rm(1)}~for any $K\in\{K^0_{r}\}\cup\{K_B:B\in\mcaB_0\}$, $K\in\B(\ell^p(\Z))$ for any $1\leq p\leq\infty$,
\vskip0.15cm
{\rm (2)}~$K_{\widetilde{P}_1}\in\B(\ell^p(\Z))$ for all $1< p<\infty$.
\vskip0.15cm
\noindent Therefore, $\mcaK_3\in\B(\ell^p(\Z))$ for any $1< p<\infty$.
\end{proposition}
\begin{proof}
\underline{{\bf(1)}}~\underline{{\bf Step 1:}}~For any $B\in\mcaB_0$, denote
$$\mcaK_B(\mu,n,m)=16\mu^3\big(R^{+}_{0}(\mu^4)v Bv(R^{+}_{0}-R^{-}_0)(\mu^4)\big)(n,m).$$
It follows from \eqref{R0vBvR0} and Lemma \ref{cancelation lemma16} that
\begin{align}\label{mcaKB regular 2}
\mcaK_B(\mu,n,m)=
\left\{\begin{aligned}&\sum\limits_{m_1,m_2\in\Z}\Big[\int_{[0,1]^2}(-1)^{n+m}\widetilde{\mcaM}^{(1)}_{0}(N_1,M_2,m_1,m_2)(f_{0}^{+,1}-f_{0}^{-,1})(\mu,N_1,M_2)d\rho_1d\rho_2\\
&\quad+ \int_{0}^{1}(-1)^{m}\widetilde{\mcaM}^{(2)}_{0}(M_2,m_1,m_2)(f_{0}^{+,2}-f_{0}^{-,2})(\mu,\widetilde{N}_1,M_2)d\rho_2\Big],\ \ B=\widetilde{Q}B_0\widetilde{Q},\\
&\sum\limits_{m_1,m_2\in\Z}\Big[\int_{0}^{1}(-1)^{n+m}\widetilde{\mcaM}^{(1)}_{11}(N_1,m_1,m_2)\big(f^{+,1}_{11}+f^{-,1}_{11}\big)(\mu,N_1,\widetilde{M}_2)d\rho_1+\\
&\quad+(-1)^{m}\widetilde{\mcaM}^{(2)}_{11}(m_1,m_2)\big(f^{+,2}_{11}+f^{-,2}_{11}\big)(\mu,\widetilde{N}_1,\widetilde{M}_2)\Big],\quad B=(2-\mu)^{\frac{1}{2}}\widetilde{Q}B^0_{11},\\
&\sum\limits_{m_1,m_2\in\Z}\int_{0}^{1}\Big[(-1)^{n+m}\widetilde{\mcaM}^{(1)}_{12}(M_2,m_1,m_2)(f^{+,1}_{12}-f^{-,1}_{12})(\mu,\widetilde{N}_1,M_2)+(-1)^{m}\times\\
&\qquad\widetilde{\mcaM}^{(2)}_{12}(M_2,m_1,m_2)(f^{+,2}_{12}-f^{-,2}_{12})(\mu,\widetilde{N}_1,M_2)\Big]d\rho_2,\quad B=(2-\mu)^{\frac{1}{2}}B^0_{12}\widetilde{Q},\\
&\sum\limits_{m_1,m_2\in\Z}\Big[(-1)^{n+m}\widetilde{\mcaM}^{(1)}_{21}(m_1,m_2)((f^{+,1}_{21}+f^{-,1}_{21})(\mu,\widetilde{N}_1,\widetilde{M}_2))\\
&\quad\qquad+(-1)^m\widetilde{\mcaM}^{(2)}_{21}(m_1,m_2)((f^{+,2}_{21}+f^{-,2}_{21})(\mu,\widetilde{N}_1,\widetilde{M}_2))\Big],\quad B=(2-\mu)B^0_{21},\end{aligned}\right.
\end{align}
where $N_1=n-\rho_1m_1$, $\widetilde{N}_1=n-m_1$, $M_2=m-\rho_2m_2$, $\widetilde{M}_2=m-m_2$,
\begin{align*}
\widetilde{\Phi}^{\pm}_1(\mu,X,Y)&=e^{i\tilde{\theta}_+(|X|\pm|Y|)},\quad \widetilde{\Phi}^{\pm}_2(\mu,X,Y)=e^{b(\mu)|X|\pm i\tilde{\theta}_+|Y|},\\
f^{\pm,1}_0(\mu,N_1,M_2)&=\frac{\tilde{\theta}^2_+}{\mu({\rm sin}\tilde{\theta}_+)^2}\widetilde{\Phi}^{\pm}_1(\mu,N_1,M_2),\quad
f^{\pm,2}_0(\mu,\widetilde{N}_1,M_2)=\frac{-a_2(\mu)\tilde{\theta}_+}{\mu^2{\rm sin}\tilde{\theta}_+}\widetilde{\Phi}^{\pm}_2(\mu,\widetilde{N}_1,M_2),\\
f^{\pm,1}_{11}(\mu,N_1,\widetilde{M}_2)&=\frac{i(2-\mu)^{\frac{1}{2}}\tilde{\theta}_+}{\mu({\rm sin}\tilde{\theta}_+)^2}\widetilde{\Phi}^{\pm}_1(\mu,N_1,\widetilde{M}_2),\quad f^{\pm,1}_{12}(\mu,\widetilde{N}_1,{M}_2)=f^{\pm,1}_{11}(\mu,\widetilde{N}_1,{M}_2),\\
f^{\pm,2}_{11}(\mu,\widetilde{N}_1,\widetilde{M}_2)&=\frac{-ia_2(\mu)(2-\mu)^{\frac{1}{2}}}{\mu^2{\rm sin}\tilde{\theta}_+}\widetilde{\Phi}^{\pm}_2(\mu,\widetilde{N}_1,\widetilde{M}_2),\quad f^{\pm,2}_{12}(\mu,\widetilde{N}_1,{M}_2)=-i\tilde{\theta}_+f^{\pm,2}_{11}(\mu,\widetilde{N}_1,{M}_2),\\
f^{\pm,1}_{21}(\mu,\widetilde{N}_1,\widetilde{M}_2)&=\frac{\mu-2}{\mu({\rm sin}\tilde{\theta}_+)^2}\widetilde{\Phi}^{\pm}_1(\mu,\widetilde{N}_1,\widetilde{M}_2),\quad
f^{\pm,2}_{21}(\mu,\widetilde{N}_1,\widetilde{M}_2)=\frac{ia_2(\mu)(\mu-2)}{\mu^2{\rm sin}\tilde{\theta}_+}\widetilde{\Phi}^{\pm}_2(\mu,\widetilde{N}_1,\widetilde{M}_2),
\end{align*}
and
\begin{itemize}
\item $\widetilde{\mcaM}^{(1)}_{0}(N_1,M_2,m_1,m_2)=({\rm sign}(N_1))({\rm sign}(M_2))(\tilde{v}_1\widetilde{Q}B_0\widetilde{Q}\tilde{v}_1)(m_1,m_2),$
    \vskip0.15cm
    \item $\widetilde{\mcaM}^{(2)}_{0}(M_2,m_1,m_2)=({\rm sign}(M_2))(v\widetilde{Q}B_0\widetilde{Q}\tilde{v}_1)(m_1,m_2),$
        \vskip0.15cm
    \item $\widetilde{\mcaM}^{(1)}_{11}(N_1,m_1,m_2)=({\rm sign}(N_1))(\tilde{v}_1\widetilde{Q}B^0_{11}\tilde{v})(m_1,m_2)$,\ \ $\widetilde{\mcaM}^{(2)}_{11}(m_1,m_2)=({v}\widetilde{Q}B^0_{11}\tilde{{v}})(m_1,m_2)$,
\vskip0.15cm
    \item $\widetilde{\mcaM}^{(1)}_{12}(M_2,m_1,m_2)=({\rm sign}(M_2))(\tilde{v}B^0_{12}\widetilde{Q}\tilde{v}_1)(m_1,m_2)$, \quad $\widetilde{\mcaM}^{(1)}_{21}(m_1,m_2)=(\tilde{v}B^0_{21}\tilde{v})(m_1,m_2)$,
\vskip0.15cm
    \item$\widetilde{\mcaM}^{(2)}_{12}(M_2,m_1,m_2)=({\rm sign}(M_2))({v}B^0_{12}\widetilde{Q}\tilde{v}_1)(m_1,m_2)$,\quad $\widetilde{\mcaM}^{(2)}_{21}(m_1,m_2)=({v}B^0_{21}\tilde{v})(m_1,m_2)$.
\end{itemize}
\vskip0.15cm
In view that $\mcaK_B(\mu,n,m)=O(1)$ as $\mu\rightarrow2$ for any $B\in\mcaB_0$, next we consider the case $B=\widetilde{Q}B_0\widetilde{Q}$ only  and other terms can be derived similarly. Let
 \begin{equation}
 \widetilde{K}_0(n,m)=\int_{0}^{2}\mu^3\chi_3(\mu)\big[R^{+}_{0}(\mu^4)v \widetilde{Q}B_0\widetilde{Q}v(R^{+}_{0}-R^{-}_0)(\mu^4)\big](n,m)d\mu.
 \end{equation}
From \eqref{mcaKB regular 2}, it reduces to establish the boundedness of the following two operators:
\begin{align}
\widetilde{K}^{\pm,1}_{0}(n,m)&=(-1)^{n+m}\int_{0}^{2}e^{i\tilde{\theta}_+(|n|\pm|m|)}\frac{\tilde{\theta}^2_+\chi_3(\mu)}{\mu({\rm sin}\tilde{\theta}_+)^2}\widetilde{L}^{\pm,1}_0(\tilde{\theta}_+,n,m)d\mu,\label{form of K0tutapm1}\\
\widetilde{K}^{\pm,2}_{0}(n,m)&=-(-1)^m\int_{0}^{2}e^{\pm i\tilde{\theta}_+(|m|\pm i|n|)}\frac{a_2(\mu)\tilde{\theta}_+}{\mu^2{\rm sin}\tilde{\theta}_+}\chi_3(\mu)\widetilde{L}^{\pm,2}_0(\mu,n,m)d\mu,\label{form of K0tutapm2}
\end{align}
where
\begin{align*}
\widetilde{L}^{\pm,1}_0(\tilde{\theta}_+,n,m)&=\sum\limits_{m_1,m_2\in\Z}\int_{[0,1]^2}\widetilde{\mcaM}^{(1)}_{0}(N_1,M_2,m_1,m_2)e^{i\tilde{\theta}_+\big(|N_1|-|n|\pm(|M_2|-|m|)\big)}d\rho_1d\rho_2,\\
\widetilde{L}^{\pm,2}_0(\mu,n,m)&=\sum\limits_{m_1,m_2\in\Z}\int_{0}^{1}\widetilde{\mcaM}^{(2)}_{0}(M_2,m_1,m_2)e^{b(\mu)|\widetilde{N}_1|+\tilde{\theta}_{+}|n|\pm i\tilde{\theta}_+(|M_2|-|m|)}d\rho_2.
\end{align*}
Applying the variable substitution to $\widetilde{K}^{\pm,j}_0$ for $j=1,2$
\begin{align}\label{varible substi2}
{\rm cos}\tilde{\theta}_{+}=\frac{\mu^2}{2}-1 \Longrightarrow\  \frac{d\mu}{d\tilde{\theta}_+}=-\frac{{\rm sin}\tilde{\theta}_+}{\mu},\quad \tilde{\theta}_+\rightarrow-\pi\ {\rm as}\ \mu\rightarrow0\ {\rm and}\ \tilde{\theta}_+\rightarrow0\ {\rm as}\  \mu\rightarrow2,
\end{align}
we further obtain
\begin{align*}
\widetilde{K}^{\pm,1}_{0}(n,m)&=(-1)^{n+m}\int_{-\pi}^{0}e^{i\tilde{\theta}_+(|n|\pm |m|)}\frac{\tilde{\theta}^2_+}{{\rm sin}\tilde{\theta}_+}\tilde{\chi}_3(\mu(\tilde{\theta}_+))\widetilde{L}^{\pm,1}_0(\tilde{\theta}_+,n,m)d\tilde{\theta}_{+},\\
\widetilde{K}^{\pm,2}_{0}(n,m)&=(-1)^m\int_{-\pi}^{0}e^{\pm i\tilde{\theta}_+(|m|\pm i|n|)}\tilde{\tilde{\chi}}_3(\mu(\tilde{\theta}_+))\tilde{\theta}_+\widetilde{L}^{\pm,2}_0(\mu(\tilde{\theta}_+),n,m)d\tilde{\theta}_+,
\end{align*}
where $\tilde{\chi}_3(\mu)=-\mu^{-2}\chi_3(\mu)$ and $\tilde{\tilde{\chi}}_3(\mu)=-\mu^{-3}\big(1+\frac{\mu^2}{4}\big)^{-\frac{1}{2}}\chi_3(\mu)$. It is clearly that for any $k=0,1,2$,
\begin{equation*}
\sup\limits_{\tilde{\theta}_{+}\in(-\pi,0)}\big|(\partial^k_{\tilde{\theta}_+}\widetilde{L}^{\pm,1}_{0})(\tilde{\theta}_+,n,m)\big|\lesssim \|\left<\cdot\right>^{2k+2}V(\cdot)\|_{\ell^1},\quad {\rm uniformly\ in}\ n,m\in\Z.
\end{equation*}
Next, we verify that the following estimates hold:
\begin{equation}\label{estimate of Lpm2tuta}
\sup\limits_{\tilde{\theta}_{+}\in[\gamma_1,0)}\big|(\partial^k_{\tilde{\theta}_+}\widetilde{L}^{\pm,2}_{0})(\mu(\tilde{\theta}_+),n,m)\big|\lesssim \|\left<\cdot\right>^{2k+2}V(\cdot)\|_{\ell^1},\quad {\rm uniformly\ in}\ n,m\in\Z,
\end{equation}
where $\gamma_1\in(-\pi,0)$ satisfying ${\rm cos}\gamma_1=\frac{(2-\mu_0)^2}{2}-1$. Combining these estimates with arguments analogous to those used for $K^{\pm,1}_{0}$, we immediately conclude that $\widetilde{K}^{\pm,j}_{0}\in\B(\ell^p(\Z))$ for $1\leq p\leq\infty$ and $j=1,2$.

To see this, we first observe that $b(\mu)<0$ and $b'(\mu)<0$ on the interval $(0,2)$, which implies that for any $k\in\N$,
\begin{align}\label{estimate for NebN}
\sup\limits_{\mu\in[2-\mu_0,2)}|N_1|^ke^{b(\mu)|N_1|}\leq\sup\limits_{N_1}|N_1|^ke^{b(2-\mu_0)|N_1|}<\infty.
\end{align}
This estimate immediately verifies \eqref{estimate of Lpm2tuta} for $k=0$. For the cases $k=1,2$, we can calculate
\begin{align*}
(\partial^k_{\tilde{\theta}_+}\widetilde{L}^{\pm,2}_{0})(\mu(\tilde{\theta}_+),n,m)=\sum\limits_{m_1,m_2\in\Z}\int_{0}^{1}\widetilde{\mcaM}^{(2)}_{0}(M_2,m_1,m_2)\underbrace{\partial^k_{\tilde{\theta}_+}\big(e^{b(\mu(\tilde{\theta}_+))|\widetilde{N}_1|+\tilde{\theta}_{+}|n|\pm i\tilde{\theta}_+(|M_2|-|m|)}\big)}_{:=\mcaL_k(\tilde{\theta}_+,\widetilde{N}_1,n,M_2,m)}d\rho_2
\end{align*}
and
\begin{align*}
\mcaL_1(\tilde{\theta}_+,\widetilde{N}_1,n,M_2,m)&=\big[(b'(\mu(\tilde{\theta}_+))\mu'(\tilde{\theta}_+)+1)|\widetilde{N}_1|+|n|-|\widetilde{N}_1|\pm i(|M_2|-|m|)\big]\\
&\quad\times e^{b(\mu(\tilde{\theta}_+))|\widetilde{N}_1|+\tilde{\theta}_{+}|n|\pm i\tilde{\theta}_+(|M_2|-|m|)},\\
\mcaL_2(\tilde{\theta}_+,\widetilde{N}_1,n,M_2,m)&=\Big(\big[(b'(\mu(\tilde{\theta}_+))\mu'(\tilde{\theta}_+)+1)|\widetilde{N}_1|+|n|-|\widetilde{N}_1|\pm i(|M_2|-|m|)\big]^2\\
&\qquad +(b'(\mu(\tilde{\theta}_+))\mu'(\tilde{\theta}_+))'|\widetilde{N}_1|\Big)\times e^{b(\mu(\tilde{\theta}_+))|\widetilde{N}_1|+\tilde{\theta}_{+}|n|\pm i\tilde{\theta}_+(|M_2|-|m|)}.
\end{align*}
Combining this with \eqref{estimate for NebN}, $\tilde{\theta}_+<0$ and the continuous differentiability of $b'(\mu)$ and $\mu'(\tilde{\theta}_+)$:
$$b'(\mu)=-(2+\mu^2)^{-1}\big((4+\mu^2)^{\frac{1}{2}}+\mu^2(4+\mu^2)^{-\frac{1}{2}}\big),\quad \mu'(\tilde{\theta}_{+})=\big(1-\frac{\mu^2}{4}\big)^{\frac{1}{2}},\quad\mu''(\tilde{\theta}_{+})=-\frac{\mu}{4},$$
we obtain
\begin{equation*}
\sup\limits_{\tilde{\theta}_+\in[\gamma_1,0)}|\mcaL_k(\tilde{\theta}_+,\widetilde{N}_1,n,M_2,m)|\lesssim \left<m_1\right>^k\left<m_2\right>^k, \quad {\rm uniformly\ in}\ \widetilde{N}_1,M_2,n,m.
\end{equation*}
Hence, by H\"{o}lder's inequality, the desired estimate \eqref{estimate of Lpm2tuta} is obtained.
\vskip0.3cm
\underline{{\bf Step 2:}}~For $K=K^0_{r}$, it follows from \eqref{kernel of Kr0} that
\begin{align*}
K^0_r(n,m)=\frac{1}{16}\sum\limits_{j=1}^{2}(K^{+}_{rj}+K^{-}_{rj})(n,m),
\end{align*}
where
\begin{align*}
K^{\pm}_{r1}(n,m)&=(-1)^{n+m}\int_{0}^{2}\frac{-\chi_3(\mu)}{\mu({\rm sin}\tilde{\theta}_+)^2}\sum\limits_{m_1,m_2\in\Z}\widetilde{\Phi}^{\pm}_1(\mu,\widetilde{N}_1,\widetilde{M}_2)(\tilde{v}\Gamma^0_{\frac{3}{2}}(2-\mu)\tilde{v})(m_1,m_2)d\mu,\\
\xlongequal[]{\eqref{varible substi2}}&(-1)^{n+m}\int_{-\pi}^{0}e^{i\tilde{\theta}_+(|n|\pm |m|)}\sum\limits_{m_1,m_2\in\Z}e^{i\tilde{\theta}_+(|\widetilde{N}_1|-|n|\pm(|\widetilde{M}_2|-|m|))}(\tilde{v}\varphi_1(\mu(\tilde{\theta}_+))\tilde{v})(m_1,m_2)d\tilde{\theta}_+,\\
K^{\pm}_{r2}(n,m)&=(-1)^{m}\int_{0}^{2}\frac{-ia_2(\mu)}{\mu^2{\rm sin}\tilde{\theta}_+}\chi_3(\mu)\sum\limits_{m_1,m_2\in\Z}\widetilde{\Phi}^{\pm}_2(\mu,\widetilde{N}_1,\widetilde{M}_2)({v}\Gamma^0_{\frac{3}{2}}(2-\mu)\tilde{v})(m_1,m_2)d\mu,\\
\xlongequal[]{\eqref{varible substi2}}(-1)&^{m}\int_{-\pi}^{0}e^{\pm i\tilde{\theta}_+(|m|\pm i |n|)}\sum\limits_{m_1,m_2\in\Z}e^{b(\mu(\tilde{\theta}_+))|\widetilde{N}_1|+\tilde{\theta}_+|n|\pm i\tilde{\theta}_+(|\widetilde{M}_2|-|m|))}(v\varphi_2(\mu(\tilde{\theta}_+))\tilde{v})(m_1,m_2)d\tilde{\theta}_+,
\end{align*}
with $\Gamma(\mu)=\frac{\Gamma^0_{\frac{3}{2}}(2-\mu)}{(2-\mu)^{\frac{1}{2}}},\ \varphi_1(\mu)=-2\mu^{-3}(2+\mu)^{-\frac{1}{2}}\chi_{3}(\mu)\Gamma(\mu),\   \varphi_2(\mu)=i\mu^{-3}(2-\mu)^{\frac{1}{2}}a_2(\mu)\chi_{3}(\mu)\Gamma(\mu).$
Observe that $\mu'(\tilde{\theta}_+)$ contributes a factor of $(2-\mu)^{\frac{1}{2}}$. Consequently, from \eqref{estimate of Gamma} we obtain that for $\mu\in[2-\mu_0,2)$,
\begin{align*}
\Big\|\frac{d^k(\Gamma(\mu(\tilde{\theta}_+)))}{d\tilde{\theta}_+}\Big\|_{\ell^2\rightarrow\ell^2}\lesssim (2-\mu)^{\frac{2-k}{2}},\quad k=0,1,2.
\end{align*}
This estimate together with argument analogous to case (1) gives that $K^{\pm}_{rj}\in\B(\ell^p(\Z))$ for all $1\leq p\leq\infty$, and so does $K^0_r$.
\vskip0.3cm
\underline{{\bf(2)}}~For $K_{\widetilde{P}_1}$, from \eqref{kernel of KPtuta1} and the expression \eqref{mcaKB regular 2} for $K_B$ with $B=(2-\mu)B^0_{21}$, we have
\begin{align*}
K_{\widetilde{P}_1}(n,m)=\frac{1}{16}\sum\limits_{j=1}^{2}(K^{+,j}_{\widetilde{P}_1}+K^{-,j}_{\widetilde{P}_1})(n,m),
\end{align*}
where
\begin{align*}
K^{\pm,1}_{\widetilde{P}_1}(n,m)&=(-1)^{n+m}\int_{0}^{2}\frac{-\chi_3(\mu)(2-\mu)^{\frac{1}{2}}}{\mu({\rm sin}\tilde{\theta}_+)^2}\sum\limits_{m_1,m_2\in\Z}\widetilde{\Phi}^{\pm}_1(\mu,\widetilde{N}_1,\widetilde{M}_2)(\tilde{v}{\widetilde{P}_1}\tilde{v})(m_1,m_2)d\mu,\\
K^{\pm,2}_{\widetilde{P}_1}(n,m)&=(-1)^{m}\int_{0}^{2}\frac{-ia_2(\mu)(2-\mu)^{\frac{1}{2}}}{\mu^2{\rm sin}\tilde{\theta}_+}\chi_3(\mu)\sum\limits_{m_1,m_2\in\Z}\widetilde{\Phi}^{\pm}_2(\mu,\widetilde{N}_1,\widetilde{M}_2)({v}{\widetilde{P}_1}\tilde{v})(m_1,m_2)d\mu.
\end{align*}
Note that $K^{\pm,2}_{\widetilde{P}_1}=O(1)$ as $\mu\rightarrow2$, this means that through a treatment similar to $\widetilde{K}^{\pm,2}_{0}$, one has $$K^{\pm,2}_{\widetilde{P}_1}(n,m)=O(\left<|n|\pm|m|\right>^{-2}).$$
As for $K^{\pm,1}_{\widetilde{P}_1}$, we first apply the variable substitution \eqref{varible substi2}, and then do the same decomposition \eqref{decom of KpmP1} as $K^{\pm,1}_{P_1}$ in Proposition \ref{proposition K110KP1 regular} obtaining
\begin{align*}
K^{\pm,1}_{\widetilde{P}_1}(n,m)=4(-1)^{n+m}\big(k^{+}_1(n,m)+k^{-}_1(n,m)\big)+O(\left<|n|\pm|m|\right>^{-2}).
\end{align*}
Here we also used the fact that $\sum\limits_{m_1,m_2\in\Z}(\tilde{v}\widetilde{P}_1\tilde{v})(m_1,m_2)=-32i.$
Thus we have
\begin{align}\label{expres of Kptuta1}
K_{\widetilde{P}_1}(n,m)=\frac{(-1)^{n+m}}{4}g_{1,0}(n,m)+O(\left<|n|\pm|m|\right>^{-2}),
\end{align}
where $g_{1,0}(n,m)$ is defined in \eqref{gab hab}. Therefore,
$K_{\widetilde{P}_1}\in\B(\ell^p(\Z))$ for all $1< p<\infty$ by Lemmas \ref{shur test} and \ref{C-Z lemma}, and we complete the whole proof.
\end{proof}
\begin{remark}{\rm Compared with $\mcaK_1$ discussed in Section \ref{sect of mcaK1}, further remarks are given as follows.
\vskip0.15cm
 {\rm(1)}~We remark that both variable substitution \eqref{varible substi1} and \eqref{varible substi2} play important roles in estimating the integral kernels. However, they exhibit slight differences in addressing singularity. Specifically, \eqref{varible substi1} does not alter the singularity near $\mu=0$, whereas \eqref{varible substi2} decreases a singularity of order $(2-\mu)^{-\frac{1}{2}}$.
\vskip0.15cm
{\rm (2)}~Moreover, recalling from \eqref{R0vBvR0} that
\begin{align}\label{new R0vBvR0}
R^{+}_{0}(\mu^4)v Bv(R^{+}_{0}-R^{-}_0)(\mu^4)&=\frac{1}{4\mu^4}\big[JR^{-}_{-\Delta}(4-\mu^2)\tilde{v}B\tilde{v}(R^{-}_{-\Delta}-R^{+}_{-\Delta})(4-\mu^2)J\notag\\
&\qquad\qquad+R_{-\Delta}(-\mu^2){v}B\tilde{v}(R^{-}_{-\Delta}-R^{+}_{-\Delta})(4-\mu^2)J\big],
\end{align}
we observe that the singularity of the second term is always weaker than that of the first term. This differs from the zero resonance case, where both terms exhibit the same singularity. Due to this difference, the second term demonstrates better boundedness at the endpoints \( p = 1 \) and \( p = \infty \), simplifying the endpoint analysis compared to the zero resonance case.
\vskip0.15cm
{\rm (3)}~We notice that the method used for $K^{\pm,2}_0(n,m)$ cannot be applied to the integral kernel corresponding to the second term in \eqref{new R0vBvR0}, since $\tilde{\theta}'_+(\mu)$ becomes singular near $\mu=2$.}
\end{remark}
\subsection{$16$ is a resonance of $H$}\label{subsec of W3 resonance}
In this subsection, we consider the case where $16$ is a resonance of $H$. Taking the expansion
\begin{align*}
M^{-1}(\mu)&=(2-\mu)^{-\frac{1}{2}}\widetilde{S}_0B_{-1}\widetilde{S}_0+\big(\widetilde{S}_0B^{1}_{01}+B^{1}_{02}\widetilde{S}_0+\widetilde{Q}B^{1}_{03}\widetilde{Q}\big)+
(2-\mu)^{\frac{1}{2}}(\widetilde{Q}B^{1}_{11}+B^{1}_{12}\widetilde{Q})\notag\\
&\quad+(2-\mu)^{\frac{1}{2}}\widetilde{P}_1+(2-\mu) B^{1}_{21}+\Gamma^{1}_{\frac{3}{2}}(2-\mu)
\end{align*}
into \eqref{kernel of W3}, then $\mcaK_3$ can be written as 
\begin{equation}\label{decom1 of W3 resonance}
\mcaK_3=\sum\limits_{B\in\mcaB_{11}\cup\mcaB_{12}}^{}K_B+K_{\widetilde{P}_1}+K^1_{r},
\end{equation}
where $K_B,K_{\widetilde{P}_1}$ are defined in \eqref{kernel of KPtuta1}, $\mcaB_{11}=\{(2-\mu)^{-\frac{1}{2}}\widetilde{S}_0B_{-1}\widetilde{S}_0,\widetilde{S}_0B^{1}_{01},B^{1}_{02}\widetilde{S}_0\}$ and
\begin{align*} \mcaB_{12}&=\{\widetilde{Q}B^{1}_{03}\widetilde{Q},(2-\mu)^{\frac{1}{2}}\widetilde{Q}B^{1}_{11},(2-\mu)^{\frac{1}{2}}B^{1}_{12}\widetilde{Q},(2-\mu) B^{1}_{21}\},\notag\\
K^{1}_r(n,m)&=\int_{0}^{2}\mu^3\chi_3(\mu)\big[R^{+}_{0}(\mu^4)v \Gamma^1_{\frac{3}{2}}(2-\mu)v(R^{+}_{0}-R^{-}_0)(\mu^4)\big](n,m)d\mu.
\end{align*}
Recalling from the established result in Proposition \ref{proposition good1 regular 2}, we have
\begin{itemize}
\item $K\in\B(\ell^p(\Z))\  {\rm for\  all\ } 1\leq p\leq\infty$ with $K\in\{K^1_{r}\}\cup\{K_{B}:B\in\mcaB_{12}\}$,
\vskip0.15cm
\item $K_{\widetilde{P}_1}\in\B(\ell^p(\Z))\  {\rm for\  all\ } 1< p<\infty$.
\end{itemize}
Hence, it suffices to focus on the operators $K_{B}$ with $B\in\mcaB_{11}$. Compared to the regular case, these additional terms, while being of the same order $O((2-\mu)^{-\frac{1}{2}})$ as $K_{\widetilde{P}_1}$ in the vicinity of $\mu=2$,
exhibit more subtle behaviors in boundedness analysis which slightly differs in handling. Precisely, 
\begin{proposition}\label{proposition resonance 2}
Let $H=\Delta^2+V$ with $|V(n)|\lesssim \left<n\right>^{-\beta}$ for $\beta>13$. Suppose that $H$ has no positive eigenvalues in the interval $\rm{(}0,16\rm{)}$ and $16$ is a resonance of $H$. Then for any $B\in\mcaB_{11}$, $K_{B}\in\B(\ell^p(\Z))$ for all $1< p<\infty$ and thus $\mcaK_3\in\B(\ell^p(\Z))$ for any $1< p<\infty$.
\end{proposition}
\begin{proof}
\underline{{\bf(1)}}~For $B=(2-\mu)^{-\frac{1}{2}}\widetilde{S}_0B_{-1}\widetilde{S}_0$, combining \eqref{mcaKB regular 2} and the expression of $K_B$, we have
\begin{align*}
\bm{\widetilde{K}_{-1}(n,m)}:&=\int_{0}^{2}(2-\mu)^{-\frac{1}{2}}\mu^3\chi_3(\mu)\big[R^{+}_{0}(\mu^4)v\widetilde{S}_0B_{-1}\widetilde{S}_0v(R^{+}_{0}-R^{-}_0)(\mu^4)\big](n,m)d\mu,\\
&=\frac{1}{16}\sum\limits_{j=1}^{2}(\widetilde{K}^{+,j}_{-1}-\widetilde{K}^{-,j}_{-1})(n,m),
\end{align*}
where $N_1,\widetilde{N}_1,M_2$ are defined in \eqref{mcaKB regular 2} and
\begin{align*}
\widetilde{K}^{\pm,1}_{-1}(n,m)&=\sum\limits_{m_1\in\Z}\int_{0}^{1}(-1)^{n}({\rm sign}(N_1))\int_{0}^{2}\frac{(2-\mu)^{-\frac{1}{2}}\tilde{\theta}^2_+}{\mu({\rm sin}\tilde{\theta}_+)^2}\chi_3(\mu)\sum\limits_{m_2\in\Z}\int_{0}^{1}\widetilde{\Phi}^{\pm}_1(\mu,N_1,M_2)\\
&\quad\times\widetilde{\mcaM}^{(1)}_{-1}(m,M_2,m_1,m_2)d\rho_2d\mu d\rho_1:=\sum\limits_{m_1\in\Z}\int_{0}^{1}(-1)^{n}({\rm sign}(N_1))k^{\pm,1}_{-1}(m_1,\rho_1,n,m)d\rho_1,\notag\\
\widetilde{K}^{\pm,2}_{-1}(n,m)&=\sum\limits_{m_1\in\Z}\int_{0}^{2}\frac{(2-\mu)^{-\frac{1}{2}}a_2(\mu)\tilde{\theta}_+}{-\mu^2{\rm sin}\tilde{\theta}_+}\chi_3(\mu)\sum\limits_{m_2\in\Z}\int_{0}^{1}\widetilde{\Phi}^{\pm}_2(\mu,\widetilde{N}_1,M_2)\widetilde{\mcaM}^{(2)}_{-1}(m,M_2,m_1,m_2)d\rho_2d\mu\\
&:=\sum\limits_{m_1\in\Z}k^{\pm,2}_{-1}(m_1,n,m),\\
\widetilde{\mcaM}^{(1)}_{-1}(m,&M_2,m_1,m_2)=(-1)^m({\rm sign}(M_2))(\tilde{v}_1\widetilde{S}_0B_{-1}\widetilde{S}_0\tilde{v}_1)(m_1,m_2),\\
\widetilde{\mcaM}^{(2)}_{-1}(m,&M_2,m_1,m_2)=(-1)^m({\rm sign}(M_2))({v}\widetilde{S}_0B_{-1}\widetilde{S}_0\tilde{v}_1)(m_1,m_2).
\end{align*}
For any fixed $(m_1,\rho_1)\in\Z\times[0,1]$, first perform the variable substitution \eqref{varible substi2} to $k^{\pm,1}_{-1}(m_1,\rho_1,n,m)$ into the form \eqref{form of K0tutapm1} and $k^{\pm,2}_{-1}(m_1,n,m)$ to the form \eqref{form of K0tutapm2}, and then do the similar decomposition \eqref{decom of KpmP1}, we obtain that
\begin{align*}
k^{\pm,1}_{-1}(m_1,\rho_1,n,m)&=\frac{-i}{2}k^{\pm}_1(n,m)\widetilde{C}_{1}(m_1,m)+O(\widetilde{\mcaM}_{-1}(m_1)\left<|n|\pm|m|\right>^{-2}),\\
k^{\pm,2}_{-1}(m_1,n,m)&=-\frac{\sqrt2}{8}q^{|\widetilde{N}_1|}k^{\mp}_2(n,m)\widetilde{C}_{2}(m_1,m)+O(\widetilde{\mcaM}_{-1}(m_1)\left<|n|\pm|m|\right>^{-2}),
\end{align*}
where $q=3-2\sqrt2$, $\widetilde{\mcaM}_{-1}(m_1)=\left<m_1\right>^3|v(m_1)|\big(|\widetilde{S}_0B_{-1}\widetilde{S}_0|(\left<\cdot\right>^3|v(\cdot)|)\big)(m_1)$ and
\begin{align*}
\widetilde{C}_{j}(m_1,m)=\sum\limits_{m_2\in\Z}\int_{0}^{1}\widetilde{\mcaM}^{(j)}_{-1}(m,M_2,m_1,m_2)d\rho_2.
\end{align*}
Notice that $(-1)^n({\rm sign}(N_1))$, $q^{|\widetilde{N}_1|}$ are uniformly bounded in $n,m_1,\rho_1$, and $|\widetilde{C}_{j}(m_1,m)|\leq\widetilde{\mcaM}_{-1}(m_1)$ uniformly in $m$ for $j=1,2$. It means that $\widetilde{K}_{-1}\in\B(\ell^p(\Z))$ for all $1<p<\infty$ by following the argument as $K_1$ in Proposition \ref{proposition K110KP1 regular}.
\vskip0.3cm
\underline{{\bf(2)}}~For $B=\widetilde{S}_0B^{1}_{01}$, it follows from \eqref{mcaKB regular 2} that
\begin{align*}
\bm{\widetilde{K}_{01}(n,m)}:=\int_{0}^{2}\mu^3\chi_3(\mu)\big[R^{+}_{0}(\mu^4)v\widetilde{S}_0B^1_{01}v(R^{+}_{0}-R^{-}_0)(\mu^4)\big](n,m)d\mu:=\frac{1}{16}\sum\limits_{j=1}^{2}(\widetilde{K}^{+,j}_{01}+\widetilde{K}^{-,j}_{01})(n,m),
\end{align*}
where $\widetilde{\mcaM}^1_{01}(m_1,m_2)=(\tilde{v}_1\widetilde{S}_0B^1_{01}\tilde{v})(m_1,m_2)$ and
\begin{align*}
\widetilde{K}^{\pm,1}_{01}(n,m)&=(-1)^{n+m}\int_{0}^{2}\frac{i\tilde{\theta}_+\chi_3(\mu)}{\mu({\rm sin}\tilde{\theta}_+)^2}\sum\limits_{m_1,m_2\in\Z}\int_{0}^{1}\widetilde{\Phi}^{\pm}_1(\mu,N_1,\widetilde{M}_2)({\rm sign}(N_1))d\rho_1\widetilde{\mcaM}^1_{01}(m_1,m_2)d\mu,\\
\widetilde{K}^{\pm,2}_{01}(n,m)&=(-1)^m\sum\limits_{m_1,m_2\in\Z}(v\widetilde{S}_0B^1_{01}\tilde{v})(m_1,m_2)\int_{0}^{2}\frac{-ia_2(\mu)\chi_3(\mu)}{\mu^2{\rm sin}\tilde{\theta}_+}\widetilde{\Phi}^{\pm}_2(\mu,\widetilde{N}_1,\widetilde{M}_2)d\mu.
\end{align*}
Applying the method for $k^{\pm,1}_{-1}$ to $\widetilde{K}^{\pm,1}_{01}$ directly and $k^{\pm,2}_{-1}$ to the inner integral of $\widetilde{K}^{\pm,2}_{01}$ and using the fact that
\begin{align}\label{sup (N-M)-2}
\sup\limits_{m\in\Z,a,b\in\R}\sum\limits_{n\in\Z}\left<|n-a|\pm|m-b|\right>^{-2}+ \sup\limits_{n\in\Z,a,b\in\R}\sum\limits_{m\in\Z}\left<|n-a|\pm|m-b|\right>^{-2}<\infty,
 \end{align}
we obtain
\begin{align*}
\widetilde{K}_{01}(n,m)&=-\frac{(-1)^{n+m}}{64}\widetilde{C}^{1}_{01}(n)g_{1,0}(n,m)+\frac{(-1)^m\sqrt2i}{256}\sum\limits_{m_1,m_2\in\Z}(v\widetilde{S}_0B^1_{01}\tilde{v})(m_1,m_2)q^{|\widetilde{N}_1|}g_{0,1}(\widetilde{N}_1,\widetilde{M}_2)\\
&\quad+R_{01}(n,m),
\end{align*}
where the integral operator $R_{01}\in\B(\ell^p(\Z))$ for all $1\leq p\leq\infty$ and
\begin{align}\label{Ctuta101}
\widetilde{C}^{1}_{01}(n)=\sum\limits_{m_1,m_2\in\Z}\int_{0}^{1}({\rm sign}(N_1))d\rho_1\widetilde{\mcaM}^1_{01}(m_1,m_2).
\end{align}
Notice that
\begin{align*}
q^{|\widetilde{N}_1|}g_{0,1}(\widetilde{N}_1,\widetilde{M}_2)&={\mathbf1}_{E}(\widetilde{N}_1,\widetilde{M}_2)\frac{2|\widetilde{N}_1|q^{|\widetilde{N}_1|}}{(\widetilde{N}_1)^2+(\widetilde{M}_2)^2}+q^{|\widetilde{N}_1|}g_{0,1}(\widetilde{N}_1,\widetilde{M}_2){\mathbf1}_{E^c}(\widetilde{N}_1,\widetilde{M}_2)\\
&=O(\big<|\widetilde{N}_1|-|\widetilde{M}_2|\big>^{-2}),
\end{align*}
where $\mathbf1_E$ denotes the characteristic function on the set $E=\{(x,y):||x|-|y||\geq2\}$ and $E^c$ corresponds to the complementary set. Moreover, in the second equality we used the uniform boundedness of $|\widetilde{N}_1|q^{|\widetilde{N}_1|}$ and $k^{\pm}_1(n,m)$. This estimate combined with \eqref{sup (N-M)-2}
 yields that
 \begin{align}\label{kernel of Ktuta01}
 \widetilde{K}_{01}(n,m)&=-\frac{(-1)^{n+m}}{64}\widetilde{C}^{1}_{01}(n)g_{1,0}(n,m)+\widetilde{{R}}_{01}(n,m),
 \end{align}
 where the integral operator $\widetilde{{R}}_{01}\in\B(\ell^{p}(\Z))$ for all $1\leq p\leq \infty$. Therefore, combining the uniform boundedness of $\widetilde{C}^{1}_{01}(n)$ and Lemma \ref{C-Z lemma}, we obtain $\widetilde{K}_{01}\in\B(\ell^{p}(\Z))$ for all $1<p< \infty$.
 \vskip0.3cm
 \underline{{\bf(3)}}~For $B=B^1_{02}\widetilde{S}_0$, combining \eqref{mcaKB regular 2} and the method used for $k^{\pm,j}_{-1}$, similarly, we can derive
 \begin{align*}
\bm{\widetilde{K}_{02}(n,m)}&:=\int_{0}^{2}\mu^3\chi_3(\mu)\big[R^{+}_{0}(\mu^4)vB^1_{02}\widetilde{S}_0v(R^{+}_{0}-R^{-}_0)(\mu^4)\big](n,m)d\mu\\
&=-\frac{(-1)^{n+m}}{64}h_{1,0}(n,m)\widetilde{C}_{02}(m)+O(\left<|n|\pm|m|\right>^{-2}),
\end{align*}
where $h_{1,0}(n,m)$ is defined in \eqref{gab hab} and
\begin{align*}
\widetilde{C}_{02}(m)=\sum\limits_{m_1,m_2\in\Z}\int_{0}^{1}({\rm sign}(M_2))d\rho_2\cdot(\tilde{v}B^1_{02}\widetilde{S}_0\tilde{v}_1)(m_1,m_2).
\end{align*}
Thus, $\widetilde{K}_{02}\in\B(\ell^p(\Z))$ for all $1<p<\infty$ and we complete the whole proof.
\end{proof}
\subsection{$16$ is an eigenvalue of $H$}Finally, in this subsection, we deal with the case where $16$ is an eigenvalue of $H$. From the expansion \eqref{asy expan eigenvalue 2}
\begin{align*}
M^{-1}(\mu)&=(2-\mu)^{-1}\widetilde{S}_1B_{-2}\widetilde{S}_1+(2-\mu)^{-\frac{1}{2}}\big(\widetilde{S}_0B_{-1,1}\widetilde{Q}+\widetilde{Q}B_{-1,2}\widetilde{S}_0\big)+\big(\widetilde{Q}B^{2}_{01}+B^{2}_{02}\widetilde{Q}\big)\notag\\
&\quad +(2-\mu)^{\frac{1}{2}}(\widetilde{Q}B^{2}_{11}+B^{2}_{12}\widetilde{Q})+(2-\mu)^{\frac{1}{2}}\widetilde{P}_1+(2-\mu) B^{2}_{21}+\Gamma^{2}_{\frac{3}{2}}(2-\mu),
\end{align*}
then $\mcaK_3$ can be written as 
\begin{equation}\label{decom1 of W3 eigenvalue}
\mcaK_3=\sum\limits_{B\in\mcaB_{21}\cup\mcaB_{22}}^{}K_B+K_{\widetilde{P}_1}+K^2_{r},
\end{equation}
where $\mcaB_{21}=\{(2-\mu)^{-1}\widetilde{S}_1B_{-2}\widetilde{S}_1,(2-\mu)^{-\frac{1}{2}}\widetilde{S}_0B_{-1,1}\widetilde{Q},(2-\mu)^{-\frac{1}{2}}\widetilde{Q}B_{-1,2}\widetilde{S}_0,\widetilde{Q}B^{2}_{01},B^{2}_{02}\widetilde{Q}\}$ and
\begin{align*}
\mcaB_{22}&=\{(2-\mu)^{\frac{1}{2}}\widetilde{Q}B^{2}_{11},(2-\mu)^{\frac{1}{2}}B^{2}_{12}\widetilde{Q},(2-\mu) B^{2}_{21}\},\notag\\
K^{2}_r(n,m)&=\int_{0}^{2}\mu^3\chi_3(\mu)\big[R^{+}_{0}(\mu^4)v \Gamma^2_{\frac{3}{2}}(2-\mu)v(R^{+}_{0}-R^{-}_0)(\mu^4)\big](n,m)d\mu.
\end{align*}
Recall from the established result in Propositions \ref{proposition good1 regular 2} and \ref{proposition resonance 2}, we have
\begin{itemize}
\item $K\in\B(\ell^p(\Z))\  {\rm for\  all\ } 1\leq p\leq\infty$ with $K\in\{K^2_{r}\}\cup\{K_{B}\}_{B\in\mcaB_{22}}$,
\vskip0.15cm
\item $K_B,K_{\widetilde{P}_1}\in\B(\ell^p(\Z))\  {\rm for\  all\ } 1< p<\infty$ with $B\in\mcaB_{21}\setminus\{(2-\mu)^{-1}\widetilde{S}_1B_{-2}\widetilde{S}_1\}$.
\end{itemize}
Hence, it remains to prove that $\widetilde{K}_{-2}\in\B(\ell^p(\Z))$ for all $1< p<\infty$, where
\begin{align*}
\bm{\widetilde{K}_{-2}(n,m)}&:=\int_{0}^{2}(2-\mu)^{-1}\mu^3\chi_3(\mu)\big[R^{+}_{0}(\mu^4)v \widetilde{S}_1B_{-2}\widetilde{S}_1v(R^{+}_{0}-R^{-}_0)(\mu^4)\big](n,m)d\mu.
\end{align*}
To see this, we first note that from Lemma \ref{cancelation lemma16}, $\widetilde{K}_{-2}=O((2-\mu)^{-\frac{1}{2}})$ as $\mu\rightarrow2$. Through a similar argument as $\widetilde{K}_{01}$ in Proposition \ref{proposition resonance 2}, one can obtain
\begin{align}\label{kernel of Ktuta-2}
\widetilde{K}_{-2}(n,m)=\frac{(-1)^{n+m}}{32}\widetilde{C}_{-2}(n)g_{1,0}(n,m)+R_{-2}(n,m),
\end{align}
where the integral operator ${R}_{-2}\in\B(\ell^{p}(\Z))$ for all $1\leq p\leq \infty$ and
\begin{align}\label{Ctuta-2n}
\widetilde{C}_{-2}(n)=\sum\limits_{m_1,m_2\in\Z}\int_{0}^{1}({\rm sign}(N_1))d\rho_1(\tilde{v}_1\widetilde{S}_1B_{-2}\widetilde{S}_1\tilde{v}_2)(m_1,m_2).
\end{align}
This gives the desired result. Therefore, to sum up, we have the following conclusion.
\begin{proposition}\label{proposition eigenvalue 2}
Let $H=\Delta^2+V$ with $|V(n)|\lesssim \left<n\right>^{-\beta}$ for $\beta>17$. Suppose that $H$ has no positive eigenvalues in the interval $\rm{(}0,16\rm{)}$ and $16$ is an eigenvalue of $H$. Then  $\mcaK_{3}\in\B(\ell^p(\Z))$ for all $1< p<\infty$.
\end{proposition}

Hence, this together with Propositions \ref{proposition good1 regular 2} and \ref{proposition resonance 2} completes the whole proof of Theorem \ref{theorem for W3}.
\section{Counterexample for the boundedness at endpoints}\label{sec of counterexample}
In this section, we establish the unboundedness of the wave operators $W_{\pm}$ at endpoints $p=1,\infty$, i.e., Theorem \ref{theorem at endpoints}. 
As before, we focus our analysis on $W_+$. Prior to the proof, we state our strategy. Recall from \eqref{expre of W-} that $W_{+}$ is given by
$$W_+=I-\frac{2}{\pi i}\sum\limits_{j=1}^{3}\mcaK_j.$$
By Theorem \ref{theorem for middle part}, $\mcaK_2$ is always bounded on $\ell^p(\Z)$ for all $1\leq p\leq\infty$. Consequently, the unboundedness of $W_+$ at $p=1$ and $p=\infty$ reduces to analyzing the behaviors of the remaining low energy part $\mcaK_1$ and high energy part $\mcaK_3$. Building on the results from Sections \ref{sect of mcaK1} and \ref{sec of mcaK3}, we can further reduce this analysis to studying the key operator $\mcaK$, whose specific form depends on resonance types, as detailed below.

{\bf Case (I): Assume that $0$ is a regular point of $H$}, then
\begin{align}\label{mcaK for regular}
\mcaK=K_{P_1}+K_1+\left\{\begin{aligned}&K_{\widetilde{P}_1},\ 16\ is\ a\ regular\ point\ of\ H,\\
&K_{\widetilde{P}_1}+\widetilde{K}_{-1}+\widetilde{K}_{01}+\widetilde{K}_{02},\ 16\ is\ a\ resonance\ of\ H,\\
&K_{\widetilde{P}_1}+\sum\limits_{B\in\mcaB_{21}}K_B,\ 16\ is\ an\ eigenvalue\ of\ H,\end{aligned}\right.
\end{align}
where $K_1,K_{P_1},K_{\widetilde{P}_1}$ are defined in \eqref{decom1 of W1 regular}, \eqref{decom1 of W3 regular}, respectively, $\widetilde{K}_{-1}$, $\widetilde{K}_{01}$, $\widetilde{K}_{02}$ are defined in Proposition \ref{proposition resonance 2} and $\mcaB_{21}$ is defined in \eqref{decom1 of W3 eigenvalue}.
\vskip0.2cm
{\bf Case (II): Assume that $0$ is a first kind resonance of $H$}, then
\begin{align}\label{mcaK for 1st}
\mcaK=K_{P_1}+\sum\limits_{A\in\mcaA_{11}}K_A+\left\{\begin{aligned}&K_{\widetilde{P}_1},\ 16\ is\ a\ regular\ point\ of\ H,\\
&K_{\widetilde{P}_1}+\widetilde{K}_{-1}+\widetilde{K}_{01}+\widetilde{K}_{02},\ 16\ is\ a\ resonance\ of\ H,\\
&K_{\widetilde{P}_1}+\sum\limits_{B\in\mcaB_{21}}K_B,\ 16\ is\ an\ eigenvalue\ of\ H,\end{aligned}\right.
\end{align}
where sets $\mcaA_{11}$ and $\mcaB_{21}$ are defined in \eqref{decom1 of mcaK1 1st} and \eqref{decom1 of W3 eigenvalue}, respectively.

\vskip0.2cm
{\bf Case (III): Assume that $0$ is a second kind resonance of $H$}, then
\begin{align}\label{mcaK for 2nd}
\mcaK=K_{P_1}+\sum\limits_{A\in\mcaA_{21}}K_A+\left\{\begin{aligned}&K_{\widetilde{P}_1},\ 16\ is\ a\ regular\ point\ of\ H,\\
&K_{\widetilde{P}_1}+\widetilde{K}_{-1}+\widetilde{K}_{01}+\widetilde{K}_{02},\ 16\ is\ a\ resonance\ of\ H,\\
&K_{\widetilde{P}_1}+\sum\limits_{B\in\mcaB_{21}}K_B,\ 16\ is\ an\ eigenvalue\ of\ H,\end{aligned}\right.
\end{align}
where $\mcaA_{21}$ is defined in \eqref{decom1 of mcaK1 2nd}.

 Throughout this section, we always choose the characteristic functions $f_{N}(n):=\chi_{[-N,N]}(n)$ on the interval $[-N,N]$ with $N\in\N^{+}$ as test functions. The proof of Theorem \ref{theorem at endpoints} will be divided into the following four propositions.
 \begin{proposition}\label{proposition of Counterexample regular}
 Let $H=\Delta^2+V$ with $|V(n)|\lesssim \left<n\right>^{-\beta}$ for $\beta>15$. Suppose that $H$ has no positive eigenvalues in the interval $\rm{(}0,16\rm{)}$ and {\textbf{both ${\bm{0}}$ and ${\bm{16}}$ are regular points of ${\bm{H}}$}}, then 
 \begin{itemize}
 \item[{\rm(1)}]$|(K_{P_1}+K_{\widetilde{P}_1})f_{N}(N+2)|\rightarrow+\infty,N\rightarrow+\infty$ and $(K_{P_1}+K_{\widetilde{P}_1})f_{1}\notin\ell^{1}(\Z)$,
 \vskip0.2cm
 \item[{\rm(2)}]$\sup\limits_{N\in\N^+}\|K_1f_{N}\|_{\ell^{\infty}}<\infty$ and $K_1f_{1}\in\ell^1(\Z)$.
 \end{itemize}
 In particular, $\mcaK=K_{P_1}+K_1+K_{\widetilde{P}_1}$ is neither bounded on $\ell^{\infty}(\Z)$ nor on $\ell^1(\Z)$.
 \end{proposition}
 \begin{proof}
 \underline{{\bf{(1)}}} 
 It follows from \eqref{expres of Kp1} and \eqref{expres of Kptuta1} that
 \begin{align*}
 (K_{P_1}+K_{\widetilde{P}_1})(n,m)&=\Big(\frac{i-1}{8}+\frac{(-1)^{n+m}}{4}\Big)(k^+_1+k^-_1)(n,m)+\frac{i-1}{8}(k^+_2+k^-_2)(n,m)\notag\\
 &\quad+O(\left<|n|\pm|m|\right>^{-2}).
 \end{align*}
 By virtue of the uniform boundedness of $k^{\pm}_{\ell}(n,m)$, we can further decomposition it as
 \begin{align}\label{decom of kpmell}
 k^{\pm}_{\ell}(n,m)=(\mathbf1_{E}+\mathbf1_{E^c})(n,m)k^{\pm}_{\ell}(n,m)=\mathbf1_{E}(n,m)k^{\pm}_{\ell}(n,m)+O(\left<|n|-|m|\right>^{-2}),\quad\ell=1,2,
 \end{align}
 where $\mathbf1_E$ denotes the characteristic function on the set $E=\{(x,y):||x|-|y||\geq2\}$ and $E^c$ corresponds to the complementary set.
 This together with the definition of $\phi$ in Lemma \ref{C-Z lemma} allows us to rewrite the kernel of $K_{P_1}+K_{\widetilde{P}_1}$ as
 \begin{align}\label{expr1 of KP1 plus Kptuta1}
 (K_{P_1}+K_{\widetilde{P}_1})(n,m)&=\Big(\frac{i-1}{8}+\frac{(-1)^{n+m}}{4}\Big)\Big(\frac{1}{|n|+|m|}+\frac{1}{|n|-|m|}\Big)\mathbf1_{E}(n,m)+\frac{(i-1)|n|}{4(n^2+m^2)}\mathbf1_{E}(n,m)\notag\\
&\quad +O(\left<|n|-|m|\right>^{-2})\notag\\
&:=G_1(n,m)+G_2(n,m)+R(n,m).
 \end{align}
 Notice that the integral operator $G_2\in\B(\ell^{\infty}(\Z))$ through the following estimate
 \begin{equation}\label{ell infty bounded of g0a}
\sup\limits_{n\in\Z}\sum\limits_{m\in\Z}|G_2(n,m)|\lesssim\sup\limits_{n\in\Z}\sum\limits_{m\in\Z}\frac{|n|}{n^2+m^2}<\infty.
 \end{equation}
 This observation and $R\in\B(\ell^{p}(\Z))$ for all $1\leq p\leq\infty$, indicates that it suffices to establish
$$|(G_1f_{N})(N+2)|\rightarrow+\infty,N\rightarrow+\infty \ {\rm and}\
(G_1+G_2)f_{1}\notin\ell^1(\Z).$$
 Indeed, for any $N\in\N^+$, a direct calculation yields that
 \begin{align*}
 (G_1f_{N})(N+2)&=\sum\limits_{m=-N}^{N}\Big(\frac{i-1}{8}+\frac{(-1)^{N+2+m}}{4}\Big)\Big(\frac{1}{N+2+|m|}+\frac{1}{N+2-|m|}\Big)\notag\\
 &=\frac{i-1}{4}\sum\limits_{k=2}^{2N+2}\frac{1}{k}+\frac{1}{2}\sum\limits_{k=2}^{2N+2}\frac{(-1)^k}{k},
 \end{align*}
 and thus $ |(G_1f_{N})(N+2)|=+\infty$ as $N\rightarrow+\infty$. For the latter, we have
\begin{align*}
\|(G_1+G_2)f_1\|_{\ell^1(\Z)}&=\sum\limits_{n\in\Z}\big|\sum\limits_{m=-1}^{1}(G_1+G_2)(n,m)\big|\gtrsim\sum\limits_{n=3}^{+\infty}\sum\limits_{m=-1}^{1}\Big(\frac{1}{n+|m|}+\frac{1}{n-|m|}+\frac{2n}{n^2+m^2}\Big)\\
&\geq\sum\limits_{n=3}^{+\infty}\sum\limits_{m=-1}^{1}\Big(\frac{1}{n+|m|}+\frac{1}{n-|m|}\Big)\gtrsim\sum\limits_{k=4}^{+\infty}\frac{1}{k}=+\infty.
\end{align*}
 \underline{{\bf{(2)}}} 
 Recall from the \eqref{expre of K110}, through rewriting $K^{\pm,j}_1(n,m)$ as
 \begin{align*}
 K^{\pm,j}_{1}(n,m)&=\sum\limits_{m_1,m_2\in\Z}(v_1QA_1Qv_1)(m_1,m_2)\int_{[0,1]^2}({\rm sign}(N_1))({\rm sign}(M_2))k^{\pm,j}_{1}(N_1,M_2)d\rho_1d\rho_2,
 \end{align*}
 with \begin{align*}
 k^{\pm,1}_{1}(N_1,M_2)&=\int_{0}^{2}e^{-i\theta_+(|N_1|\pm|M_2|)}\chi_1(\mu)\frac{\theta^2_+}{({\rm sin}\theta_+)^2}d\mu,\\
 k^{\pm,2}_{1}(N_1,M_2)&=\int_{0}^{2}e^{b(\mu)|N_1|\pm i\theta_+|M_2|}\chi_1(\mu)\frac{a_2(\mu)b(\mu)\theta_+}{\mu{\rm sin}\theta_+}d\mu,
 \end{align*}
 and then applying the argument used for $K^{\pm,j}_{P_1}$ to $k^{\pm,j}_{1}$ and the estimate \eqref{sup (N-M)-2}, we have
 \begin{align*}
16K_1(n,m)=\sum\limits_{m_1,m_2\in\Z}(v_1QA_1Qv_1)(m_1,m_2)\int_{[0,1]^2}k_1(N_1,M_2)d\rho_1d\rho_2+R_1(n,m),
 \end{align*}
where $R_1\in\B(\ell^p(\Z))$ for all $1\leq p\leq\infty$ and
$$k_1(N_1,M_2)=({\rm sign}(N_1))({\rm sign}(M_2))h_{i,1}(N_1,M_2).$$
 Utilizing the decomposition \eqref{decom of kpmell} and the estimate \eqref{sup (N-M)-2} again, it further reduces to show that
$$\sup\limits_{N\in\N^+}\|\widetilde{K}_{1}f_N\|_{\ell^{\infty}}<\infty\ {\rm and\ }\widetilde{K}_1f_{1}\in\ell^1(\Z), $$
with
 \begin{align*}
 \widetilde{K}_{1}(n,m)=\sum\limits_{m_1,m_2\in\Z}(v_1QA_{1}Qv_1)(m_1,m_2)\int_{[0,1]^2}\mathbf1_{E}(N_1,M_2)k_{1}(N_1,M_2)d\rho_1d\rho_2.
 \end{align*}
 To see this, for any $N\in\N^+$, we decompose $\Z^2=D_{N}\cup D^c_{N}$ with
 $$D_{N}=\{(m_1,m_2)\in\Z^2~\big||m_1|\geq\frac{N}{2}\ {\rm or}\ |m_2|\geq\frac{N}{2}\},$$
 then for any $n\in\Z$,
 \begin{align*}
 |(\widetilde{K}_{1}f_{N})(n)|&\leq\sum\limits_{m=-N}^{N}\sum\limits_{(m_1,m_2)\in D_N}|(v_1QA_{1}Qv_1)(m_1,m_2)|\int_{[0,1]^2}\big|k_{1}(N_1,M_2)\big|d\rho_1d\rho_2\\
 &\quad+\sum\limits_{(m_1,m_2)\in D^c_N}|(v_1QA_{1}Qv_1)(m_1,m_2)|\int_{[0,1]^2}\Big|\sum\limits_{m=-N}^{N}\mathbf1_{E}(N_1,M_2)k_{1}(N_1,M_2)\Big|d\rho_1d\rho_2\\
 &:=K^{(1)}_{N}(n)+K^{(2)}_{N}(n).
 \end{align*}
 \vskip0.15cm
 \underline{{\bf{(i)}} For $K^{(1)}_{N}(n)$}, by virtue of the uniform boundedness of $k_{1}(N_1,M_2)$, it yields that
 \begin{align*}
 &K^{(1)}_{N}(n)
 \lesssim\sum\limits_{m=-N}^{N}\Big(\sum\limits_{|m_1|\geq\frac{N}{2}}\sum\limits_{m_2\in\Z}+\sum\limits_{m_1\in\Z}\sum\limits_{|m_2|\geq\frac{N}{2}}\Big)|(v_1QA_{1}Qv_1)(m_1,m_2)|\\
 &\quad\lesssim {N}\Big[\sum\limits_{|m_1|\geq\frac{N}{2}}\sum\limits_{m_2\in\Z}\big|(\frac{1}{N}\left<\cdot\right>v_1QA_{1}Qv_1)(m_1,m_2)\big|+\sum\limits_{m_1\in\Z}\sum\limits_{|m_2|\geq\frac{N}{2}}\big|(v_1QA_{1}Qv_1\left<\cdot\right>\frac{1}{N})(m_1,m_2)\big|\Big]\\
 &\quad\lesssim \sum\limits_{m_1,m_2\in\Z}|\big(\left<\cdot\right>v_1QA_{1}Qv_1\left<\cdot\right>\big)(m_1,m_2)|< \infty.
 \end{align*}

 \vskip0.15cm
 \underline{{\bf{(ii)}} For $K^{(2)}_{N}(n)$}, noting that $k_{1}(N_1,M_2)$ is an odd function about $M_2$, then for any $|m_2|\leq\frac{N}{2}$, the sum $\sum\limits_{m=-N}^{N}\mathbf1_{E}(N_1,M_2)k_{1}(N_1,M_2)$ contains at most $2|m_2|$ terms. This combined with the uniform boundedness of $k_{1}(N_1,M_2)$ yields that
 \begin{align*}
 \Big|\sum\limits_{m=-N}^{N}\mathbf1_{E}(N_1,M_2)k_{1}(N_1,M_2)\Big|\lesssim\left<m_2\right>,\quad {\rm uniformly\ in}\ m,m_1,\rho_1,\rho_2,N.
 \end{align*}
 Hence, $K^{(2)}_{N}(n)\lesssim1$ uniformly in $n$ and $N$ and this establishes $\sup\limits_{N\in\N^+}\|\widetilde{K}_{1}f_N\|_{\ell^{\infty}}<\infty$.

 On the other hand, basing on
 \begin{align*}
|\mathbf1_{E}(N_1,M_2)k_{1}(N_1,M_2)|\lesssim |M_2|\left<|N_1|-|M_2|\right>^{-2},
\end{align*}
and the estimate \eqref{sup (N-M)-2}, we have
\begin{align*}
\|\widetilde{K}_{1}f_1\|_{\ell^1}
&\lesssim \sum\limits_{m_1,m_2\in\Z}|(v_1QA_{1}Qv_1)(m_1,m_2)|\int_{[0,1]^2}\sum\limits_{n\in\Z}\sum\limits_{m=-1}^{1}|\mathbf1_{E}(N_1,M_2)k_{1}(N_1,M_2)|d\rho_1d\rho_2\\
&\lesssim\sum\limits_{m_1,m_2\in\Z}\left<m_2\right>|(v_1QA_{1}Qv_1)(m_1,m_2)|<\infty,
\end{align*}
 and this completes the whole proof.
 
 \end{proof}
\begin{proposition}\label{proposition of Counterexample regular+resonance}
 Let $H=\Delta^2+V$ and $V$ be compactly supported. Suppose that $H$ has no positive eigenvalues in the interval $\rm{(}0,16\rm{)}$ and ${\bm{0}}$ {\textbf{is a regular point of}} ${\bm{H}}$. Then the following statements hold:
 \vskip0.15cm
 \noindent{\rm(1)}~if ${\bm{16}}$ {\textbf{is a resonance of}} ${\bm{H}}$, then
 \begin{itemize}
\item for any $K\in\{K_{1},\widetilde{K}_{-1},\widetilde{K}_{02}\}$, $\sup\limits_{N\in\N^+}\|Kf_{N}\|_{\ell^{\infty}}<\infty$ and $Kf_1\in\ell^1(\Z)$,
\item
 $\lim\limits_{N\rightarrow\infty}|(K_{P_1}+K_{\widetilde{P}_1}+\widetilde{K}_{01})f_{N}(N+2)|=\infty$. Moreover, if  $\mcaC_1\neq16(1\mp3\sqrt2)$, then $(K_{P_1}+K_{\widetilde{P}_1}+\widetilde{K}_{01})f_1\notin\ell^1(\Z)$,
 \end{itemize}
 where
 \begin{equation}\label{mcaC1}
\mcaC_1=\sum\limits_{m_1,m_2\in\Z}(\tilde{v}_1\widetilde{S}_0B^1_{01}\tilde{v})(m_1,m_2).
 \end{equation}
 In particular, $\mcaK=K_{P_1}+K_{1}+K_{\widetilde{P}_1}+\widetilde{K}_{-1}+\widetilde{K}_{01}+\widetilde{K}_{02}$ is unbounded on $\ell^{\infty}(\Z)$, and if additionally $\mcaC_1\neq16(1\mp3\sqrt2)$, then $\mcaK$ is unbounded on $\ell^1(\Z)$.
\vskip0.15cm
\noindent{\rm(2)}~If  ${\bm{16}}$ {\textbf{is an eigenvalue of}} ${\bm{H}}$, then
\begin{itemize}
\item for any $K\in\{K_1\}\cup\{K_B:B\in\mcaB_{21}\setminus\{\widetilde{K}_{-2},\widetilde{\widetilde{K}}_{01}\}\}$, $\sup\limits_{N\in\N^+}\|Kf_{N}\|_{\ell^{\infty}}<\infty$\ and $Kf_1\in\ell^1(\Z)$,
\item $\lim\limits_{N\rightarrow\infty}|(K_{P_1}+K_{\widetilde{P}_1}+\widetilde{\widetilde{K}}_{01}+\widetilde{K}_{-2})f_{N}(N+2)|=\infty$. Moreover, if $\mcaC_2\neq16(1\mp3\sqrt2)$, then $(K_{P_1}+K_{\widetilde{P}_1}+\widetilde{\widetilde{K}}_{01}+\widetilde{K}_{-2})f_{1}\not\in\ell^1(\Z)$,
\end{itemize}
where $\widetilde{\widetilde{K}}_{01}=K_B$ with $B=\widetilde{Q}B^{2}_{01}$ and
\begin{equation}\label{mcaC2}
\mcaC_2=\sum\limits_{m_1,m_2\in\Z}(\tilde{v}_1\widetilde{Q}B^2_{01}\tilde{v})(m_1,m_2)-2\sum\limits_{m_1,m_2\in\Z}(\tilde{v}_1\widetilde{S}_1B_{-2}\widetilde{S}_1\tilde{v}_2)(m_1,m_2).
\end{equation}
Therefore, $\mcaK=K_{P_1}+K_{1}+K_{\widetilde{P}_1}+\sum\limits_{B\in\mcaB_{21}}K_B$ is unbounded on $\ell^{\infty}(\Z)$, and if additionally $\mcaC_2\neq16(1\mp3\sqrt2)$, then $\mcaK$ is unbounded on $\ell^1(\Z)$.
 \end{proposition}
 \begin{proof}
 {{\bf(1)}}~\underline{{Step 1:}}~For the first item, combining Proposition \ref{proposition of Counterexample regular}, it remains to prove \begin{equation}\label{regular 0 and resonance 16 positive result}
\sup\limits_{N\in\N^+}\|Kf_{N}\|_{\ell^{\infty}}<\infty\ and\  Kf_1\in\ell^1(\Z),\quad  K=\widetilde{K}_{-1},\widetilde{K}_{02}.
 \end{equation}
 Basing on Proposition \ref{proposition resonance 2}, we reformulate $\widetilde{K}^{\pm,j}_{-1}(n,m)$ and $\widetilde{K}_{02}(n,m)$ as the form of $\widetilde{K}^{\pm,2}_{01}$, then an analogous argument yields that
 \begin{align*}
 16\widetilde{K}_{-1}(n,m)&=(\widetilde{K}^{(1)}_{-1}+\widetilde{K}^{(2)}_{-1})(n,m)+R_{-1}(n,m),\\
 -64\widetilde{K}_{02}(n,m)&=\sum\limits_{m_1,m_2\in\Z}(\tilde{v}B^1_{02}\widetilde{S}_0\tilde{v}_2)(m_1,m_2)\int_{0}^{1}(-1)^{n+\rho_2m_2}k_{02}(\widetilde{N}_1,M_2)d\rho_2+R_{02}(n,m),
 \end{align*}
where $R_{-1},R_{02}\in\B(\ell^p(\Z))$ for $1\leq p\leq\infty$ and
\begin{align*}
\widetilde{K}^{(1)}_{-1}(n,m)&=\sum\limits_{m_1,m_2\in\Z}(\tilde{v}_1\widetilde{S}_0B_{-1}\widetilde{S}_0\tilde{v}_1)(m_1,m_2)\int_{[0,1]^2}^{}(-1)^{n+\rho_2m_2}{k}^{(1)}_{-1}(N_1,M_2)d\rho_1d\rho_2,\\
\widetilde{K}^{(2)}_{-1}(n,m)&=\sum\limits_{m_1,m_2\in\Z}(v\widetilde{S}_0B_{-1}\widetilde{S}_0\tilde{v}_1)(m_1,m_2)\int_{0}^{1}(-1)^{\rho_2m_2}{k}^{(2)}_{-1}(\widetilde{N}_1,M_2)d\rho_2,
\end{align*}
with
\begin{align*}
k_{02}(\widetilde{N}_1,M_2)&=(-1)^{M_2}({\rm sign}(M_2))h_{1,0}(\widetilde{N}_1,M_2),\\
{k}^{(1)}_{-1}(N_1,M_2)&=\frac{-i}{2}(-1)^{M_2}({\rm sign}(N_1))({\rm sign}(M_2))h_{1,0}({N}_1,M_2),\\
{k}^{(2)}_{-1}(\widetilde{N}_1,M_2)&=\frac{\sqrt2}{8}q^{|\widetilde{N}_1|}(-1)^{M_2}({\rm sign}(M_2))h_{0,1}(\widetilde{N}_1,M_2).
\end{align*}
Since all $k_{02},k^{(1)}_{-1},k^{(2)}_{-1}$ are uniformly bounded in $N_1,\widetilde{N}_1,M_2$ and are odd functions about $M_2$, then the argument for ${K}_1$ in Proposition \ref{proposition of Counterexample regular} is valid for $\widetilde{K}_{-1}$ and $\widetilde{K}_{02}$, and thus the desired result \eqref{regular 0 and resonance 16 positive result} is obtained.
\vskip0.3cm
\underline{{Step 2:}} For the second item, from \eqref{expr1 of KP1 plus Kptuta1} and \eqref{kernel of Ktuta01}, we have
\begin{align*}
(K_{P_1}+K_{\widetilde{P}_1}+\widetilde{K}_{01})(n,m)&=\Big(\frac{i-1}{8}+(-1)^{n+m}\Big(\frac{1}{4}-\frac{\widetilde{C}^1_{01}(n)}{64}\Big)\Big)\Big(\frac{1}{|n|+|m|}+\frac{1}{|n|-|m|}\Big)\mathbf1_{E}(n,m)\notag\\
&\quad+\frac{(i-1)|n|}{4(n^2+m^2)}\mathbf1_{E}(n,m)+\Big(O(\left<|n|-|m|\right>^{-2})+\widetilde{R}_{01}(n,m)\Big)\notag\\
&:=G^{(1)}_1(n,m)+G^{(1)}_2(n,m)+R^{(1)}(n,m).
\end{align*}
It suffices to show that \begin{equation}\label{regular 0 and resonance 16 negative results}
|(G^{(1)}_1f_N)(N+2)|\rightarrow\infty, N\rightarrow\infty,\ and\ (G^{(1)}_1+G^{(1)}_2)f_1\notin\ell^1(\Z)\ if\ \mcaC_1\neq16(1\mp3\sqrt2). \end{equation}
Since $V$ is compactly supported, that is, there exists an integer $N_0\in\N^+$, such that $suppV\subseteq\{m:|m|\leq N_0\}$. Now take $N>N_0+2$, by \eqref{Ctuta101}, we have
\begin{equation*}
\widetilde{C}^1_{01}(N+2)=\sum\limits_{m_1,m_2\in\Z}\widetilde{\mcaM}^1_{01}(m_1,m_2)=\mcaC_1<\infty.
\end{equation*}
This means that by the argument as (2) in Proposition \ref{proposition of Counterexample regular}, we can derive
$$(G^{(1)}_1f_N)(N+2)=\frac{i-1}{4}\sum\limits_{k=2}^{2N+2}\frac{1}{k}+\Big(\frac{1}{2}-\frac{\mcaC_1}{32}\Big)\sum\limits_{k=2}^{2N+2}\frac{(-1)^k}{k}\rightarrow\infty,\quad N\rightarrow\infty.$$
On the other hand,
\begin{align*}
\|(G^{(1)}_1+G^{(1)}_2)f_1\|_{\ell^1(\Z)}&\geq\sum\limits_{n=N_0+2}^{+\infty}\left|\sum\limits_{m=-1}^{1}\Big[\big(a+c(-1)^{n+m}\big)\Big(\frac{1}{n+|m|}+\frac{1}{n-|m|}\Big)+\frac{2bn}{n^2+m^2}\Big]\right|\\
&:=\sum\limits_{n=N_0+2}^{+\infty}\big|G_{a,b,c}(n)\big|,
\end{align*}
where the coefficients $a,b,c$ are defined as follows:
\begin{equation*}
a=\frac{i-1}{8}=b,\quad c=\frac{1}{4}-\frac{1}{64}\mcaC_1.
\end{equation*}
A direct calculation yields that
\begin{align*}
G_{a,b,c}(n)=\frac{\big(6(a+b)-2c(-1)^n\big)n^4+4((a-b-c(-1)^n)n^2-2(a+b+c(-1)^n)}{n(n^2+1)(n^2-1)}.
\end{align*}
By means of the triangle inequality and the condition that $\mcaC_1\neq16(1\mp3\sqrt2)$, that is, $|3(a+b)|\neq|c|$, then we have
\begin{align*}
\|(G^{(1)}_1+G^{(1)}_2)f_1\|_{\ell^1(\Z)}&\geq\sum\limits_{n=N_0+2}^{+\infty}\big|G_{a,b,c}(n)\big|\gtrsim \sum\limits_{n=N_0+2}^{+\infty}\Big||3(a+b)|-|c|\Big|\cdot\frac{1}{n}+C^{'}=+\infty,
\end{align*}
where $C^{'}$ is a constant. This completes the proof of \eqref{regular 0 and resonance 16 negative results}.
\vskip0.3cm
{{\bf(2)}}~For the former, by the definition of $\mcaB_{21}$, since $\widetilde{Q}$ has the same cancelation property as $\widetilde{S}_{0}$, the proof is completely same as the Step 1 in (1) above apart from the difference in the notation. For the latter, from the expression \eqref{kernel of Ktuta-2} of $\widetilde{K}_{-2}(n,m)$ and the definition of $\widetilde{\widetilde{K}}_{01}$,  we have
\begin{align*}
&(K_{P_1}+K_{\widetilde{P}_1}+\widetilde{\widetilde{K}}_{01}+\widetilde{K}_{-2})(n,m)\\
&=\Big(\frac{i-1}{8}+(-1)^{n+m}\Big(\frac{1}{4}-\frac{\widetilde{C}^2_{01}(n)}{64}+\frac{\widetilde{C}_{-2}(n)}{32}\Big)\Big)\Big(\frac{1}{|n|+|m|}+\frac{1}{|n|-|m|}\Big)\mathbf1_{E}(n,m)\\
&\quad+\frac{(i-1)|n|}{4(n^2+m^2)}\mathbf1_{E}(n,m)+R^{(2)}(n,m),
\end{align*}
where $R^{(2)}\in\B(\ell^p(\Z))$ for all $1\leq p\leq\infty$, $\widetilde{C}_{-2}(n)$ is defined in \eqref{Ctuta-2n} and $\widetilde{C}^2_{01}(n)$ is defined in \eqref{Ctuta101} by replacing $\widetilde{S}_0B^1_{01}$ with $\widetilde{Q}B^2_{01}$. Using an analogue argument as Step 2 in (1), the desired results can be also derived, for brevity, we omit the details and finish the whole proof.
 \end{proof}
This proposition, together with Proposition \ref{proposition of Counterexample regular} thus gives the proof of Theorem \ref{theorem at endpoints} for the case where $0$ is a regular point of $H$. Next, we turn to the remaining two resonant cases.
\begin{proposition}\label{proposition of Counterexample resonance+all}
 Let $H=\Delta^2+V$ and $V$ be compactly supported. Suppose that $H$ has no positive eigenvalues in the interval $\rm{(}0,16\rm{)}$ and ${\bm{0}}$ {\textbf{is a first kind resonance of}} ${\bm{H}}$. Then the following statements hold:
 \vskip0.15cm
  \noindent{\rm (1)}~If $\mcaC_3\neq0$, then for any $\mcaK$ defined in \eqref{mcaK for 1st}, $\mcaK$ is unbounded on $\ell^{\infty}(\Z)$,
 where the constant $\mcaC_3$ is given by
 \begin{align}\label{C}
 \mcaC_3=\frac{i-1}{8}-\frac{i}{64}C_{-1}+\frac{1}{32}C_{02}+\frac{iC_{11}}{32}+\frac{iC_{12}}{32}-\frac{C_{21}}{16}-\frac{iC_{33}}{16}
 \end{align}
 with $C_{-1},C_{11},C_{12},C_{33}$ defined in \eqref{C-1} and \eqref{kernel of 1st}, respectively and
 \begin{align*}
 C_{02}=\sum\limits_{m_1,m_2\in\Z}(v_1QA^1_{02}S_0v_2)(m_1,m_2),\quad C_{21}=\sum\limits_{m_1,m_2\in\Z}(v_1QA^1_{21}v)(m_1,m_2).
 \end{align*}
\vskip0.15cm
\noindent{\rm (2)} Let $C_{02},C_{11},C_{12},C_{21},C_{33}$ be as in \eqref{C} and $\mcaC_1,\mcaC_2$ be as in \eqref{mcaC1} and \eqref{mcaC2}, respectively. Define
\begin{equation}\label{D}
D=\frac{i-1}{4}+\frac{i+1}{32}C_{02}+\frac{i-1}{32}C_{11}+\frac{i}{16}C_{12}-\frac{i+1}{16}C_{21}-\frac{i}{8}C_{33}.
\end{equation}
Under the condition that
\begin{align}\label{condition on p=1 for 1st}
 192|D|\neq\left\{\begin{aligned}&16,\ 16\ is\ a\ regular\ point\ of\ H,\\
&\big|16-\mcaC_1\big|,\ 16\ is\ a\ resonance\ of\ H,\\
&\big|16-\mcaC_2\big|,\ 16\ is\ an\ eigenvalue\ of\ H,\end{aligned}\right.
\end{align}
the corresponding $\mcaK$ defined in \eqref{mcaK for 1st} is unbounded on $\ell^{1}(\Z)$.
 \end{proposition}
\begin{proof}
When $0$ is a first kind resonance of $H$, recalling from \eqref{mcaK for 1st} that
\begin{align*}
\mcaK=K_{P_1}+\sum\limits_{A\in\mcaA_{11}}K_A+\left\{\begin{aligned}&K_{\widetilde{P}_1},\ 16\ is\ a\ regular\ point\ of\ H,\\
&K_{\widetilde{P}_1}+\widetilde{K}_{-1}+\widetilde{K}_{01}+\widetilde{K}_{02},\ 16\ is\ a\ resonance\ of\ H,\\
&K_{\widetilde{P}_1}+\sum\limits_{B\in\mcaB_{21}}K_B,\ 16\ is\ an\ eigenvalue\ of\ H.\end{aligned}\right.
\end{align*}
Based on Propositions \ref{proposition of Counterexample regular} and \ref{proposition of Counterexample regular+resonance} above, we first observe the following two facts.
\vskip0.2cm
\begin{itemize}
\item  {\textbf {\underline{For the operator 
in the low energy part $\mcaK_1$}}}, recalling the operators $\{K_A:A\in\mcaA_{11}\}$
\end{itemize}
 from \eqref{decom1 of mcaK1 1st} and \eqref{kernel of KA 1st}, we have
\begin{equation}\label{the RO of 1st}
K_{P_1}+\sum\limits_{A\in\mcaA_{11}}K_A=K_{P_1}+K_{-1}+\sum\limits_{j=1}^{2}(K_{0j}+K_{2j})+\sum\limits_{j=1}^{3}K_{1j}+K_{33},
\end{equation}
where $K_{13}$ is the integral operator with the kernel as \eqref{kernel of K110} by replacing $A_1$ with $A^1_{13}$. By applying the method used for $K_1$ in Proposition \ref{proposition of Counterexample regular} to $K=K_{01},K_{13},K_{22}$, we conclude that
\begin{equation*}
\sup\limits_{N\in\N^+}\|Kf_{N}\|_{\ell^{\infty}}<\infty\ and\  Kf_1\in\ell^1(\Z),\quad  K=K_{01},K_{13},K_{22}.
\end{equation*}
\begin{itemize}
\item  {\textbf{\underline{For the operator in the high energy part $\mcaK_3$}}} , we have
\end{itemize}
\begin{equation*}
\sup\limits_{N\in\N^+}\|Kf_{N}\|_{\ell^{\infty}}<\infty\ and\  Kf_1\in\ell^1(\Z),
\end{equation*}
where
\begin{align*}
 K\in\left\{\begin{aligned}&\emptyset,\ 16\ is\ a\ regular\ point\ of\ H,\\
&\{\widetilde{K}_{-1},\widetilde{K}_{02}\},\ 16\ is\ a\ resonance\ of\ H,\\
&\{K_B:B\in\mcaB_{21}\setminus\{\widetilde{K}_{-2},\widetilde{\widetilde{K}}_{01}\}\},\ 16\ is\ an\ eigenvalue\ of\ H.\end{aligned}\right.
\end{align*}
Denote
\begin{equation*}
\mcaK_0:=K_{P_1}+K_{-1}+K_{02}+K_{21}+\sum\limits_{j=1}^{2}K_{1j}+K_{33}.
\end{equation*}
Then the analysis of $\mcaK$ above reduces to $\mcaK_r$, where
\begin{align}\label{remain operator for 1st}
 \mcaK_r=\left\{\begin{aligned}&\mcaK_0+K_{\widetilde{P}_1},\ 16\ is\ a\ regular\ point\ of\ H,\\
&\mcaK_0+\widetilde{K}_{01},\ 16\ is\ a\ resonance\ of\ H,\\
&\mcaK_0+\widetilde{K}_{-2}+\widetilde{\widetilde{K}}_{01},\ 16\ is\ an\ eigenvalue\ of\ H.\end{aligned}\right.
\end{align}
By a similar argument to that used in part (2) of Proposition \ref{proposition of Counterexample regular+resonance}, the desired conclusion follows.
\end{proof} \begin{proposition}\label{proposition of Counterexample 2nd resonance+all}
 Let $H=\Delta^2+V$ and $V$ be compactly supported. Suppose that $H$ has no positive eigenvalues in the interval $\rm{(}0,16\rm{)}$ and ${\bm{0}}$ {\textbf{is a second kind resonance of}} ${\bm{H}}$. Then the following statements hold:
 \vskip0.15cm
  \noindent{\rm (1)}~If $\mcaC_4\neq0$, then for any $\mcaK$ defined in \eqref{mcaK for 2nd},  $\mcaK$ is unbounded on $\ell^{\infty}(\Z)$,
 where the constant $\mcaC_4$ is given by
 \begin{align}\label{C4}
 \mcaC_4=\frac{i-1}{8}-\frac{i}{64}C_{-1,3}+\frac{1}{32}C_{03}+\frac{iC^{*}_{11}}{32}+\frac{iC^{*}_{12}}{32}-\frac{C^{*}_{21}}{16}-\frac{iC^{*}_{33}}{16}-\frac{C_{-2,1}}{32}+\frac{C^{(2)}_{01}}{32}
 \end{align}
 with $C^{*}_{ij}$ defined in \eqref{kernel of 1st} by replacing $A^1_{ij}$ with $A^{2}_{ij}$ and
 \begin{align*}
 C_{-1,3}&=\sum\limits_{m_1,m_2\in\Z}(v_2S_0A_{-1,3}S_0v_2)(m_1,m_2),\quad C_{03}=\sum\limits_{m_1,m_2\in\Z}(v_1QA^2_{03}S_0v_2)(m_1,m_2),\\
 C_{-2,1}&=\frac{1}{6}\sum\limits_{m_1,m_2\in\Z}(v_3S_2A_{-2,1}S_0v_2)(m_1,m_2),\quad C^{(2)}_{01}=\frac{1}{3}\sum\limits_{m_1,m_2\in\Z}(v_3S_2A^2_{01}v)(m_1,m_2).
 \end{align*}
\vskip0.15cm
\noindent{\rm (2)} Let $C_{-1,3},C_{03},C^{*}_{11},C^{*}_{12},C^{*}_{21},C^{*}_{33},C_{-2,1},C^{(2)}_{01}$ be as in \eqref{C4} and $\mcaC_1,\mcaC_2$ be as in \eqref{mcaC1} and \eqref{mcaC2}, respectively. Define
\begin{equation}\label{E}
E=\frac{i-1}{4}+\frac{i+1}{32}C_{03}+\frac{i-1}{32}C^{*}_{11}+\frac{i}{16}C^*_{12}-\frac{i+1}{16}C^*_{21}-\frac{i}{8}C^*_{33}.+\frac{i-1}{32}C_{-2,1}+\frac{1-i}{32}C^{(2)}_{01}.
\end{equation}
Under the condition that
\begin{align}\label{condition on p=1 for 1st}
 192|E|\neq\left\{\begin{aligned}&16,\ 16\ is\ a\ regular\ point\ of\ H,\\
&\big|16-\mcaC_1\big|,\ 16\ is\ a\ resonance\ of\ H,\\
&\big|16-\mcaC_2\big|,\ 16\ is\ an\ eigenvalue\ of\ H,\end{aligned}\right.
\end{align}
the corresponding $\mcaK$ defined in \eqref{mcaK for 2nd} is unbounded on $\ell^{1}(\Z)$.
 \end{proposition}
 \begin{proof}
Compared \eqref{mcaK for 1st} with \eqref{mcaK for 2nd}, the difference lies in the part $\{K_A:A\in\mcaA_{21}\}$, where $A_{21}=\mcaA^{(1)}_{21}\cup\mcaA^{(2)}_{21}$ by \eqref{decom1 of mcaK1 2nd}.
 Note that the operators $\{K_A:A\in\mcaA^{(2)}_{21}\}$ essentially the same as $\{K_A:A\in\mcaA_{11}\}$, apart from the difference in the notation. Therefore, in this case, more attention should be paid to the additional operators $\{K_A:A\in\mcaA^{(1)}_{21}\}$, compared to the first kind resonant case.

 From Proposition \ref{proposition 2nd} and \eqref{kernel of KA 2nd}, we further obtain
 \begin{equation*}
\sum\limits_{A\in\mcaA^{(1)}_{21}}K_A=K_{-3}+\sum\limits_{j=1}^{2}(K_{-2,j}+K_{-1,j}+K^{(2)}_{0j}).
 \end{equation*}
It can be observed that the following terms in the integral kernels appear newly compared to the previous two cases:
\begin{itemize}
\item $\mcaT_1(n,m)=\big<(S_2A_{-3}S_2\varphi_m)(\cdot),|n-\cdot|v(\cdot)\big>h_{0,-1}(n,m)$,
\item $\mcaT_2(n,m)=\big<(S_2A_{-2,1}S_0v_2)(\cdot),|n-\cdot|v(\cdot)\big>g_{0,-i}(n,m)$,
\item $\mcaT_3(n,m)=\big<(S_2A_{-1,1}Q\widetilde{\varphi}_m)(\cdot),|n-\cdot|v(\cdot)\big>h_{0,1}(n,m)$,
\item $\mcaT_4(n,m)=\big<(S_2A^2_{01}v)(\cdot),|n-\cdot|v(\cdot)\big>g_{0,i}(n,m).$
\end{itemize}

Obviously, it follows from \eqref{decom of kpmell}, \eqref{ell infty bounded of g0a} and the uniform boundedness of the inner product that
$\mcaT_2$ and $\mcaT_4$ are $\ell^{\infty}$ bounded. More interestingly, under the assumption that 
$suppV\subseteq\{m:|m|\leq N_0\}$ for some integer $N_0$, when we consider the characteristic function as test function, we can prove that for $N>N_0$,
\begin{equation*}
    \mcaT_jf_{N}(N+2)=0 \ {\rm and}\ \| \mcaT_jf_1\|_{\ell^1}<\infty,\quad 1\leq j\leq4.
\end{equation*}
To see this, we consider $\mcaT_1$ only for simplicity. When $ \pm n>N_0$, note that
$$\big<(S_2A_{-3}S_2\varphi_m)(\cdot),|n-\cdot|v(\cdot)\big>h_{0,-1}(n,m)=\pm\big<(S_2A_{-3}S_2\varphi_m)(\cdot),(n-\cdot)v(\cdot)\big>h_{0,-1}(n,m)=0,$$
where the last equality follows from the orthogonality $\big<S_2f,v_j\big>=0$ for $j=0,1$.
This immediately yields that $\mcaT_1f_{N}(N+2)=0$ for $N>N_0$. Regarding $\|T_1f_1\|_{\ell^1}$, using the uniform boundedness of $h_{0,-1}(n,m)$ and $\varphi_{m}$, we obtain
\begin{equation*}
\|\mcaT_1f_1\|_{\ell^1}\leq\Big(\sum\limits_{|n|\leq N_0}+\sum\limits_{|n|>N_0}\Big) \sum\limits_{m=-1}^{1}|\mcaT_{1}(n,m)|=\sum\limits_{|n|\leq N_0} \sum\limits_{m=-1}^{1}|\mcaT_{1}(n,m)|<\infty.
\end{equation*}
This shows that these newly emerged terms behave well under such test functions. A similar analysis as in the first kind resonance case can therefore be applied, for brevity, we omit the details.
 \end{proof}
 Summing up Propositions \ref{proposition of Counterexample regular}$\sim$\ref{proposition of Counterexample 2nd resonance+all}, we consequently complete the whole proof of Theorem \ref{theorem at endpoints}.
 \section{Application}\label{sec of application}
  As an application of Theorem \ref{main theorem}, in this section, we will establish the $\ell^p-\ell^{p'}$ decay estimates for the solution to the discrete beam equation with parameter $a\in\R$ on the lattice $\Z$:
\begin{equation*}
\left\{\begin{aligned}&(\partial_{tt}u)(t,n)+[(\Delta^2+V+a^2)u](t,n)=0,\ \ (t,n)\in\R\times\Z,\\
&u(0,n)=\varphi_1(n),\ \left(\partial_{t}u\right)(0,n)=\varphi_2(n),\end{aligned}\right.
\end{equation*}
whose solution can be expressed as
\begin{equation*}
u_a(t,n)={\rm cos}(t\sqrt {\Delta^2+V+a^2})\varphi_1(n)+\frac{{\rm sin}(t\sqrt {\Delta^2+V+a^2})}{\sqrt {\Delta^2+V+a^2}}\varphi_2(n).
\end{equation*}
More precisely, we have
 \begin{theorem}\label{free decay}
Let $H=\Delta^2+V$ satisfy the assumptions of Theorem \ref{main theorem}. Let $1<p\leq2$ and $\frac{1}{p}+\frac{1}{p'}=1$. Then for any $a\in\R$,
\begin{equation}\label{cos-sin decay-estimate}
\|{\rm cos}(t\sqrt {H+a^2})P_{ac}(H)\|_{\ell^p\rightarrow\ell^{p'}}+\left\|\frac{{\rm sin}(t\sqrt {H+a^2})}{t\sqrt {H+a^2}}P_{ac}(H)\right\|_{\ell^p\rightarrow\ell^{p'}}\lesssim|t|^{-\frac{1}{3}(\frac{1}{p}-\frac{1}{p'})},\quad t\neq0.
\end{equation}
 \end{theorem}
To derive this theorem, using the interwining property \eqref{interwinning property} and the $\ell^{p'}$ boundedness of $W_{\pm}$, for any $a\in\R$ and $j=1,2$, we obtain
\begin{equation*}
\|f_{a,j}(H)P_{ac}(H)\|_{\ell^p\rightarrow\ell^{p'}}\leq\|W_{\pm}\|_{\ell^{p'}\rightarrow\ell^{p'}}\|f_{a,j}(\Delta^2)\|_{\ell^p\rightarrow\ell^{p'}}\|W^*_{\pm}\|_{\ell^p\rightarrow\ell^p}\lesssim \|f_{a,j}(\Delta^2)\|_{\ell^p\rightarrow\ell^{p'}},
\end{equation*}
where
$$f_{a,1}(x)={\rm cos}(t\sqrt {x+a^2}),\quad f_{a,2}(x)=\frac{{\rm sin}(t\sqrt {x+a^2})}{t\sqrt {x+a^2}}.$$
Consequently, it reduces to establish the corresponding estimates for the free propagators $f_{a,j}(\Delta^2)$ with $j=1,2$. To this end, it suffices to establish the following $\ell^1-\ell^{\infty}$ decay estimate.
\begin{lemma}\label{lemma of free cos-sin estimate}
    For any $a\in\R$ and $t\neq0$, we have
   \begin{equation}\label{cos sin free 1-infty estimate}
    \big\|{\rm cos}(t\sqrt {\Delta^2+a^2})\big\|_{\ell^1\rightarrow\ell^{\infty}}+\left\|\frac{{\rm sin}(t\sqrt {\Delta^2+a^2})}{t\sqrt {\Delta^2+a^2}}\right\|_{\ell^1\rightarrow\ell^{\infty}}\lesssim |t|^{-\frac{1}{3}}.
\end{equation}
\end{lemma}
Once this lemma is proved, based on $\|e^{-it\sqrt {\Delta^2+a^2}}\|_{\ell^2\rightarrow\ell^{2}}=1$ and the relations
\begin{equation}\label{relation cos-sin and eit}
{\rm cos}(t\sqrt {\Delta^2+a^2})=\frac{e^{-it\sqrt {\Delta^2+a^2}}+e^{it\sqrt {\Delta^2+a^2}}}{2},\quad \frac{{\rm sin}(t\sqrt {\Delta^2+a^2})}{t\sqrt {\Delta^2+a^2}}=\frac{1}{2t}\int_{-t}^{t}{\rm cos}\left(s\sqrt {\Delta^2+a^2}\right)ds,
\end{equation}
the desired \eqref{cos-sin decay-estimate} for the free case then follows by the Riesz-Thorin interpolation theorem.
\begin{remark}
{\rm We point out that the sharp decay estimate $|t|^{-\frac{1}{3}}$ is not affected by the values of parameter $a$, which is quite different from its continuous counterpart where it is influenced by $a$. For instance, the continuous analogue of \eqref{cos sin free 1-infty estimate} exhibits a decay rate of $|t|^{-\frac{1}{2}}$ when $a=0$, whereas for $a=1$, the decay is $|t|^{-\frac{1}{4}}$ in the low-energy part and $|t|^{-\frac{1}{2}}$ in the high-energy part. For more details, we refer to \cite{CWY25}.

}
\end{remark}
\begin{proof}[Proof of Lemma \ref{lemma of free cos-sin estimate}]
For any $a\in\R$, from \eqref{relation cos-sin and eit} above, the problem reduces to proving
\begin{equation}\label{esti for free}
    \big\|e^{-it\sqrt {\Delta^2+a^2}}\big\|_{\ell^1\rightarrow\ell^{\infty}}\lesssim |t|^{-\frac{1}{3}},\quad t\neq0.
\end{equation} When $a=0$, since such sharp $\ell^1-\ell^{\infty}$ decay estimate was established in \cite{SK05}, here we focus on the case $a\neq0$. Indeed, by virtue of Fourier transform \eqref{fourier transform}, the kernel of $e^{-it\sqrt {\Delta^2+a^2}}$ is given by
\begin{equation*}
   \big(e^{-it\sqrt {\Delta^2+a^2}}\big)(n,m)=(2\pi)^{-\frac{1}{2}}\int_{-\pi}^{\pi}e^{-it\sqrt{(2-2{\rm cos}\theta)^2+a^2}}e^{i(n-m)\theta}d\theta.
\end{equation*}
We claim that the following estimate holds:
\begin{equation}\label{oscillatory}
  \sup\limits_{s\in\R}\Big|\int_{-\pi}^{\pi}e^{-it\big[\sqrt{(2-2{\rm cos}\theta)^2+a^2}-s\theta\big]}d\theta\Big|\lesssim |t|^{-\frac{1}{3}},\quad t\neq0.
\end{equation}
To establish this estimate, it suffices to consider the interval $[-\pi,0]$, as the estimate on $[0,\pi]$ follows by the change of variable $\theta\mapsto-\theta$. For any $s\in\R$, we define
\begin{equation*}
    \Phi_{a,s}(\theta)=\sqrt{(2-2{\rm cos}\theta)^2+a^2}-s\theta,\quad \theta\in[-\pi,0].
\end{equation*}
A direct computation yields that
\begin{equation*}
  \Phi'_{a,s}(\theta)=4\big((2-2{\rm cos}\theta)^2+a^2\big)^{-\frac{1}{2}}(1-{\rm cos}\theta){\rm sin}\theta-s
  \end{equation*}
  and
  \begin{equation*}
  \Phi''_{a,s}(\theta)=4\big((2-2{\rm cos}\theta)^2+a^2\big)^{-\frac{3}{2}}(1-{\rm cos}\theta)\big(4{\rm cos^3}\theta-8{\rm cos^2}\theta+(2a^2+4){\rm cos}\theta+a^2\big).
\end{equation*}
Let
$$h_{a}(x):=4x^3-8x^2+(2a^2+4)x+a^2,\quad x\in[-1,1].$$
We observe that $h_{a}(x)>0$ for $x\geq0$, $h_{a}(-1)=-a^2-16<0$ and $h'_{a}(x)>0$ for $x<0$. Let $x_0$ denote the unique root of $h_{a}(x)$ in the interval $[-1,1]$. Then
$$\Phi''_{a,s}(\theta)=0\Leftrightarrow\theta=0\ {\rm or}\ \theta=\theta_0\in(-\pi,-\frac{\pi}{2}),\quad {\rm where\ cos}\theta_0=x_0.$$
This implies that $\Phi'_{a,s}(\theta)$ is monotonically decreasing on $[-\pi,\theta_0]$ and increasing on $[\theta_0,0]$. Combining this with $\Phi'_{a,s}(-\pi)=-s=\Phi'_{a,s}(0)$, we conclude that for any $s\in\R$, the equation $\Phi'_{a,s}(\theta)=0$ has at most two solutions on $[-\pi,0]$. By Van der Corput lemma (see e.g. \cite[P. ${332-334}$]{Ste93}), the slower decay rates of the oscillatory integral \eqref{oscillatory} on $[-\pi,0]$ occur in the cases of $s=0$ and $s=s_0$, and for the other values of $s$, the decay rate is either $|t|^{-1}$ or $|t|^{-\frac{1}{2}}$, where $s_0=4\big((2-2{\rm cos}\theta_0)^2+a^2\big)^{-\frac{1}{2}}(1-{\rm cos}\theta_0){\rm sin}\theta_0$.
\vskip0.1cm
If $s=0$, then
\begin{equation*}
    \Phi'_{a,0}(\theta)=0\Leftrightarrow\theta=0\ {\rm or}\ \theta=-\pi.
\end{equation*}
Moreover, we can compute
\begin{equation*}
    \Phi''_{a,0}(-\pi)\neq0,\quad  \Phi''_{a,0}(0)=0\ {\rm but}\ \Phi^{(3)}_{a,0}(0)\neq0,
\end{equation*}
thus by Van der Corput lemma, the decay rate is $|t|^{-\frac{1}{3}}$.
\vskip0.1cm
If $s=s_0$, then $\Phi'_{a,s_0}(\theta)=0\Leftrightarrow\theta=\theta_0.$ And $\Phi''_{a,s_0}(\theta_0)=0$ but $\Phi^{(3)}_{a,s_0}(\theta_0)\neq0$, then the decay rate is $|t|^{-\frac{1}{3}}$. In summary, this completes the proof of \eqref{oscillatory}, from which \eqref{esti for free} follows.
\end{proof}

\appendix
\section{Discrete Calder\'{o}n Zygmund operators on the lattice $\Z^d$}\label{sec of appendix}

The study of harmonic analysis in the discrete setting, particularly concerning singular integrals, has a long history. As a typical model of discrete singular integral, the discrete Hilbert transform was first introduced by D.~Hilbert and proven to be bounded on $\ell^p(\Z)$ for $1<p<\infty$ by M.~Riesz \cite{Rie28} as a consequence of his proof for the continuous case on $L^p(\R)$. Subsequent developments can be traced through the works of Calder\'{o}n-Zygmund \cite{CZ52}, Stein-Wainger \cite{SW99}, Lust-Piquard\cite{Lus04}, Laeng \cite{Lae07}, Pierce \cite{Pie09} and Krause \cite{Kra22}, among others. Notably, 
in recent work \cite{BKK24} by Ba\~{n}uelos, Kim and Kwa\'{s}nicki, they established the \(\ell^p\) boundedness of discrete analogues of classical convolution-type Calder\'on-Zygmund operators for \(1 < p < \infty\). The idea of their work is that the $\ell^p$ norm of such discrete operators can be controlled by the $L^p$ norm of their continuous counterparts.

Following this idea, this appendix is devoted to extending the results of \cite[Proposition 6.1]{BKK24} to the discrete non-convolution type Calder\'{o}n-Zygmund operators on the lattice $\Z^d$.
Let $T$ be a linear operator acting on the Schwartz space of rapidly decreasing function on $\R^d$. We say that $T$ is a {\textit{Calder\'{o}n-Zygmund operator}} if it is bounded on $L^2(\R^d)$ and admits the integral representation:
\begin{align}
(Tf)(x)=p.v.\int_{\R^d}K(x,y)f(y)dy,
\end{align}
where the kernel $K\in C^1(\R^{d}\setminus\{(x,x):x\in\R^d\})$ and satisfies
\begin{align}\label{smooth condition}
|K(x,y)|\lesssim |x-y|^{-d},\quad |(\partial_xK)(x,y)|+|(\partial_yK)(x,y)|\lesssim |x-y|^{-(d+1)},\quad x\neq y.
\end{align}

It is well-known that such operators extend to bounded linear operators on $L^p(\R^d)$ for $1<p<\infty$ (see \cite[Chapter 4]{Gra14}). We consider its discrete analogue $T_{{\rm dis}}$ defined by
\begin{align}
(T_{{\rm dis}}f)(n)=\sum\limits_{m\in\Z^d\setminus{\{n\}}}K(n,m)f(m),\quad f\in\ell^p(\Z^d).
\end{align}
By virtue of the idea of \cite{BKK24}, we can establish the following conclusion.  
\begin{theorem}\label{Tdis lemma}
Let $T$ and $T_{{\rm dis}}$ be defined as above. Then we have $T_{{\rm dis}}\in\B(\ell^p(\Z^d))$ for $1< p<\infty$.
\end{theorem}
\begin{proof}
For simplicity, we focus on $d=1$ and the cases $d\geq2$ can be obtained similarly. For any $1<p<\infty$, let $f\in\ell^p(\Z)$ and $g\in\ell^q(\Z)$ with $\frac{1}{p}+\frac{1}{q}=1$. Given $x\in\R$, there exist unique $n\in\Z$ and $x_0\in U=[0,1)$ such that $x=n+x_0$. We then define $F(x)=f(n)$, $G(x)=g(n)$, which immediately yields that $\|F\|_{L^p(\R)}=\|f\|_{\ell^p(\Z)}$ and $\|G\|_{L^q(\R)}=\|g\|_{\ell^q(\Z)}$.
Furthermore,
\begin{align*}
\left<TF,G\right>&=\int_{\R^{2}}K(x,y)F(y)\overline{G(x)}dydx=\sum\limits_{n,m\in\Z}\int_{n+U}^{}\int_{m+U}^{}K(x,y)dydxf(m)\overline{g(n)}\\
&=\sum\limits_{n,m\in\Z}\Big(\underbrace{\int_{n+U}^{}\int_{m+U}^{}K(x,y)dydx-K(n,m)}_{\widetilde{K}(n,m)}+K(n,m)\Big)f(m)\overline{g(n)}\\
&=\sum\limits_{n,m\in\Z}{\widetilde{K}}(n,m)f(m)\overline{g(n)}+\left<T_{{\rm dis}}f,g\right>.
\end{align*}
Through variable substitution and the differential mean value theorem, we can rewrite
\begin{align*}
\widetilde{K}(n,m)&=\int_{U}\int_{U}(K(x+n,y+m)-K(n,m))dxdy\\
&=\int_{U}\int_{U}\Big(x\partial_xK(n+x\theta,m+y\theta)+y\partial_yK(n+x\theta,m+y\theta)\Big)dxdy,
\end{align*}
 for some $\theta\in[0,1]$. Under the smoothness condition \eqref{smooth condition}, we have the decay estimate:
\begin{align*}
|\widetilde{K}(n,m)|\lesssim |n-m|^{-2},\quad |n-m|\gg1.
\end{align*}
Applying H\"{o}lder's inequality yields
\begin{align*}
\Big|\sum\limits_{n,m\in\Z}{\widetilde{K}}(n,m)f(m)\overline{g(n)}\Big|&\leq\Big(\sum\limits_{n,m\in\Z}|{\widetilde{K}}(n,m)|\cdot|f(m)|^p\Big)^{\frac{1}{p}}\Big(\sum\limits_{n,m\in\Z}|{\widetilde{K}}(n,m)|\cdot|g(n)|^q\Big)^{\frac{1}{q}}\\
&\lesssim\|f\|_{\ell^p(\Z)}\|g\|_{\ell^q(\Z)}.
\end{align*}
Hence, using $T\in\B(L^p(\R))$ for $1<p<\infty$ and triangle inequality, the desired result is obtained.
\end{proof}
\normalem

\end{document}